\documentclass[12pt]{amsart}
\usepackage{amssymb,amsfonts,amsmath,amsopn,amstext,amscd,color,
latexsym,mathrsfs,skak,stackrel,url,verbatim,yhmath}
\usepackage{mathabx}
\usepackage{hyperref}
\usepackage[utf8]{inputenc}

\usepackage[all]{xy}

\theoremstyle{remark}

\newtheorem{para}{\bf}[subsection]
\newtheorem{claim}[para]{\bf Claim}

\newtheorem{rem}[para]{\bf Remark}
\newtheorem{convention}[para]{\bf Convention}
\theoremstyle{definition}

\newtheorem{dfn}[para]{Definition}

\theoremstyle{plain}

\newtheorem{thm}[para]{Theorem}

\newtheorem{lemma}[para]{Lemma}
\newtheorem{cor}[para]{Corollary}
\newtheorem{prop}[para]{Proposition}

\newenvironment{numequation}{\addtocounter{para}{1}
\begin{equation}}{\end{equation}}

\newcommand{\vpi}{{\varpi}}

\newcommand{\vep}{{\varepsilon}}
\newcommand{\Ga}{\Gamma}

\newcommand{\bbB}{{\mathbb B}}

\newcommand{\bbF}{{\mathbb F}}
\newcommand{\bbG}{{\mathbb G}}

\newcommand{\bbN}{{\mathbb N}}

\newcommand{\bbQ}{{\mathbb Q}}

\newcommand{\bbT}{{\mathbb T}}

\newcommand{\bbX}{{\mathbb X}}

\newcommand{\bbZ}{{\mathbb Z}}

\newcommand{\bM}{{\bf M}}

\newcommand{\frb}{{\mathfrak b}}

\newcommand{\frg}{{\mathfrak g}}

\newcommand{\frn}{{\mathfrak n}}
\newcommand{\fro}{{\mathfrak o}}

\newcommand{\frt}{{\mathfrak t}}

\newcommand{\frB}{{\mathfrak B}}

\newcommand{\frG}{{\mathfrak G}}

\newcommand{\frI}{{\mathfrak I}}

\newcommand{\frK}{{\mathfrak K}}

\newcommand{\frQ}{{\mathfrak Q}}

\newcommand{\frU}{{\mathfrak U}}

\newcommand{\frX}{{\mathfrak X}}

\newcommand{\cA}{{\mathcal A}}
\newcommand{\cB}{{\mathcal B}}
\newcommand{\cC}{{\mathcal C}}
\newcommand{\cD}{{\mathcal D}}
\newcommand{\cE}{{\mathcal E}}
\newcommand{\cF}{{\mathcal F}}
\newcommand{\cG}{{\mathcal G}}

\newcommand{\cI}{{\mathcal I}}
\newcommand{\cJ}{{\mathcal J}}

\newcommand{\cL}{{\mathcal L}}
\newcommand{\cM}{{\mathcal M}}
\newcommand{\cN}{{\mathcal N}}
\newcommand{\cO}{{\mathcal O}}

\newcommand{\cR}{{\mathcal R}}
\newcommand{\cS}{{\mathcal S}}

\newcommand{\cX}{{\mathcal X}}

\newcommand{\sC}{{\mathscr C}}
\newcommand{\sD}{{\mathscr D}}
\newcommand{\sE}{{\mathscr E}}
\newcommand{\sF}{{\mathscr F}}
\newcommand{\sG}{{\mathscr G}}

\newcommand{\sK}{{\mathscr K}}

\newcommand{\sM}{{\mathscr M}}
\newcommand{\sN}{{\mathscr N}}

\newcommand{\sT}{{\mathscr T}}

\newcommand{\tH}{R}

\newcommand{\Q}{{\mathbb Q}}
\newcommand{\Z}{{\mathbb Z}}

\newcommand{\Qp}{{\mathbb Q_p}}

\newcommand{\Fq}{{\mathbb F_q}}

\newcommand{\der}{\partial}
\newcommand{\lan}{\langle}
\newcommand{\ran}{\rangle}

\newcommand{\Ad}{{\rm Ad}}

\newcommand{\bksl}{\backslash}

\newcommand{\Coh}{{\rm Coh}}

\newcommand{\GL}{{\rm GL}}

\newcommand{\Hom}{{\rm Hom}}
\newcommand{\hra}{\hookrightarrow}

\newcommand{\Lie}{{\rm{Lie}}}
\newcommand{\Loc}{{\mathscr Loc}}
\newcommand{\lra}{\longrightarrow}
\newcommand{\midc}{{\; | \;}}

\newcommand{\pr}{{\rm pr}}

\newcommand{\ra}{\rightarrow}
\newcommand{\rig}{{\rm rig}}

\newcommand{\Spec}{{\rm Spec}}
\newcommand{\Spf}{{\rm Spf}}

\newcommand{\sub}{\subset}

\newcommand{\Sym}{{\rm Sym}}

\newcommand{\tsD}{{\widetilde{\sD}}} 
\newcommand{\hsD}{{\widehat{\sD}}} 

\newcommand{\car}{\stackrel{\simeq}{\longrightarrow}}

\newcommand{\GO}{G_0}
\newcommand{\GOv}{G_{v,0}}
\newcommand{\Gkc}{\bbG(k)^\circ}
\newcommand{\Dgk}{\cD^{\rm an}(\bbG(k)^\circ)}
\newcommand{\Dgkt}{\cD^{\rm an}(\bbG(k)^\circ)_{\theta_0}}
\newcommand{\DgkO}{D(\bbG(k)^\circ,\GO)}
\newcommand{\DgkOt}{D(\bbG(k)^\circ,\GO)_{\theta_0}}
\newcommand{\DGO}{D(\GO,L)}
\newcommand{\DGOt}{D(\GO,L)_{\theta_0}}

\newcommand{\uder}{\underline{\partial}}

\newcommand{\uxi}{\underline{\xi}}

\newcommand{\unu}{\underline{\nu}}

\begin{document}
\raggedbottom

\title{$\sD^\dagger$-affinity of formal models of flag varieties}
\author{Christine Huyghe}
\address{IRMA, Universit\'e de Strasbourg, 7 rue Ren\'e Descartes, 67084 Strasbourg cedex, France}
\email{huyghe@math.unistra.fr}
\author{Deepam Patel}
\address{Department of Mathematics, Purdue University,
150 N. University Street, West Lafayette, IN 47907, U.S.A.}
\email{deeppatel1981@gmail.com}
\author{Tobias Schmidt}
\address{IRMAR, Universit\'e de Rennes 1, Campus Beaulieu, 35042 Rennes cedex, France}
\email{Tobias.Schmidt@univ-rennes1.fr}
\author{Matthias Strauch}
\address{Indiana University, Department of Mathematics, Rawles Hall, Bloomington, IN 47405, U.S.A.}
\email{mstrauch@indiana.edu}

\thanks{D.P. would like to acknowledge support from IH\'ES and the ANR program $p$-adic Hodge Theory and beyond (Th\'eHopaD) ANR-11-BS01-005. T.S. would like to acknowledge support of the Heisenberg programme of Deutsche Forschungsgemeinschaft (SCHM 3062/1-1).
M.S. would like to acknowledge the support of the National Science Foundation (award DMS-1202303).
}

\begin{abstract} Let $\bbG$ be a connected split reductive group over a finite extension $L$ of $\bbQ_p$, denote by $\bbX$ the flag variety of $\bbG$, and let $G = \bbG(L)$. In this paper we prove that formal models $\frX$ of the rigid analytic flag variety $\bbX^\rig$ are $\sD^\dagger_{\frX,k}$-affine for certain sheaves of arithmetic differential operators $\sD^{\dagger}_{\frX,k}$. Furthermore, we show that the category of admissible locally analytic $G$-representations with trivial central character is naturally anti-equivalent to a full subcategory of the category of $G$-equivariant families $(\sM_{\frX,k})$ of modules $\sM_{\frX,k}$ over $\sD^\dagger_{\frX,k}$ on the projective system of all formal models $\frX$ of $\bbX^\rig$.
\end{abstract}

\maketitle

\tableofcontents

\section{Introduction}

Let $L/\Qp$ be a finite extension with ring of integers $\fro = \fro_L$.
In \cite{PSS4} the authors introduced certain sheaves of differential operators\footnote{These sheaves were denoted $\tsD^\dagger_{n,k}$ in \cite{PSS4} to distinguish them from the sheaves of arithmetic differential operators introduced by P. Berthelot. For ease of notation, we have decided to drop the tilde throughout this paper.} $\sD^\dagger_{n,k}$ on a family of semistable formal models $\frX_n$ of the rigid-analytic projective line over $L$ (the notion of formal model is in the sense of \cite[Def. 4 in sec. 7.4]{BoschLectures}). A key result there is that $\frX_n$ is $\sD^\dagger_{n,k}$-affine. Moreover, it was shown in loc. cit. how admissible locally $L$-analytic representations with trivial infinitesimal character of the $L$-analytic group $\GL_2(L)$, or rather their associated coadmissible modules, can be described in terms of $\GL_2(L)$-equivariant projective systems of coherent sheaves $\sM_n$ over $\sD^\dagger_{n,n}$. We generalized the construction of the sheaves $\sD^\dagger_{n,k}$ to higher-dimensional formal schemes, which are not necessarily semi-stable, in \cite{HSS}.

\vskip8pt

In this paper we generalize the previous results on $\sD^\dagger$-affinity, as well as the representation theoretic results to (not necessarily semistable) formal models of general flag varieties of split reductive groups. So let $\bbG_0$ be a connected split reductive group scheme over $\fro$, and denote by $\frX_0$ the formal completion of the flag scheme $X_0$ of $\bbG_0$. We then consider a formal admissible blow-up $\frX$ of $\frX_0$. In section \ref{new_sheaves} we briefly recall the definition of the sheaves of differential operators $\sD^{\dagger}_{\frX,k}$ as introduced in \cite{HSS}. Here $k$ is an integer which we call the {\it congruence level}. It is bounded below by a non-negative integer $k_\frX$ which depends on the blow-up morphism $\frX \ra \frX_0$. Our first main result is then

\vskip8pt

{\bf Theorem 1} (cf. \ref{thm-equivalence}){\bf .} {\it For all $k \ge k_\frX$ the formal scheme $\frX$ is $\sD^\dagger_{\frX,k}$-affine.}

\vskip8pt

This means that the global sections functor furnishes an equivalence of categories between coherent modules over $\sD^\dagger_{\frX,k}$ and finitely presented modules over the ring $H^0(\frX,\sD^\dagger_{\frX,k})$. It is shown that $H^0(\frX,\sD^\dagger_{\frX,k})$ can be identified with the central reduction $\cD^{\rm an}(\Gkc)_{\theta_0}$ of Emerton's analytic distribution algebra $\cD^{\rm an}(\Gkc)$ of the wide open rigid-analytic $k^{\rm th}$ congruence subgroup $\Gkc$ of $\bbG_0$, cf. \cite[5.2, 5.3]{EmertonA}, \cite[5.3]{HS13}. The functor $M \rightsquigarrow \Loc^\dagger_{\frX,k}(M) := \sD^\dagger_{\frX,k} \otimes_{\cD^{\rm an}(\Gkc)_{\theta_0}} M$ is quasi-inverse to the global sections functor. Compare \cite{BB81, BK80, BK81} for the classical setting of modules over the Lie algebra of $\bbG =
\bbG_0 \times_{\Spec(\fro)} \Spec(L)$ and localization on the flag variety $\bbX$ of $\bbG$.

\vskip8pt

As in \cite{PSS4} our main motivation for this result concerns locally analytic representations. The category of admissible locally analytic representations of the locally $L$-analytic group $G := \bbG(L)$ with trivial infinitesimal character $\theta_0$ is anti-equivalent to the category of coadmissible modules over $D(G,L)_{\theta_0}$, the central reduction of the locally $L$-analytic distribution algebra $D(G,L)$ of $G$ at $\theta_0$.

\vskip8pt

On the geometric side, we consider the (semisimple) Bruhat-Tits building $\cB$ of $G$ \cite{BruhatTitsI, BruhatTitsII}. This is a simplicial complex whose dimension equals the semisimple rank of $\bbG$ and which is equipped with an action of $G$. Most important for our purposes is the $G$-stable subset of $\cB$ of so-called special vertices. To any such vertex $v$ the theory of Bruhat and Tits associates a reductive group scheme $\bbG_v$ over $\fro$ whose generic fiber comes equipped with a canonical isomorphism to $\bbG$. (The group scheme $\bbG_0$ we considered before can be taken to be one of those group schemes $\bbG_{v_0}$, say.) The flag scheme $X_{v,0}$ of $\bbG_v$ therefore has the property that its generic fiber is canonically isomorphic to $\bbX$\footnote{The index ``$0$'' of $X_{v,0}$ indicates that we think of $X_{v,0}$ as the bottom layer of the tower of admissible blow-ups of this scheme.}. Passing to formal completions we thus obtain a family of smooth formal schemes $\frX_{v,0}$, indexed by the set of special vertices of $\cB$, which is equipped with a $G$-action. Furthermore, we consider for every special vertex $v$ the set $\cF_v$ of all admissible blow-ups $\frX$ of $\frX_{v,0}$, and we define $\underline{\cF}_v \sub \cF_v \times \bbN$ to be the set of pairs $(\frX,k)$ with $\frX \in \cF_v$ and $k  \ge k_\frX$. There is a natural partial ordering on $\underline{\cF} := \coprod_v \underline{\cF}_v$ which makes this a directed set (\ref{partial_orderingII}), and $\cF := \coprod_v \cF_v$ naturally carries a $G$-action, cf. \ref{G_action} for details.

\vskip8pt

A {\it coadmissible $G$-equivariant arithmetic $\sD$-module} on $\cF$ consists of a family $$\sM = (\sM_{\frX,k})_{(\frX,k) \in \underline{\cF}}$$ of coherent $\sD^\dagger_{\frX,k}$-modules $\sM_{\frX,k}$ satisfying certain compatibility properties, cf. \ref{dfn-coadmod}. In particular, these properties make it possible to form the projective limit

$$\Gamma(\sM) := \varprojlim_{(\frX,k) \in \underline{\cF}} H^0(\frX,\sM_{\frX,k})$$

\vskip8pt

which, as we show, carries the structure of a coadmissible $D(G,L)_{\theta_0}$-module. On the other hand, given a coadmissible $D(G,L)_{\theta_0}$-module $M$ we let $V = M'$ be its continuous dual, which is an admissible locally analytic representation of $G$. We then let $M_{v,k}$ be the continuous dual of the subspace $V_{\bbG_v(k)^\circ-{\rm an}} \sub V$ of $\bbG_v(k)^\circ$-analytic vectors in $V$. For any $(\frX,k) \in \underline{\cF}_v$ we have the coherent $\sD^\dagger_{\frX,k}$-module

$$\Loc^\dagger_{\frX,k} (M_{v,k})=\sD^\dagger_{\frX,k} \otimes_{\cD^{\rm an}(\bbG_v(k)^\circ)_{\theta_0}} M_{v,k} \;.$$

\vskip8pt

We denote the family of all those modules by $\Loc^G(M)$. Our main result is then

\vskip8pt

{\bf Theorem 2} (cf. \ref{thm_G_equiv}){\bf .} {\it The functors $\Loc^G$ and $\Gamma$ are quasi-inverse equivalences between the category of coadmissible $D(G,L)_{\theta_0}$-modules and the category $\sC^G_\cF$ of coadmissible $G$-equivariant arithmetic $\sD$-modules on $\cF$.}

\vskip8pt

The projective limit $\frX_\infty := \varprojlim_{\frX \in \cF} \frX$ is the Zariski-Riemann space attached to $\bbX^\rig$. The latter space is in turn isomorphic (as a ringed space, after inverting $p$ on the structure sheaf) to the adic space attached to $\bbX^\rig$, cf. \cite[Thm. 4 in sec. 2, Thm. 4 in sec. 3]{SchnVPut}. One can also form the projective limit $\sD_\infty$ of the sheaves $\sD^\dagger_{\frX,k}$ which is then a $G$-equivariant sheaf of $p$-adically complete rings of differential operators on $\frX_\infty$, cf. \ref{para-sheaf}. Similarly, for any object $\sM = (\sM_{\frX,k})$ in $\sC^G_\cF$ one can form the projective limit $\sM_\infty$ of the sheaves $\sM_{\frX,k}$ which is then a $G$-equivariant $\sD_\infty$-module. The assignment $\sM \rightsquigarrow \sM_\infty$ is a faithful functor from $\sC^G_\cF$ to the category of $G$-equivariant $\sD_\infty$-modules, cf. \ref{faithful2}.

\vskip8pt

In a final section we illustrate this localization theory by computing the $\sD^\dagger_{\frX,k}$-modules associated to certain classes of locally analytic representations.

\vskip8pt

In this paper we only treat the case of the central character $\theta_0$, but there is an extension of this theorem available for characters more general than $\theta_0$ by using twisted versions of the sheaves $\sD^\dagger_{\frX,k}$. Moreover, the construction of the sheaf $\sD^\dagger_{\infty}$ carries over to general smooth rigid-analytic (or adic) spaces over $L$. These questions will be addressed in future work.

\vskip8pt

We would also like to mention that K. Ardakov and S. Wadsley are  developing a theory of $D$-modules on general rigid-analytic spaces, cf. \cite{ArdakovICM, AWDcapI, AWDcapII}. In their work they consider deformations of the sheaves of crystalline differential operators (as in \cite{AW}), whereas we take as a starting point deformations of Berthelot's rings of arithmetic differential operators. That the rings of differential operators considered by us are close in spirit to the theory of rigid cohomology will, as we hope, open a way to use techniques and results from rigid cohomology to investigate locally analytic representations. A first example for such an interaction can be found in \cite[sec. 7]{PSS4}.

\vskip8pt

{\it Notation.} $L$ denotes a finite extension of $\Qp$, with ring of integers $\fro$ and uniformizer $\vpi$. Let $q$ be the cardinality of the residue field $\fro/(\vpi)$ which we also denote by $\Fq$. $\bbG_0$ denotes a split connected reductive group scheme over $\fro$ and $\bbB_0 \sub \bbG_0$ a Borel subgroup scheme. We let $\bbG = \bbG_0 \times_{\Spec(\fro)} \Spec(L)$ be the generic fiber of $\bbG_0$. The Lie algebra of $\bbG_0$ is denoted by $\frg_{\fro}$. If $X$ is a smooth scheme over $\Spec(\fro)$, we denote by $\sT_{X}$ its relative tangent sheaf, i.e., $\sT_X = \sT_{X/\Spec(\fro)}$. If $X$ (resp. $\frX$) is a  scheme (resp. formal scheme) over $\Spec(\fro)$ (resp. $\Spf(\fro)$), a coherent sheaf of ideals $\cI\subset \cO_{X}$ (resp. $\frI \sub \cO_\frX$) is said to be {\it open} (w.r.t. the $p$-adic topology) if $\vpi$ is locally nilpotent on $ \mbox{{\bf Spec}}(\cO_X/\cI)$ (resp. $\mbox{{\bf Spf}}(\cO_\frX/\frI)$). A scheme (or a formal scheme) over $\Spec(\fro)$ (resp. $\Spf(\fro)$) which arises from blowing up an
open ideal sheaf on $X$ (resp. $\frX$) will be called {\it an admissible blow-up} of $X$ (resp. {\it admissible formal blow-up} of $\frX$). If $X$ denotes a scheme over $\fro$, we always denote by $\frX$ the completion of $X$ along its special fiber $X \times_{\Spec(\fro)} \Spec(\Fq)$. The set of non-negative integers will be denoted by $\bbN$ (in particular, our convention is such that $\bbN$ contains zero). If $V$ is a topological vector space over $L$, then $V' = \Hom^{\rm cont}_L(V,L)$ denotes space of continuous linear forms on $V$, and when we write $V'_b$, then the subscript ''$b$'' indicates that we equip this space with the strong topology of bounded convergence. If not said otherwise, all modules are tacitly assumed to be left modules.

\vskip8pt

{\it Acknowledgments.} C.H. and M.S. benefited from an invitation to MSRI during the Fall 2014 and thank this institution for excellent working conditions. M.S. gratefully acknowledges the support of the Institut de Recherche Math\'ematique Avanc\'ee (IRMA) of the University of Strasbourg during a stay in research in June 2016. We would also like to thank the anonymous referees for their careful reading and very helpful reports from which this paper has greatly benefited.

\section{The sheaves \texorpdfstring{$\sD^{(k,m)}_X$}{} and \texorpdfstring{$\hsD^{(k,m)}_\frX$}{}}
\label{new_sheaves}

While sections 3-6 of this paper are only about flag varieties and their formal models, we work in this section in somewhat greater generality, as this is more natural for the material considered here. For more details about the constructions discussed below, as well as the proofs of the main result of this section, we refer the reader to~\cite{HSS}.

\subsection{Differential operators with levels and congruence levels}

Here we briefly recall the local description of Berthelot's sheaf $\sD^{(m)}$ of differential operators of level $m$. Moreover, we introduce a kind of deformation of this sheaf, to be denoted by $\sD^{(k,m)}$, where $k \in \bbN$ is what we call a {\it congruence level}. For $k=0$ we have $\sD^{(0,m)} = \sD^{(m)}$. As will become apparent in section \ref{global_sec}, this terminology is motivated by the relation of these sheaves, in the case of flag varieties, to principal congruence subgroups. In the special case of the projective line, the sheaves with congruence levels have been introduced in \cite{PSS4}, and similar constructions also appeared earlier in \cite{AW}.

\vskip8pt

Let $X_0$ be a smooth scheme over $\fro$ and $\frX_0$
the associated formal scheme, i.e., the completion of $X_0$ along the special fiber $X_0 \times_{\Spec(\fro)} \Spec(\Fq)$. The usual sheaf of relative differential operators \cite[16.8]{EGA_IV_4} on $X_0$ over $\fro$ will be denoted by $\sD_{X_0/\Spec(\fro)}$ (without superscripts as `decorations'). Let $U_0$ be an affine open subset of $X_0$, endowed with local coordinates $x_1,\ldots,x_M$, and let $\der_1,\ldots,\der_M$ be the corresponding derivations. Denote by $m$ a fixed non-negative integer. For a non-negative integer $\nu_l$, we let $q^{(m)}_{\nu_l}$ be the quotient of the euclidean division of $\nu_l$ by $p^m$, i.e., $q^{(m)}_{\nu_l} = \lfloor \frac{\nu_l}{p^m} \rfloor$. Then we set

\begin{numequation}\label{notbracket}
\der_{l}^{\lan \nu_l \ran_{(m)}} = q^{(m)}_{\nu_l}!\der_{l}^{[\nu_l]},
\end{numequation}

where, as usual, $\der_{l}^{[\nu_l]}\in \Ga(U_0,\sD_{U_0/\Spec(\fro)})$ is such that $l! \der_{l}^{[\nu_l]}=\der_{l}^{\nu_l}$. For $\unu = (\nu_1,\ldots,\nu_M) \in \bbN^M$, we put $\uder^{\lan \unu \ran_{(m)}} = \prod_{l=1}^M \der_{l}^{\lan \nu_l \ran_{(m)}}$, $\uder^{[\unu ]}=\prod_{l=1}^M \der_{l}^{[\nu_l]}$, and $|\unu| = \nu_1 + \ldots + \nu_M$.

\vskip8pt

Denote by $\sD^{(m)}_{X_0} := \sD^{(m)}_{X_0/\Spec(\fro)} \sub \sD_{X_0/\Spec(\fro)}$ the ring of level $m$ differential operators of Berthelot, cf. \cite[sec. 2]{BerthelotDI} (from now on we agree on omitting the base scheme $\Spec(\fro)$ in the notation as in \cite[2.2.3]{BerthelotDI}). Then we have the following description in local coordinates:

$$\Ga(U_0,\sD^{(m)}_{X_0})= \left\{\sum_{\unu}^{< \infty} a_{\unu}\uder^{\lan \unu \ran_{(m)}}\,|\, a_{\unu}\in
\Ga(U_0,\cO_{X_0})\right\} \;,$$

\vskip8pt

as follows from \cite[2.2.5]{BerthelotDI}. Now let $k \in \bbN$ be another non-negative integer (the congruence level mentioned above). We then define a subring $\Ga(U_0,\sD^{(k,m)}_{X_0}) \sub \Ga(U_0,\sD^{(m)}_{X_0})$ by setting

\begin{numequation}\label{explicit_description_1}\Ga(U_0,\sD^{(k,m)}_{X_0})= \left\{\sum_{\unu}^{< \infty} \vpi^{k|\unu|}a_{\unu}\uder^{\lan \unu \ran}\,|\, a_{\unu}\in \Ga(U_0,\cO_{X_0})\right\} \;.
\end{numequation}

It is straightforward to see that this is indeed a subring of $\Ga(U_0,\sD^{(m)}_{X_0})$. And, as the notation already suggests, it is not hard to show that these rings glue together to give a subsheaf $\sD^{(k,m)}_{X_0}$ of $\sD^{(m)}_{X_0}$.

\vskip8pt

\begin{rem}\label{integral_structure} Let $X_{0,\eta} = X_0 \times_{\Spec(\fro)} \Spec(L)$ be the generic fiber of $X_0$ which is an open subset of $X_0$. We note that for any pair $(k,m) \in \bbN^2$ the inclusion $\sD^{(k,m)}_{X_0} \sub \sD_{X_0}$ induces a canonical isomorphism $\sD^{(k,m)}_{X_0}\Big|_{X_{0,\eta}} = \sD_{X_0}\Big|_{X_{0,\eta}} = \sD_{X_{0,\eta}}$, because $\vpi$ is invertible on $X_{0,\eta}$. Any of the sheaves $\sD^{(k,m)}_{X_0}$ therefore extends the sheaf $\sD_{X_{0,\eta}}$ to the whole scheme $X_0$.
\end{rem}

\subsection{Differential operators with levels and congruence levels on blow-ups}

\begin{para}\label{lifting_diff_op} {\it Lifting the sheaves to blow-ups.} Denote by $\pr: X \ra X_0$ an admissible blow-up. That is to say, $X$ is obtained by blowing up a sheaf of ideals $\cI \sub \cO_{X_0}$ containing some power of $\vpi$, say $\vpi^k$. In particular, the blow-up morphism $\rm pr$ induces a canonical isomorphism $X_\eta \simeq X_{0,\eta}$ between the generic fibers, cf. \ref{integral_structure} for the notation.

\vskip8pt

The sheaf $\pr^{-1}\Big(\sD^{(k,m)}_{X_0}\Big)$ on $X$ is again a sheaf of rings, and it follows from \ref{integral_structure} that there is a canonical isomorphism $\pr^{-1}\Big(\sD^{(k,m)}_{X_0}\Big)\Big|_{X_\eta} = \sD_{X_\eta}$. In particular, $\cO_{X_\eta}$ is naturally a module over $\pr^{-1}\Big(\sD^{(k,m)}_{X_0}\Big)\Big|_{X_\eta}$. Now the question arises for which congruence levels $k \in \bbN$ this module structure extends to a module structure on $\cO_X$ over $\pr^{-1}\Big(\sD^{(k,m)}_{X_0}\Big)$. Since functions on $X$ are determined by their restriction to $X_\eta$, any such extension of module structure is unique. As in \cite[2.1.10]{HSS} one shows that the condition $\vpi^k \in \cI$ implies that $\cO_X$ carries a natural structure of a module over $\pr^{-1}\Big(\sD^{(k,m)}_{X_0}\Big)$. Therefore, the sheaf

\begin{numequation}\label{rings_diff_op}
\sD^{(k,m)}_X := \pr^* \sD^{(k,m)}_{X_0} = \cO_X \otimes_{\pr^{-1}(\cO_{X_0})} \pr^{-1}\Big(\sD^{(k,m)}_{X_0}\Big)
\end{numequation}

can be equipped with a multiplication which extends the sheaf of rings structure of $\pr^{-1}\Big(\sD^{(k,m)}_{X_0}\Big)$. Explicitly, if $\partial_1, \partial_2$ are both {\it derivations} and local sections of $\pr^{-1}\Big(\sD^{(k,m)}_{X_0}\Big)$, and if $f_1, f_2$ are local sections of $\cO_X$, then $(f_1 \otimes \partial_1) \cdot (f_2 \otimes \partial_2) = f_1\partial_1(f_2) \otimes \partial_2 + f_1f_2 \otimes \partial_1 \partial_2$. We set

\begin{numequation}\label{defk_X} k_X = \min_\cI \min\{k \in \bbN \midc \vpi^k \in \cI \} \;,
\end{numequation}

where the first minimum is taken over all open ideal sheaves $\cI$ such that the blow-up of $\cI$ is isomorphic to $X$ (over $X_0$). Suppose $U_0 \sub X_0$ is an affine open subset which is endowed with local coordinates $x_1,\ldots,x_M$. Consider an affine open subset $U \sub \pr^{-1}(U_0) \sub X$. Then we have the following description of the sections of $\sD^{(k,m)}_{X}$ over $U$:

\begin{numequation}\label{explicit_description_2}\Ga(U,\sD^{(k,m)}_X)= \left\{\sum_{\unu}^{< \infty} \vpi^{k|\unu|}a_{\unu}\uder^{\lan \unu \ran_{(m)}}\,|\, a_{\unu}\in \Ga(U,\cO_{X})\right\} \;.
\end{numequation}
\end{para}

\begin{para}\label{filtrations} {\it Filtrations on $\sD^{(k,m)}_X$.} Using this description, we observe that the sheaf $\sD^{(k,m)}_{X}$ is filtered by the order of differential operators. More precisely, if $d \in \bbN$ is given, we define the subsheaf $\sD^{(k,m)}_{X,d}$ as follows. Let $V \sub X$ be any open subset. Then $\Ga(V,\sD^{(k,m)}_{X,d})$ consists of those elements $P \in \Ga(V,\sD^{(k,m)}_X)$ such that for any open affine $U_0 \sub X_0$ as above, and for any open affine $U \sub V \cap \pr^{-1}(U_0)$, the restriction $P|_U$ is of the form $\sum_{|\unu| \le d}\vpi^{k|\unu|}a_{\unu}\uder^{\lan \unu \ran_{(m)}}$ with $a_{\unu}\in \Ga(U,\cO_{X})$ and where, as usual, $|\unu| = \nu_1 + \ldots + \nu_M$.
There are canonical isomorphisms $\sD^{(k,m)}_{X,d} = \pr^* \sD^{(k,m)}_{X_0,d}$. We put

\begin{numequation}\label{tangent_sheaf_with_congr_level}
\sT_{X,k} := \vpi^k \pr^*(\sT_{X_0}) \sub \pr^*(\sT_{X_0}) \;,
\end{numequation}

and we denote by

$$\Sym^{(m)}(\sT_{X,k}) = \bigoplus_d\Sym^{(m)}_d(\sT_{X,k})$$

\vskip8pt

the graded level $m$ symmetric algebra generated by the sheaf $\sT_{X,k}$, cf. \cite[sec. 1.2]{Huyghe97}. If $U_0$ is affine endowed
with local coordinates $x_1,\ldots,x_M$  as before,  and $\xi_1, \ldots,\xi_M$ a basis of $\sT_{X_0}$ restricted to $U_0$, then using notations of ~\ref{notbracket} one has for an open affine $U\subset \pr^{-1}(U_0)$

$$\Ga(U,\Sym^{(m)}_d(\sT_{X,k}))=\bigoplus_{|\unu|=d}\cO(U)\vpi^{kd} \uxi^{\lan \unu \ran_{(m)}} \;.$$

\vskip8pt
\end{para}

In \cite[2.2.2]{HSS} we show the following

\begin{prop}\label{finite_tDm} Suppose $k \ge k_X$. Then the associated graded algebra of $\sD^{(k,m)}_X$ for the filtration by the order of differential operators is isomorphic to $\Sym^{(m)}(\sT_{X,k})$.
\end{prop}

\begin{para}\label{completions} {\it $p$-adic completions.} We denote the completion of $X_0$ and $X$ along their special fibers by $\frX_0$ and $\frX$, respectively, and we let $\hsD^{(k,m)}_\frX$ be the $p$-adic completion of $\sD^{(k,m)}_X$ which we consider as a sheaf on the formal scheme $\frX$. For fixed $k \ge k_X$, cf. \ref{defk_X}, we also define

$$\sD^\dagger_{\frX,k}=\varinjlim_m \hsD^{(k,m)}_{\frX,\Q} \;.$$

\vskip8pt

{\bf Remark.} We emphasize that the sheaves $\sD^{(k,m)}_X$, $\hsD^{(k,m)}_{\frX}$, $\sD^\dagger_{\frX,k}$ do not only depend on $X$, resp. $\frX$, but in an essential way on the blow-up morphism to $X_0$, resp. $\frX_0$.

\end{para}

In this paper we will only be working with formal schemes $\frX$ which are completions along their special fibers of admissible blow-ups $X \ra X_0$ of a smooth scheme $X_0$ over $\Spec(\fro)$. In this regard we have the following

\begin{prop}\label{algebraization} Let $\frX \ra \frX_0$ be an admissible formal blow-up, obtained by blowing up an open ideal sheaf $\frI \sub \cO_{\frX_0}$. Then there is an open ideal sheaf $\cI \sub \cO_{X_0}$ such that $\frI$ is the restriction of the $p$-adic completion of $\cI$ to $\frX_0$, and $\frX$ is therefore the completion of the blow-up $X$ of $\cI$ along its special fiber.
\end{prop}

\begin{proof} We remark that $X_0$ being smooth over $\fro$ implies that it is locally noetherian, which is all we need for this statement to hold. Consider the quotient sheaf $\frQ = \cO_{\frX_0}/\frI$ and the canonical surjection

$$\sigma: \cO_{\frX_0} \lra \frQ$$

\vskip8pt

of sheaves on $\frX_0$, and let $i: \frX_0 \ra X_0$ be the closed embedding of the special fiber. This is a morphism of ringed spaces. We consider the corresponding map of sheaves $\cO_{X_0} \ra i_* \cO_{\frX_0}$ which we compose with $i_* \sigma$ to obtain the morphism of sheaves on $X_0$

$$\tau: \cO_{X_0} \ra i_* \frQ \;.$$

\vskip8pt

Our first goal is to show that $\tau$ is surjective. Let $U \sub X_0$ be an affine open subscheme, and $\frU \sub \frX_0$ be the completion along its special fiber. We have $\vpi^n \frQ_U = 0$ for some $n \in \bbN$, and hence $\vpi^n \frQ_\frU = 0$. The restriction of the surjection $\sigma$ to $\frU$ thus factors as

$$\sigma|_\frU: \cO_{\frX_0}|_\frU  = \cO_\frU \lra \cO_\frU \otimes_\fro \fro/(\vpi^n) \lra \frQ|_\frU \;.$$

\vskip8pt

Since $\cO_\frU$ is the restriction to $\frU$ of the $p$-adic completion of $\cO_U$, we see that the canonical map $\cO_U \ra i_* \cO_\frU$ induces an  isomorphism  $\cO_U \otimes_\fro \fro/(\vpi^n) \car i_*\Big(\cO_\frU \otimes_\fro \fro/(\vpi^n)\Big)$ and therefore a surjection

$$\cO_U \twoheadrightarrow \cO_U \otimes_\fro \fro/(\vpi^n) = i_*\Big(\cO_\frU \otimes_\fro \fro/(\vpi^n)\Big) \twoheadrightarrow i_*(\frQ|_\frU) = (i_*\frQ)|_U \;.$$

\vskip8pt

Of course, this map is the same as $\tau|_U$, and $\tau|_U$ is thus surjective. Therefore, $\tau$ is surjective. Put $\cI = \ker(\tau)$ and consider the tautological exact sequence of coherent sheaves on $X_0$

$$0 \lra \cI \lra \cO_{X_0} \lra i_* \frQ \lra 0 \;.$$

\vskip8pt

By \cite[10.8.8]{EGA_I}, the completion functor is exact on coherent sheaves, and the previous exact sequence thus yields an exact sequence of sheaves on $\frX_0$

$$0 \lra \widehat{\cI}|_{\frX_0} \lra \cO_{\frX_0} \stackrel{\sigma}{\lra} \frQ \lra 0 \;.$$

\vskip8pt

This shows that $\frI$ is the restriction to $\frX_0$ of the $p$-adic completion of $\cI$. The very definition of admissible formal blow-up, cf. \cite[Def.3 in sec. 8.2]{BoschLectures} shows that then $\frX$ is equal to the formal completion along its special fiber of the blow-up of $\cI$. \end{proof}

\vskip8pt

Given an admissible formal blow-up $\frX \ra \frX_0$ we put

\begin{numequation}\label{defkX} k_\frX = \min_\frI \min \{k \in \bbN \midc \vpi^N\in \frI \} \;,
\end{numequation}

where the first minimum is taken over all open ideal sheaves $\frI \sub \cO_{\frX_0}$ such that the blow-up of $\frI$ is isomorphic to $\frX$ (over $\frX_0$).

\begin{convention}\label{convention} In the remainder of this paper, whenever we consider the sheaves $\sD^{(k,m)}_X$ on the admissible blow-up $X$ of $X_0$ we tacitly assume that $k \ge k_X$. Similarly, whenever we consider the sheaves $\hsD^{(k,m)}_\frX$, $\hsD^{(k,m)}_{\frX,\Q}$, or $\sD^\dagger_{\frX,k}$ on the admissible formal blow-up $\frX$ of $\frX_0$ we tacitly assume that $k \ge k_\frX$.
\end{convention}

\vskip8pt

We will also need the following result from \cite[2.2.2, 2.3.3]{HSS}:

\begin{thm}\label{prop-exactdirectimage}
Let $\pi: \frX' \ra \frX$ be a morphism over $\frX_0$ between admissible formal blow-ups of $\frX_0$, and let $k\geq {\rm max}\{k_{\frX},k_{\frX'}\}$.

\vskip8pt

(i) $\hsD^{(k,m)}_{\frX,\Q}$ and $\sD^\dagger_{\frX,k}$ are coherent sheaves of rings. Moreover, $\hsD^{(k,m)}_{\frX,\Q}$ has noetherian rings of sections over all open affine subsets.

\vskip8pt

(ii) There is a canonical isomorphism $\pi_* \sD^\dagger_{\frX',k} = \sD^\dagger_{\frX,k}$. If $\sM'$ is a coherent $\sD^\dagger_{\frX',k}$-module, then $R^j\pi_*\sM'=0$ for $j>0$. The functor $\pi_*$ induces an exact functor from the category of coherent modules over $\sD^\dagger_{\frX',k}$ to the category of coherent modules over $\sD^\dagger_{\frX,k}$.
\end{thm}

\section{Formal models of flag varieties}\label{models}

\subsection{Models, formal models, and group actions}\label{models_gp_actions}

\begin{para}\label{models_and_formal_models} {\it Models and formal models.} For the remainder of this paper $\bbG_0$ denotes a split connected reductive group scheme over $\fro$ and $\bbB_0 \sub \bbG_0$ a Borel subgroup scheme. The Lie algebra of $\bbG_0$ is denoted by $\frg_{\fro}$. By

$$X_0 = \bbB_0 \bksl \bbG_0$$

\vskip8pt

we denote the flag scheme of $\bbG_0$, which is smooth and projective over $\fro$ \cite[Exp. XXVI, Cor. 3.5]{SGA3}, and we let $\frX_0$ be the completion of $X_0$ along its special fiber $X_0 \times_{\Spec(\fro)} \Spec(\Fq)$. By $\bbG = \bbG_0 \times_{\Spec(\fro)} \Spec(L)$ (resp. $\bbB$) we denote the generic fiber of $\bbG_0$ (resp. $\bbB_0$), and we let $\frg$ be the Lie algebra of $\bbG$. The flag variety $\bbB \bksl \bbG$ of $\bbG$ will be denoted by $\bbX$, and we let $\bbX^\rig$ be the rigid-analytic space associated by the GAGA functor to $\bbX$, cf. \cite[5.4]{BoschLectures}. Any admissible formal $\fro$-scheme $\frX$ (in the sense of \cite[Def. 1 in sec. 7.4]{BoschLectures}) whose associated rigid-analytic space is isomorphic to $\bbX^\rig$ will be called a {\it formal model} of $\bbX^\rig$, or simply a formal model of the flag variety associated to $\bbG$, cf. \cite[Def. 4 in sec. 7.4]{BoschLectures}. For any two formal models $\frX_1, \frX_2$ of $\bbX^\rig$ there is a third formal model $\frX'$ and admissible formal blow-up morphisms $\frX' \ra \frX_1$ and $\frX' \ra \frX_2$, cf. \cite[Remark 10 in sec. 8.2]{BoschLectures}. In particular, for every formal model $\frX$ there is a formal model $\frX'$ and admissible formal blow-up morphisms $\frX' \ra \frX$ and $\frX' \ra \frX_0$.
\end{para}

\begin{para}\label{group_actions} {\it Group actions.}  We equip $X_0$ with the translation action on the {\it right} by $\bbG_0$, i.e.,

$$X_0 \times_{\Spec(\fro)} \bbG_0 \ra X_0 \;, \;\; (\bbB_0g,h) \mapsto \bbB_0 g h \;.$$

\vskip8pt

The right action of $\bbG_0$ on $X_0$ induces a right action\footnote{We remark that the flag schemes, or flag varieties, considered in \cite{PSS4} and \cite{HS17} are also equipped with right group actions. This will be of some importance later when we consider certain ring homomorphisms. Namely, those ring homomorphisms are indeed homomorphisms and not anti-homomorphisms, cf. \ref{global_sections_tcD}.} of $\bbG$ on $\bbX$. We fix once and for all a very ample line bundle $\cO_{X_0}(1)$ on $X_0$ over $\Spec(\fro)$.
\end{para}

\subsection{Preliminaries on blow-ups of the flag scheme \texorpdfstring{$X_0$}{}} \label{ample_sh} Let $\pr: X \ra X_0$ be an admissible blow-up, and let $\cI \sub X_0$ be the ideal sheaf that is blown up.
The inverse image ideal sheaf $\pr^{-1}(\cI) \cdot \cO_{X}$ is an invertible sheaf on $X$ which we denote by $\cO_{X/X_0}(1)$, cf. \cite[ch. II, 7.13]{HartshorneA}. By \cite[remark after 8.1.3]{EGA_II} the blow-up morphism is projective, and $X$ is thus itself projective over $\fro$.

\vskip8pt

\begin{lemma}\label{v_ample_sh_lemma} There is $a_0\in \Z_{>0}$ such that the line bundle

$$\cL_{X} = \cO_{X/X_0}(1) \otimes \pr^*\Big(\cO_{X_0}(a_0)\Big)$$

\vskip8pt

on $X$ is very ample over $\Spec(\fro)$, and it is very ample over $X_0$.
\end{lemma}

\begin{proof} By \cite[ch. II, ex. 7.14 (b)]{HartshorneA}, the sheaf

$$\cL = \cO_{X/X_0}(1) \otimes \pr^*\Big(\cO_{X_0}(a_0)\Big)$$

\vskip8pt

is very ample on $X$ over $\Spec(\fro)$ for suitable $a_0 > 0$. We fix such an $a_0$. By \cite[4.4.10 (v)]{EGA_II} it is then also very ample over $X_0$. \end{proof}

\begin{para} {\it Twisting by $\cL_X$.} We fix $a_0 \in \Z_{>0}$ such that the line bundle $\cL_X$ from \ref{v_ample_sh_lemma} is very ample over $\Spec(\fro)$. In the following we will always use this line bundle to `twist' $\cO_{X}$-modules. If $\cF$ is a $\cO_{X}$-module and $r \in \Z$ we thus put

$$\cF(r) = \cF \otimes_{\cO_{X}} \cL_X^{\otimes r} \;.$$

\vskip8pt

Some caveat is in order when we deal with sheaves which are equipped with both a left and a right $\cO_{X}$-module structure (which may not coincide). For instance, if $\cF_d = \sD^{(k,m)}_{X,d}$, cf. \ref{filtrations}, then we let

$$\cF_d(r) = \sD^{(k,m)}_{X,d}(r) = \sD^{(k,m)}_{X,d} \otimes_{\cO_{X}} \cL_X^{\otimes r} \;,$$

\vskip8pt

where we consider $\cF_d = \sD^{(k,m)}_{X,d}$ as a {\it right} $\cO_{X}$-module. Similarly we put

$$\sD^{(k,m)}_X(r) = \sD^{(k,m)}_{X} \otimes_{\cO_{X}} \cL_X^{\otimes r} \;,$$

\vskip8pt

where we consider $\sD^{(k,m)}_{X}$ as a {\it right} $\cO_{X}$-module. Then we have $\sD^{(k,m)}_{X}(r) = \varinjlim_d
\cF_d(r)$. When we consider the associated graded sheaf of $\sD^{(k,m)}_{X}(r)$, it is with respect to the filtration by
the $\cF_d(r)$. The sheaf $\sD^{(k,m)}_{X}(r)$ is a coherent left $\sD^{(k,m)}_{X}$-module since it is locally
isomorphic with $\sD^{(k,m)}_{X}$ as $\sD^{(k,m)}_{X}$-module.
\end{para}

\begin{lemma}\label{tcT_lemma} Let $\pr: X \ra X_0$ and $\pr': X' \ra X_0$ be admissible blow-ups of $X_0$, and let $\pi: X' \ra X$ be a morphism over $X_0$, i.e., $\pr \circ \pi = \pr'$. Furthermore, let $k,k'$ be two non-negative integers (not necessarily greater or equal to $k_X$ or $k_{X'}$).

\vskip8pt

(i) In the case $\pi_* \cO_{X'} = \cO_X$, one has

$$\vpi^{k'-k}\sT_{X,k} = {\rm \pi}_*(\sT_{X',k'})$$

as subsheaves of $\sT_{X} \otimes_\fro L$ (cf. \ref{tangent_sheaf_with_congr_level} for the definition of $\sT_{X,k}$).

\vskip8pt

(ii) The group action of $\bbG_0$ on $X_0$ induces a morphism $\frg_\fro \ra H^0(X_0, \sT_{X_0})$ of Lie algebras over $\fro$. This map induces an $\cO_{X_0}$-linear map $\alpha: \cO_{X_0} \otimes_\fro \frg_\fro \ra \sT_{X_0}$. The map $\vpi^k \pr^* \alpha: \cO_X \otimes_\fro \vpi^k\frg_\fro \ra \sT_{X,k}$ is an $\cO_X$-linear map which in turn induces a morphism $\vpi^k \frg_\fro \ra H^0(X, \sT_{X,k})$ of Lie algebras over $\fro$.

\vskip8pt

(iii) If $X$ is normal, then $\pi_* \cO_{X'} = \cO_X$. This holds, in particular, if $X = X_0$ and $\pi$ is the blow-up morphism $X' \ra X_0$.
\end{lemma}

\begin{proof} The assertion (i) follows from the projection formula and the fact that

$$\vpi^{k'-k}{\rm \pi}^*\sT_{X,k}=\sT_{X',k'}$$

\vskip8pt

by definition of the sheaves, if $k' \geq k$. Otherwise, we have ${\rm \pi}^*\sT_{X,k}=\vpi^{k-k'}\sT_{X',k'}$.

\vskip8pt

(ii) The first assertion is \cite[II, \S 4, 4.4]{DemazureGabriel} (note that in loc. cit. the map is an anti-homomorphism because in loc. cit. the group acts from the left on the scheme in question). The remaining assertions are immediate consequences of the first assertion.

\vskip8pt

(iii) Let $\cI \sub \cO_{X_0}$ be the ideal that is blown up to obtain $X$. The sheaf $\cS = \bigoplus_{d \ge 0} \cI^d$ is naturally a subsheaf of the sheaf of polynomial algebras $\cO_{X_0}[t]$, and is thus a sheaf of integral domains, since $X_0$ is integral. Therefore, $X$ is integral too. The same holds for $X'$. Since $\pr'$ is projective (cf. the beginning of this subsection), and since $\pr \circ \pi = \pr'$, we conclude that $\pi$ is projective too, by \cite[5.5.5]{EGA_II}. Now let $\cJ \sub \cO_X$ be the ideal sheaf which is blown up to obtain $X'$. As $\cJ$ contains a power of $\vpi$, the vanishing locus of $\cJ$ is contained in the special fiber of $X$, and $\pi$ is hence an isomorphism on the generic fibers, and hence birational. $\pi$ is thus a projective birational morphism between noetherian integral schemes. The assertion follows now from Zariski's Main Theorem, cf. \cite[11.4 in ch. III]{HartshorneA} and its proof. \end{proof}

\vskip8pt

We remind the reader of our convention \ref{convention} regarding the congruence level $k$.

\begin{prop}\label{graded_tcD} Let $\pi: X' \ra X$ be a morphism over $X_0$ of admissible blow-ups of $X_0$ (as in \ref{tcT_lemma}). If $k \ge \max\{k_X,k_{X'}\}$ and if $\pi_*\cO_{X'}=\cO_X$, then $ \pi_*\left(\sD^{(k,m)}_{X'}\right) = \sD^{(k,m)}_{X}$.
\end{prop}

\begin{proof} The sheaves $\sD^{(k,m)}_{X_0,d}$ of differential operators of order $\le d$ are locally free of finite rank, and so are the sheaves $\sD^{(k,m)}_{X,d}$, by construction. We can thus apply the projection formula and get

$$\pi_*\Big(\sD^{(k,m)}_{X',d}\Big) = \sD^{(k,m)}_{X,d} \;.$$

\vskip8pt

The claim follows because the direct image commutes with inductive limits on a noetherian space. \end{proof}

\subsection{Global sections of \texorpdfstring{$\sD^{(k,m)}_{X}$}{}, \texorpdfstring{$\hsD^{(k,m)}_{\frX}$}{}, and \texorpdfstring{$\sD^\dagger_{\frX,k}$}{}}\label{global_sec}

\begin{para}\label{groups} {\it Congruence group schemes.} We let $\bbG(k)$ denote the $k$-th scheme-theoretic congruence subgroup of the group scheme $\bbG_0$ \cite[sec. 1]{WW80}, \cite[2.8]{YuSmoothModels}. So $\bbG(0)=\bbG_0$ and $\bbG(k+1)$ equals the dilatation, in the sense of \cite[3.2]{BLR}, of the trivial subgroup of $\bbG(k) \times_{\Spec(\fro)} \Spec(\bbF_q)$ on $\bbG(k)$. In particular, if $\bbG(k) = \Spec\; \fro[t_1,\ldots ,t_N]$ with a set of parameters $t_i$ for the unit section of $\bbG(k)$, then $\bbG(k+1)=\Spec \; \fro[\frac{t_1}{\varpi},\ldots ,\frac{t_N}{\varpi}]$. The $\fro$-group scheme $\bbG(k)$ is again smooth, has Lie algebra equal to $\varpi^k\frg_{\fro}$ and its generic fibre coincides with the generic fibre of $\bbG_0$.
\end{para}

\begin{para}\label{div_power_env_algs} {\it Divided power enveloping algebras.} We denote by $D^{(m)}(\bbG(k))$ the distribution algebra of the smooth $\fro$-group scheme $\bbG(k)$ of level $m$ \cite[4.1.3]{HS13}.
It is noetherian and admits the following explicit description. Let $\frg_\fro=\frn^{-}_\fro\oplus\frt_\fro\oplus\frn_\fro$ be a triangular decomposition of $\frg_\fro$. We fix basis elements $(f_i),(h_j)$ and $(e_i)$ of the $\fro$-modules $\frn^{-}_\fro, \frt_\fro$ and $\frn_\fro$ respectively. Then $D^{(m)}(\bbG(k))$ equals the $\fro$-subalgebra of $U(\frg) = U_\fro(\frg_\fro) \otimes_\fro L$ generated by the elements

\begin{numequation}\label{generators_n}
q^{(m)}_{\underline{\nu}}! \frac{(\vpi^k \underline{e})^{\underline{\nu}}}{\underline{\nu}!}  \cdot q^{(m)}_{\underline{\nu}'}! \vpi^{k |\underline{\nu}'|}\binom{\underline{h}}{\underline{\nu}'} \cdot q^{(m)}_{\underline{\nu}''}! \frac{(\vpi^k \underline{f})^{\underline{\nu}''}}{\underline{\nu}''!} \;.
\end{numequation}

An element of this type has order $d = |\unu|+|\unu'|+|\unu''|$, and the $\fro$-span of elements of order $\le d$ form an $\fro$-submodule $D^{(m)}_d(\bbG(k)) \sub D^{(m)}(\bbG(k))$, and $D^{(m)}(\bbG(k))$ becomes in this way a filtered $\fro$-algebra. In the case of the group ${\rm GL}_2$ we considered the same algebra in \cite[3.3.1]{PSS4} (denoted differently there). $D^{(m)}(\bbG(k))$ is a noetherian ring \cite[4.1.13]{HS13}, and so is its $p$-adic completion $\widehat{D}^{(m)}(\bbG(k))$ \cite{McConnell78}. The ring $D^{(m)}(\bbG(k))$ obviously contains the enveloping algebra $U_\fro(\vpi^k \frg_\fro)$ of $\vpi^k \frg_\fro$ over $\fro$, and the inclusion $U_\fro(\vpi^k \frg_\fro) \ra D^{(m)}(\bbG(k))$ induces an isomorphism of $L$-algebras $U(\frg) \car  D^{(m)}(\bbG(k)) \otimes_\fro L$. Denote by $Z(\frg)$ the center of $U(\frg)$, and let $\theta_0: Z(\frg) \ra L$ be the character with which the center acts on the trivial one-dimensional representation of $\frg$. We are now going to use a key result by Beilinson and Bernstein from \cite{BB81}.
\end{para}

\begin{prop}\label{BB} (i) Let $\pr: X \ra X_0$ be an admissible blow-up. There is a unique filtered $L$-algebra homomorphism

\begin{numequation}\label{BB_surj} Q_{X,k,L}: U(\frg) \lra H^0(X,\sD^{(k,m)}_X) \otimes_\fro L \;,
\end{numequation}

such that the following diagram is commutative

\begin{numequation}\label{BB_surj_diagram}
\xymatrix{
\frg \ar@{^{(}->}[d] \ar[r] & H^0(X, \sT_{X,k}) \otimes_\fro L \ar[d]\\
U(\frg) \ar@{->>}[r] & H^0(X,\sD^{(k,m)}_X) \otimes_\fro L \\
}
\end{numequation}

Here, the upper horizontal map is obtained
from the map $\vpi^k\frg_\fro \ra H^0(X, \sT_{X,k})$ in \ref{tcT_lemma} by tensoring with $L$. The vertical map on the right is induced by the canonical homomorphism of sheaves $\sT_{X,k} \ra \sD^{(k,m)}_X$.

\vskip8pt

(ii) $Q_{X,k,L}$ is surjective and its kernel is the two-sided ideal $U(\frg)\ker(\theta_0)$ so that $Q_{X,k,L}$ induces an isomorphism $U(\frg)_{\theta_0} \car H^0(X,\sD^{(k,m)}_X) \otimes_\fro L$, where $U(\frg)_{\theta_0} = U(\frg)/U(\frg)\ker(\theta_0)$.
\end{prop}

\begin{proof} We first note that by \ref{graded_tcD} and \ref{tcT_lemma} we have $\pr_*(\sD^{(k,m)}_X) = \sD^{(k,m)}_{X_0}$ and therefore $H^0(X,\sD^{(k,m)}_X) = H^0(X_0,\sD^{(k,m)}_{X_0})$. Flat base change gives us $H^0(X_0,\sD^{(k,m)}_{X_0}) \otimes_\fro L = H^0(\bbX,\sD_\bbX) \otimes_\fro L$, where $\sD_\bbX$ is the sheaf of differential operators on the flag variety $\bbX$. The existence and uniqueness of $Q_{X,k,L}$ follow from the universal property of $U(\frg)$. The assertions about the surjectivity and kernel of this map are simply restatements of \cite[Lemme 3]{BB81}, cf. also \cite[11.2.2]{Hotta}. \end{proof}

\vskip8pt

\begin{prop}\label{global_sections_tcD}  Let $\pr: X \ra X_0$ be an admissible blow-up. There is a canonical homomorphism of filtered $\fro$-algebras

\begin{numequation}\label{xi_uncompl} Q^{(k,m)}_X: D^{(m)}(\bbG(k)) \lra H^0(X, \sD^{(k,m)}_X) \;,
\end{numequation}

such that the following diagram is commutative

\begin{numequation}\label{xi_diagram}
\xymatrix{
D^{(m)}(\bbG(k)) \ar@{^{(}->}[d] \ar[r] & H^0(X,\sD^{(k,m)}_X) \ar[d]\\
U(\frg) \ar@{->>}[r] & H^0(X,\sD^{(k,m)}_X) \otimes_\fro L \\
}
\end{numequation}

Here, the lower horizontal map is the map $Q_{X,k,L}$ in \ref{BB_surj}. In particular, the map $Q^{(k,m)}_X$ induces an isomorphism

$$\Big(D^{(m)}(\bbG(k)) \otimes_\fro L\Big)/\Big(D^{(m)}(\bbG(k)) \otimes_\fro L\Big)\ker(\theta_0) \car H^0(X, \sD^{(k,m)}_X) \otimes_\fro L \;.$$

\vskip8pt

\end{prop}

\begin{proof} We begin with a remark on sheaves of filtered $\fro$-algebras and their associated sheaves of Rees rings. This material, in the setting of rings, instead of sheaves of rings, is well-known (cf. \cite[ch. 12, \S 6]{MCR}, \cite[ch. I, \S 4]{LiHuishi}), and its version for sheaves is entirely analogous. A sheaf of filtered $\fro$-algebra $\cA$ with positive filtration $(F_d \cA)_{d \ge 0}$ and $\fro \sub F_0 \cA$ gives rise to the sheaf of graded rings $R(\cA) := \oplus_{d\geq 0} F_d\cA t^d$, its associated sheaf of {\it Rees rings}. This is a sheaf of subrings of the polynomial algebra $\cA[t]$ over $\cA$. The sheaf of Rees rings is equipped with the filtration by the sheaves of subgroups $R_d(\cA) = \oplus_{i = 0}^d F_i\cA t^i \sub R(\cA)$. Specialising $R(\cA)$ in an element $\lambda \in \fro$ yields a sheaf of filtered subrings $\cA_{\lambda}$ of $\cA$. Precisely, $\cA_\lambda$ equals the image under the homomorphism of sheaves of rings $R(\cA) \ra \cA, t \mapsto \lambda$. We equip $\cA_\lambda = \sum_{d\geq 0} \lambda^{d} F_d \cA$ with the filtration induced by $\cA$.

\begin{claim}\label{claim} If the sheaf of graded rings ${\rm gr}(\cA)$, associated with the filtration $(F_d \cA)_d$, is flat over $\fro$, then for all $d$

$$F_d(\cA_\lambda)=\sum_{0 \le i \le d} \lambda^{i}F_i \cA \;.$$

\vskip8pt
\end{claim}

{\it Proof of the claim.} The right hand side is obviously contained in the left hand side. So we only have to show the other inclusion. Consider an element $x \in F_d(\cA_\lambda)$, and write it as  $x = \sum_{i=0}^n \lambda^i x_i$ with $n \ge d$ and $x_i \in F_i \cA$ for $i = 0, \ldots, n$. Put $x' = \sum_{i=0}^d \lambda^i x_i$. Then $x'$ is contained in the right hand
side, and it suffices to see that $x'' = x - x'$ lies in the right hand
side too. Set $y = \sum_{i=d+1}^n\lambda^{i-d-1} x_i$ so that $x'' = \lambda^{d+1} y$. If $y$ does not lie in $F_d \cA$, then choose $j>d$ such that $y \in F_j\cA \setminus F_{j-1}\cA$. Then the symbol $\sigma(y) := y + F_{j-1}\cA$ in $F_j\cA / F_{j-1}\cA$ is nonzero, but
$\lambda^{d+1}\sigma(y) = \lambda^{d+1}y + F_{j-1}\cA = x'' + F_{j-1}\cA$ is zero in ${\rm gr}_j \cA$, since $x''$ lies in $F_d(\cA_\lambda) \subset F_d \cA \subset F_{j-1}\cA$. Because we assume that ${\rm gr}(\cA)$ is flat over $\fro$, this implies that $\lambda^{d+1} = 0$, i.e., $\lambda = 0$. But then $x = x_0$ is contained in the right hand side. On the other hand, if $y$ lies in $F_d(\cA)$, then $x'' = \lambda^{d+1}y$ lies in the right hand side. \qed

\vskip8pt

For fixed $\lambda$, the formation of $\cA_{\lambda}$ is functorial in $\cA$. We now consider the canonical homomorphism of filtered $\fro$-algebras

$$Q_m: D^{(m)}(\bbG(0)) \lra H^0(X_0,\sD^{(m)}_{X_0})$$

\vskip8pt

appearing in \cite[4.4.5]{HS13}. It comes by functoriality from the right $\bbG_0$-action on $X_0$. After tensoring with $L$ the morphism $Q_m$ is equal to the map $Q_{X_0,0,L}$ of \ref{BB_surj}. Given an $\fro$-algebra $A$ we will denote by $\underline A$ the corresponding constant sheaf on $X_0$. The map $Q_m$ then gives rise to an homomorphism of associated constant sheaves of filtered $\fro$-algebras

$${\underline Q}_m: \underline{D^{(m)}(\bbG(0))} \lra \underline{H^0(X_0,\sD^{(m)}_{X_0})} \;.$$

\vskip8pt

We compose this map with the canonical map of sheaves $\underline{H^0(X_0,\sD^{(m)}_{X_0})} \ra \sD^{(m)}_{X_0}$ and obtain a homomorphism of sheaves of filtered $\fro$-algebras

$$\underline{D^{(m)}(\bbG(0))} \lra  \sD^{(m)}_{X_0} \;.$$

\vskip8pt

To this map we now apply the remark regarding Rees rings (and sheaves of Rees rings) we made in the beginning. That is, we pass to the sheaves of Rees rings associated with the filtrations (on the domain and target of this map), and then we specialize the parameter on both sides to $t = \varpi^k$. This gives a filtered homomorphism of sheaves of filtered $\fro$-algebras

$$\underline{D^{(m)}(\bbG(0))}_{\varpi^k} \lra \Big(\sD^{(m)}_{X_0}\Big)_{\varpi^k} \;.$$

\vskip8pt

The definition of the filtration on $D^{(m)}(\bbG(k))$, cf. \ref{div_power_env_algs}, together with \ref{claim}, imply that $D^{(m)}(\bbG(0))_{\varpi^k} = D^{(m)}(\bbG(k))$ as filtered subrings of $D^{(m)}(\bbG(0))$, and it follows from this that there is a canonical identification

$$\underline{D^{(m)}(\bbG(0))}_{\varpi^k} = \underline{D^{(m)}(\bbG(k))} \;.$$

\vskip8pt

The explicit description of sections over open affine subsets $U_0 \sub X_0$ in \ref{explicit_description_1}, together with \ref{claim}, imply that the sheaf $(\sD^{(m)}_{X_0})_{\varpi^k}$ coincides with $\sD^{(k,m)}_{X_0}$ as filtered subsheaves of $\sD^{(m)}_{X_0}$. We obtain thus a homomorphism of sheaves of filtered $\fro$-algebras

$$\underline{D^{(m)}(\bbG(k))} \lra \sD^{(k,m)}_{X_0} \;.$$

\vskip8pt

Taking global sections we obtain a homomorphism of filtered $\fro$-algebras

$$H^0\Big(X_0, \underline{D^{(m)}(\bbG(k))}\Big) \lra H^{0}(X_0,\sD^{(k,m)}_{X_0})$$

\vskip8pt

As $X_0$ is connected, the domain of this map is $D^{(m)}(\bbG(k))$. Moreover, in the situation considered here, we can apply \ref{tcT_lemma} (iii) and get that $\pr_* \cO_X = \cO_{X_0}$. We can thus use \ref{graded_tcD} and conclude that $H^0(X, \sD^{(k,m)}_{X}) = H^0(X_0, \sD^{(k,m)}_{X_0})$. This gives the homomorphism of filtered $\fro$-algebras

$$Q^{(k,m)}_X: D^{(m)}(\bbG(k)) \ra H^0(X, \sD^{(k,m)}_{X}) \;,$$

\vskip8pt

as claimed. The last assertion follows now from \ref{BB} (ii). \end{proof}

\vskip8pt

We put $\cA^{(k,m)}_X = \cO_{X}\otimes_{\fro}D^{(m)}(\bbG(k))$, and we equip this sheaf with the skew ring multiplication (smash product) coming from the action of $D^{(m)}(\bbG(k))$ on $\cO_{X}$ via $Q^{(k,m)}_X$. This is a sheaf of associative $\fro$-algebras.\footnote{The point here is that the algebra $D^{(m)}(\bbG(k))$ is an integral form of the universal enveloping algebra $U(\frg)$ and its action on $\cO_X$ is induced by the usual action of $U(\frg)$ on $\cO_{X,\bbQ}$. Since elements from $\frg$ act as derivations one may form Sweedler's smash product algebra $\cO_{X,\bbQ}\#U(\frg)$ \cite[7.2]{Sweedler}, cf. also \cite[1.7.10]{MCR}. It is associative and hence so is the subalgebra $\cA^{(k,m)}_X$.} This sheaf has a natural filtration whose associated graded equals the $\cO_{X}$-algebra $\cO_{X}\otimes_{\fro}\Sym^{(m)}(\Lie(\bbG(k)))$ \cite[Cor. 4.4.7 (iii)]{HS13}. In particular,
$\cA^{(k,m)}_X$ has noetherian sections over open affines. The map $Q_X^{(k,m)}$ induces a unique $\cO_X$-linear map
$\xi^{(k,m)}_X: \cA^{(k,m)}_X \ra \sD^{(k,m)}_X$ which is also a morphism of sheaves of filtered $\fro$-algebras.

\begin{prop}\label{prop-auxiliaryI}
The homomorphism $\xi^{(k,m)}_X: \cA^{(k,m)}_X \ra \sD^{(k,m)}_X$ is surjective.
\end{prop}

\begin{proof} We are going to adapt the argument of \cite[4.4.8.2 (ii)]{HS13}. The homomorphism is filtered. Applying $\Sym^{(m)}$ to the surjection in (ii) of \ref{tcT_lemma} we obtain a surjection

$$\cO_{X}\otimes_{\fro}\Sym^{(m)}(\Lie(\bbG(k)) \ra\Sym^{(m)}(\sT_{X,k})$$

\vskip8pt

which equals the associated graded homomorphism by \ref{finite_tDm}. Hence the homomorphism is surjective as claimed. \end{proof}

\vskip8pt

\begin{prop}\label{prop-auxiliaryII} Let $\cM$ be a coherent left $\cA^{(k,m)}_X$-module.

\vskip8pt

(i) $H^0(X,\cA^{(k,m)}_X)=D^{(m)}(\bbG(k))$.

\vskip8pt

(ii) There is a surjection $\cA^{(k,m)}_X(-r)^{\oplus s} \ra\cM$ of $\cA^{(k,m)}_X$-modules for suitable $r,s\geq 0$.

\vskip8pt

(iii) For any $i\geq 0$ the group $H^{i}(X,\cM)$ is a finitely generated $D^{(m)}(\bbG(k))$-module.

\vskip8pt

(iv) The ring $H^{0}(X,\sD^{(k,m)}_X)$ is a finitely generated $D^{(m)}(\bbG(k))$-module and hence noetherian.
\end{prop}

\begin{proof} Points (i)-(iii) are a restatement of \cite[3.3]{HS17}. By \ref{prop-auxiliaryI} the sheaf $\sD^{(k,m)}_X$ is a coherent $\cA^{(k,m)}_X$-module to which we can apply assertion (iii) with $i=0$. This proves statement (iv). \end{proof}

\vskip8pt

\begin{para}\label{passing_to_completion} {\it Passing to the completion.}  We now consider the formal scheme $\frX$ which is the formal completion of $X$ along its special fiber. We are interested in certain properties of the sheaves of rings $\hsD^{(k,m)}_{\frX}$ and $\sD^\dagger_{\frX,k}$ introduced in~\ref{rings_diff_op}. Put

$$\widehat{D}^{(m)}(\bbG(k))_{L,\theta_0} = \Big(\widehat{D}^{(m)}(\bbG(k))\otimes_\fro L \Big)/\Big(\widehat{D}^{(m)}(\bbG(k))\otimes_\fro L \Big)\ker(\theta_0) \;.$$

\vskip8pt

This is the same central reduction considered in \cite[sec. 3.3.1]{PSS4} for the group ${\rm GL}_2$.

\vskip8pt

In the proposition below, and in the remainder of this paper, certain rigid-analytic `wide open' groups $\Gkc$ will be important. To define them, consider first the formal completion $\frG(k)$ of the group scheme $\bbG(k)$ along its special fiber, which is a formal group scheme (of topologically finite type) over $\Spf(\fro)$. Then let $\widehat{\frG}(k)^\circ$ be the completion of ${\frG}(k)$ along its unit section $\Spf(\fro) \ra \frG(k)$, and denote by $\Gkc$ its associated rigid-analytic space, which is a rigid-analytic group.

\vskip8pt

Wide-open rigid-analytic groups play a special role in M. Emerton's approach to locally analytic representations of $p$-adic groups, cf. \cite{EmertonA}. The {\it analytic distribution algebra} of $\Gkc$ is defined to be the continuous dual space of the space of rigid-analytic functions on $\Gkc$, i.e.,

$$\cD^{\rm an}(\Gkc) := \cO_{\Gkc}(\Gkc)'_b = \Hom_L^{\rm cont}\Big(\cO_{\Gkc}(\Gkc),L\Big)_b \;,$$

\vskip8pt

which is equipped with the strong topology. This is a topological $L$-algebra of compact type. In \cite[sec. 5.2]{EmertonA} Emerton gives a description of this ring as the inductive limit of completions of the rings $\widehat{D}^{(m)}(\bbG(k))\otimes_\fro L$, i.e.,

\begin{numequation}\label{Emerton_descr} \cD^{\rm an}(\Gkc) \simeq \varinjlim_m \widehat{D}^{(m)}(\bbG(k))\otimes_\fro L \;.
\end{numequation}

This is an isomorphism of topological $L$-algebras of compact type, cf. \cite[5.2.6, 5.3.11]{EmertonA}, \cite[5.3.1]{HS13}.
\end{para}

\vskip8pt

\begin{prop}\label{global_sec_tsD} (i) The homomorphism $Q^{(k,m)}_X$ induces an algebra isomorphism

$$\widehat{D}^{(m)}(\bbG(k))_{L,\theta_0} \car H^0(\frX,\hsD^{(k,m)}_{\frX,\Q}) \,.$$

\vskip8pt

(ii) $H^0(\frX,\sD^\dagger_{\frX,k})$ and $\cD^{\rm an}(\Gkc)_{\theta_0}$ are canonically isomorphic topological $L$-algebras.
\end{prop}

\begin{proof} (i) For the purpose of this proof put $\ker(\theta_0)_\fro = D^{(m)}(\bbG(k)) \cap \ker(\theta_0)$. Because $D^{(m)}(\bbG(k))$ is an $\fro$-form of $U(\frg)$, it follows that $\ker(\theta_0)_\fro \otimes_\fro L = \ker(\theta_0)$. Now set $D^{(m)}(\bbG(k))_{\theta_0} := D^{(m)}(\bbG(k))/D^{(m)}(\bbG(k))\ker(\theta_0)_\fro$ and

$$D^{(m)}(\bbG(k))_{L,\theta_0} := \Big(D^{(m)}(\bbG(k)) \otimes_\fro L\Big)/\Big(D^{(m)}(\bbG(k)) \otimes_\fro L\Big)\ker(\theta_0) \;.$$

\vskip8pt

We then have $D^{(m)}(\bbG(k))_{\theta_0} \otimes_\fro L = D^{(m)}(\bbG(k))_{L,\theta_0}$. By \ref{global_sections_tcD}, the homomorphism of $\fro$-algebras $Q^{(k,m)}_X$ induces a homomorphism

$$Q^{(k,m)}_{X, \theta_0}: D^{(m)}(\bbG(k))_{\theta_0} \ra H^0(X,\sD^{(k,m)}_X) \;,$$

\vskip8pt

and the induced morphism

$$Q^{(k,m)}_{X, \theta_0} \otimes_\fro L: D^{(m)}(\bbG(k))_{L,\theta_0} \ra H^0(X,\sD^{(k,m)}_X) \otimes_\fro L$$

\vskip8pt

is an isomorphism of $L$-algebras. By \ref{prop-auxiliaryII} the ring $H^0(X,\sD^{(k,m)}_X)$ is a finitely generated $D^{(m)}(\bbG(k))_{\theta_0}$-module. We have now shown that all assumption in \cite[Lemma 3.5]{HS17} hold in the context considered here. By the very assertion of \cite[Lemma 3.5]{HS17} we find that $Q^{(k,m)}_{X, \theta_0}$ gives rise to an isomorphism

$$\widehat{D}^{(m)}(\bbG(k))_{L,\theta_0} \car \widehat{H}^0(X,\sD^{(k,m)}_X) \otimes_\fro L \;,$$

\vskip8pt

where $\widehat{H}^0(X,\sD^{(k,m)}_X)$ is the $p$-adic completion of $H^0(X,\sD^{(k,m)}_X)$. By \ref{completion}, we have a canonical isomorphism $\widehat{H}^0(X,\sD^{(k,m)}_X) \simeq H^0(\frX,\hsD^{(k,m)}_\frX)$. (We note that his does not introduce a circular argument, as section \ref{loc_flag_var} is only about sheaves of differential operators and their modules, and there is no connection made to distribution algebras.)

\vskip8pt

(ii) Follows from (i) and the isomorphism \ref{Emerton_descr} \end{proof}

\section{Localization on \texorpdfstring{$\frX$}{} via \texorpdfstring{$\sD^\dagger_{\frX,k}$}{}}\label{loc_n}
\label{loc_flag_var}
The general line of arguments developed here follows fairly closely \cite{NootHuyghe09}. As in the previous section, $\pr: X \ra X_0$ denotes an admissible blow-up of $X_0 = \bbB_0 \bksl \bbG_0$, and $\frX \ra \frX_0$ is the induced morphism between the completions of $X$ and $X_0$ along their special fibers, respectively. The number $k \geq k_X = k_\frX$, cf. \ref{defk_X}, \ref{defkX}, is fixed throughout this section so that the sheaves of rings $\sD^{(k,m)}_X$, $\hsD^{(k,m)}_\frX$, and $\sD^\dagger_{\frX,k}$ are defined.
\vskip12pt

\subsection{Cohomology of coherent \texorpdfstring{$\sD^{(k,m)}_{X}$}{}-modules}

\begin{lemma}\label{van_coh} Let $\cE$ be an abelian sheaf on $X$. For all $i >\dim X$ one has $H^i(X,\cE) =0$.
\end{lemma}

\begin{proof} Since the space $X$ is noetherian the result follows from Grothendieck's vanishing theorem \cite[Thm. 2.7]{HartshorneA}. \end{proof}

\vskip8pt

We recall that the sheaf $\sD^{(k,m)}_X$ has been equipped with a filtration, cf. \ref{filtrations}. We denote by ${\rm gr} \left( \sD^{(k,m)}_X \right)$ the associated sheaf of graded rings.

\vskip8pt

\begin{prop}\label{vanishing_coh_gr_Dnk} There is a natural number $r_0$ such that for all $r \ge r_0$ and all $i \ge 1$ one has

\begin{numequation}\label{vanishing_coh_gr_Dnk_disp}
H^i\left(X,{\rm gr} \left( \sD^{(k,m)}_X \right)(r) \right) = 0 \;.
\end{numequation}

\end{prop}

\begin{proof} Since $\cL_X$ is very ample over $\fro$ by \ref{v_ample_sh_lemma},
the Serre theorems \cite[II.5.17/III.5.2]{HartshorneA} imply that there is a number $u_0$ such that for all $u\geq u_0$ the module $\cO_X(u)$ is generated by global sections and has no higher cohomology. After this remark we prove the proposition along the lines of
\cite[Prop. 2.2.1]{NootHuyghe09}. By \cite[1.6.1]{NootHuyghe09}, the tangent sheaf $\sT_{X_0}$ is is generated by its global sections, and hence there is an $\cO_{X_0}$-linear surjection $(\cO_{X_0})^{\oplus a} \ra\sT_{X_0}$ for a suitable natural number $a$. Applying $(\pr)^*$ and multiplying by $\varpi^k$ gives an $\cO_X$-linear surjection
$(\cO_X)^{\oplus a}\simeq \varpi^k (\cO_X)^{\oplus a} \ra \sT_{X,k}$. By functoriality we get a surjective morphism of algebras

$$ \cC := \Sym^{(m)}((\cO_{X})^{\oplus a}) \lra \Sym^{(m)}(\sT_{X,k}) \;.$$

\vskip8pt

The target of this map equals ${\rm gr}\left(\sD^{(k,m)}_{X}\right)$ according to \ref{finite_tDm}. It therefore suffices to prove the following: given a coherent $\cC$-module $\cE$, there is a number $r_0$
such that for all $r\geq r_0$ and $i\geq 1$, one has $H^{i}(X,\cE(r))=0$. Since $\cE$ is $\cC$-coherent, it is a quasi-coherent $\cO_X$-module. Because $X$ is noetherian,  $\cE$ equals the union over its $\cO_X$-coherent submodules $\cE_i$ \cite[9.4.9]{EGA_I}. Again, since $\cE$ is $\cC$-coherent and $\cC$ has noetherian sections over open affines \cite[1.3.6]{Huyghe97}, there is a $\cC$-linear surjection $\cC\otimes_{\cO_{X}} \cE_i \ra \cE$. Choose a number $s_0$ such that $\cE_i(-s_0)$ is generated by global sections. We obtain a $\cO_{X}$-linear surjection $\cO_{X}(s_0)^{\oplus a_0} \ra \cE_i$ for a number $a_0$. This yields a $\cC$-linear surjection

$$\cC_0:=\cC(s_0)^{\oplus a_0} \lra \cE \;.$$

\vskip8pt

The $\cO_X$-module $\cC_0$ is graded and each homogeneous component equals a sum of copies of $\cO_X(s_0)$. It follows that
$H^{i}(X,\cC_0(r))=0$ for all $r\geq u_0-s_0$ and all $i\geq 1$. The rest of the argument proceeds now as in \cite[2.2.1]{NootHuyghe09}. \end{proof}

\vskip8pt

\begin{cor}\label{vanishing_coh_Dnk} Let $r_0$ be the number occuring in the preceding proposition. For all $r \ge r_0$ and all $i \ge 1$ one has

\begin{numequation}\label{vanishing_coh_Dnk_disp}
H^i\left(X,\sD^{(k,m)}_{X}(r)\right) = 0 \;.
\end{numequation}
\end{cor}

\begin{proof} For $d\geq 0$ we let $\cF_d = \sD^{(k,m)}_{X,d}$. We consider the exact sequence

\begin{numequation}\label{ex_fil_seq_1}
0 \ra \cF_{d-1} \ra \cF_d  \ra {\rm gr}_d\left(\sD^{(k,m)}_{X}\right) \ra 0
\end{numequation}

\vskip8pt

(where $\cF_{-1}:=0$) from which we deduce the exact sequence

\begin{numequation}\label{ex_fil_seq_2}
0 \ra \cF_{d-1}(r) \ra \cF_d(r)  \ra {\rm gr}_d\left(\sD^{(k,m)}_{X}\right)(r) \ra 0
\end{numequation}

\vskip8pt

because tensoring with a line bundle is an exact functor. Since cohomology commutes with direct sums, we have for all $r\geq r_0$ and $i\geq 1$ that

$$H^{i}(X,{\rm gr}_d\left(\sD^{(k,m)}_{X}\right)(r))=0$$

\vskip8pt

according to the preceding proposition. Using the sequence \ref{ex_fil_seq_2} we can then deduce by induction on $d$ that for all $r\geq r_0$ and $i\geq 1$

$$H^i(X, \cF_d(r)) = 0 \;.$$

\vskip8pt

Because cohomology commutes with inductive limits on a noetherian scheme we obtain the asserted vanishing result. \end{proof}

\vskip8pt

\begin{prop}\label{surjection} Let $\cE$ be a coherent $\sD^{(k,m)}_{X}$-module.
		
\vskip8pt

(i) There is a number $r=r(\cE) \in \Z$ and $s \in \Z_{\ge 0}$ and an epimorphism of $\sD^{(k,m)}_{X}$-modules

$$\Big(\sD^{(k,m)}_{X}(-r)\Big)^{\oplus s} \twoheadrightarrow \cE \;.$$

\vskip8pt
				
(ii) There is $r_1(\cE) \in \Z$ such that for all $r \ge r_1(\cE)$ and all $i >0$

$$H^i\Big(X, \cE(r)\Big) = 0 \;.$$

\vskip8pt
\end{prop}

\begin{proof} (i) As $X$ is a noetherian scheme, $\cE$ is the inductive limit of its coherent subsheaves.  There is thus a coherent $\cO_{X}$-submodule $\cF \sub \cE$ which generates $\cE$ as a $\sD^{(k,m)}_{X}$-module, i.e., there is an epimorphism of sheaves

$$\sD^{(k,m)}_{X} \otimes_{\cO_{X}} \cF \stackrel{\alpha}{\lra} \cE \;,$$

\vskip8pt

where $\sD^{(k,m)}_{X}$ is considered with its right $\cO_{X}$-module structure. Next, there is $r>0$ such that the sheaf

$$\cF(r) = \cF \otimes_{\cO_{X}} \cL_X^{\otimes r}$$

\vskip8pt

is generated by its global sections. Hence there is $s > 0$ and an epimorphism $\cO_{X}^{\oplus s} \twoheadrightarrow \cF(r)$, and thus an epimorphism of $\cO_{X}$-modules

$$\left(\cO_{X}(-r)\right)^{\oplus s} \twoheadrightarrow \cF \;.$$

\vskip8pt

From this morphism we get an epimorphism of $\sD^{(k,m)}_{X}$-modules

$$\left(\sD^{(k,m)}_{X}(-r)\right)^{\oplus s}  = \sD^{(k,m)}_{X} \otimes_{\cO_{X_n}} \left(\cO_{X}(-r)\right)^{\oplus s} \twoheadrightarrow \sD^{(k,m)}_{X} \otimes_{\cO_{X}} \cF  \stackrel{\alpha}{\lra} \cE \;.$$

\vskip8pt

(ii) Consider for $i\geq 1$ the following assertion $(a_i)$: for any coherent $\sD^{(k,m)}_{X}$-module $\cE$, there is a number $r_i(\cE)$ such that for all $r\geq r_i(\cE)$ and all $i\leq j$ one has $H^{j}(X,\cE(r))=0$. For $i>\dim X$ the assertion holds, cf. \ref{van_coh}. Suppose the statement $(a_{i+1})$ holds.
Using (i) we find an epimorphism of $\sD^{(k,m)}_{X}$-modules

$$\beta: \cC_0:=\Big(\sD^{(k,m)}_{X}(s_0)\Big)^{\oplus s} \twoheadrightarrow \cE$$

\vskip8pt

for numbers $s_0\in\bbZ$ and $s\geq 0$. By \ref{finite_tDm}, the kernel $\cR = \ker(\beta)$ is a coherent $\sD^{(k,m)}_{X}$-module.
Recall the number $r_0$ of the preceding corollary. For any $r\geq \max (r_0-s_0, r_{i+1}(\cR))$ we have the exact sequence

$$0 = H^{i}(X,\cC_0(r)) \lra H^{i}(X,\cE(r)) \lra H^{i+1}(X, \cR(r)) = 0$$

\vskip8pt

which shows $H^{i}(X,\cE(r))=0$ for these $r$. So we may take as $r_i(\cE)$ any of these $r$ which is larger than $r_{i+1}(\cE)$ and obtain the statement $(a_i)$. In particular, $(a_1)$ holds which proves (ii). \end{proof}

\vskip8pt

\begin{prop}\label{finite_power}

\vskip8pt

(i) Fix $r \in \Z$. There is $c_1 = c_1(r) \in \Z_{\ge 0}$ such that for all $i>0$ the cohomology group $H^i(X,\sD^{(k,m)}_{X}(r))$ is annihilated by $p^{c_1}$.

\vskip8pt

(ii) Let $\cE$ be a coherent $\sD^{(k,m)}_{X}$-module. There is $c_2 = c_2(\cE) \in \Z_{\ge 0}$ such that for all $i>0$ the cohomology group $H^i(X,\cE)$ is annihilated by $p^{c_2}$.
\end{prop}

\begin{proof} (i) Since the blow-up morphism $\pr: X \ra X_0$ becomes an isomorphism over $X_0\times_{\fro} L$ any coherent module over $\sD^{(k,m)}_{X}\otimes \Q$ induces a coherent module over the sheaf of usual differential operators on $X_0\times_\fro L$. By \cite{BB81} we conclude that the global section functor on $X$ is exact for coherent $\sD^{(k,m)}_{X} \otimes_\Z \Q$-modules. In particular, the cohomology group $H^i(X,\sD^{(k,m)}_{X}(r))$ is $p$-torsion. To see that the torsion is bounded, we deduce from \ref{prop-auxiliaryI} that $\sD^{(k,m)}_{X}(r)$ is a coherent module over $\cA^{(k,m)}_X$. According to \ref{prop-auxiliaryII}, $H^i(X,\sD^{(k,m)}_{X}(r))$ is therefore finitely generated over $D^{(m)}(\bbG(k))$. Now consider a finite set of generators of $H^i(X,\sD^{(k,m)}_{X}(r))$ as $D^{(m)}(\bbG(k))$-module. These are annihilated by a finite power $p^{c_{1,i}}$ of $p$, and since there are only finitely many integers $i>0$ with non-zero $H^i(X,\sD^{(k,m)}_{X}(r))$, cf. \ref{van_coh}, we can take $c_2 := \max
\{c_{2,i} \midc i \ge 0\}$.

\vskip8pt

(ii) We consider for any $i\geq 1$ the following assertion $(a_i)$: for any coherent $\sD^{(k,m)}_{X}$-module $\cE$, there is a number $r_i(\cE)$ such that the groups $H^{j}(X,\cE), i\leq j$ are all annihilated by $p^{r_i(\cE)}$. For $i>\dim X$ the assertion is true, cf. \ref{van_coh}. Let us assume that $(a_{i+1})$ holds and consider an arbitrary coherent $\sD^{(k,m)}_{X}$-module $\cE$. Acccording to \ref{surjection} we have a $\sD^{(k,m)}_{X}$-linear surjection

$$\cE_0:=\sD^{(k,m)}_{X}(r)^{\oplus s} \lra \cE$$

\vskip8pt

for numbers $r\in\bbZ$ and $s\geq 0$. Let $\cE'$ be the kernel. We have an exact sequence

$$ H^{i}(X,\cE_0)\stackrel{\iota}{\ra}H^{i}(X,\cE)\stackrel{\delta}{\ra}H^{i+1}(X,\cE') \;.$$

\vskip8pt

Then $p^{c_1(r)}$ annihilates the image of $\iota$ according to (i) and $p^{r_{i+1}(\cE')}$ annihilates the image of $\delta$ according to $(a_{i+1})$. So we may take as $r_i(\cE)$ any number greater than the maximum of $r_{i+1}(\cE)$ and $c_1(r)+r_{i+1}(\cE')$ and obtain the statement $(a_i)$. In particular, $(a_1)$ holds which proves (ii). \end{proof}

\subsection{Cohomology of coherent \texorpdfstring{$\hsD^{(k,m)}_{\frX,\Q}$}{}-modules}\label{coh_coh_tsD_mod}

We denote by $X_{j}$ the reduction of $X$ modulo $p^{j+1}$.

\vskip8pt

\begin{prop}\label{completion} Let $\cE$ be a coherent $\sD^{(k,m)}_{X}$-module on $X$ and $\widehat{\cE} = \varprojlim_j \cE/p^{j+1}\cE$ its $p$-adic completion, which we consider as a sheaf on $\frX$.
		
\vskip8pt

(i) For all $i \ge 0$ one has $H^i(\frX,\widehat{\cE}) = \varprojlim_j H^i\left(X_{j},\cE/p^{j+1}\cE \right)$.

\vskip8pt

(ii) For all $i>0$ one has $H^i(\frX,\widehat{\cE}) = H^i(X,\cE)$.

\vskip8pt

(iii) $H^0(\frX,\widehat{\cE}) = \varprojlim_j H^0(X,\cE)/p^{j+1}H^0(X,\cE)$.
\end{prop}

\begin{proof} Put $\cE_j = \cE/p^{j+1}\cE$. Let $\cE_t$ be the subsheaf defined by

$$\cE_t(U) = \cE(U)_{\rm tor} \;,$$

\vskip8pt

where the right hand side denotes the group of torsion elements in $\cE(U)$. This is indeed a sheaf (and not only a presheaf) because $X$ is a noetherian space. Furthermore, $\cE_t$ is a $\sD^{(k,m)}_{X}$-submodule of $\cE$. Because the sheaf $\sD^{(k,m)}_{X}$ has noetherian rings of sections over open affine subsets of $X$, cf. \ref{prop-exactdirectimage}, the submodule $\cE_t$ is a coherent $\sD^{(k,m)}_{X}$-module. $\cE_t$ is thus generated by a coherent $\cO_{X}$-submodule $\cF$ of $\cE_t$. The submodule $\cF$ is annihilated by a fixed power $p^c$ of $p$, and so is $\cE_t$. Put $\cG = \cE/\cE_t$, which is again a coherent $\sD^{(k,m)}_{X}$-module. Using \ref{finite_power}, we can then assume, after possibly replacing $c$ by a larger number, that

\vskip8pt

$$\begin{array}{cl}
(a) & p^c\cE_t = 0 \;,\\
(b) & \mbox{for all } i>0: p^cH^i(X,\cE) = 0 \;,\\
(c) & \mbox{for all } i>0: p^cH^i(X,\cG) = 0 \;.\\
\end{array}$$

\vskip8pt

From here on the proof of the proposition is exactly as in \cite[4.2.1]{PSS4}. \end{proof}

\vskip8pt

\begin{prop}\label{surj_compl} Let $\sE$ be a coherent $\hsD^{(k,m)}_{\frX}$-module.

\vskip8pt

(i) There is $r_1(\sE) \in \Z$ such that for all $r \ge r_1(\sE)$ there is $s \in \Z_{\ge 0}$ and an epimorphism of $\hsD^{(k,m)}_{\frX}$-modules

$$\Big(\hsD^{(k,m)}_{\frX}(-r)\Big)^{\oplus s} \twoheadrightarrow \sE \;.$$

\vskip8pt

(ii) There is $r_2(\sE) \in \Z$ such that for all $r \ge r_2(\sE)$ and all $i >0$
				
$$H^i\Big(\frX, \sE(r)\Big) = 0 \;.$$

\vskip8pt
\end{prop}

\begin{proof} (i) Because $\sE$ is a coherent $\hsD^{(k,m)}_{\frX}$-module, and because $H^0(U,\hsD^{(k,m)}_{\frX})$ is a noetherian ring for all open affine subsets $U \sub \frX$, cf. \ref{prop-exactdirectimage}, the torsion submodule $\sE_t \sub \sE$ is again a coherent $\hsD^{(k,m)}_{\frX}$-module.
As $\frX$ is quasi-compact, there is $c \in \Z_{\ge 0}$ such that $p^c \sE_t = 0$.
Put $\sG = \sE/\sE_t$ and $\sG_0 = \sG/p\sG$. For $j \ge c$ one has an exact sequence

$$0 \ra \sG_0 \stackrel{p^{j+1}}{\lra} \sE_{j+1} \ra \sE_j \ra 0 \;.$$

\vskip8pt

We note that the sheaf $\sG_0$ is a coherent module over
$\hsD^{(k,m)}_{\frX}/p\hsD^{(k,m)}_{\frX}$. We view $\frX$ as a closed subset of $X$ and denote the closed embedding temporarily by $i$. Because the canonical map of sheaves of rings

\begin{numequation}\label{iso_uncompl_compl}
\sD^{(k,m)}_{X}/p\sD^{(k,m)}_{X} \car i_*\Big(\hsD^{(k,m)}_{\frX}/p\hsD^{(k,m)}_{\frX}\Big)
\end{numequation}

is an isomorphism, $i_*\sG_0$ can be considered a coherent $\sD^{(k,m)}_{X}$-module via this isomorphism. Hence we can apply \ref{surjection} to $i_*\sG_0$ and deduce that there is $r_2(\sG_0)$ such that for all $r \ge r_2(\sG_0)$ one has

$$H^1(\frX,\sG_0(r))= H^1(X,i_*\sG_0(r)) = 0 \;.$$

\vskip8pt

The canonical maps

\begin{numequation}\label{surj_H0}
H^0(\frX,\sE_{j+1}(r)) \lra H^0(\frX,\sE_j(r))
\end{numequation}

are thus surjective for $r \ge r_2(\sG_0)$ and $j \ge c$. Similarly, $\sE_c$ is a coherent module over $\sD^{(k,m)}_{X}/p^c\sD^{(k,m)}_{X}$-module, in particular a coherent $\sD^{(k,m)}_{X}$-module. By \ref{surjection} there is $r_1(\sE_c)$ such that for every $r \ge r_1(\sE_c)$ there is $s \in \Z_{\ge 0}$ and a surjection

$$\lambda: \Big(\sD^{(k,m)}_{X}/p^c\sD^{(k,m)}_{X}\Big)^{\oplus s} \twoheadrightarrow \sE_c(r) \;.$$

\vskip8pt

Let $r_1(\sE) = \max\{r_2(\sG_0),r_1(\sE_c)\}$, and assume from now on that $r \ge r_1(\sE)$. Let $e_1, \ldots, e_s$ be the standard basis of the domain of $\lambda$, and use \ref{surj_H0} to lift each $\lambda(e_t)$, $1 \le t \le s$, to an element of

$$\varprojlim_j H^0(\frX,\sE_j(r)) \simeq H^0(\frX,\widehat{\sE(r)}) \;,$$

\vskip8pt

by \ref{completion} (i). But $\widehat{\sE(r)} = \widehat{\sE}(r)$, and $\widehat{\sE} = \sE$, as follows from \cite[3.2.3 (v)]{BerthelotDI}. This defines a morphism

$$\Big(\hsD^{(k,m)}_{\frX}\Big)^{\oplus s} \lra \sE(r)$$

\vskip8pt

which is surjective because, modulo $p^c$, it is a surjective morphism of sheaves coming from coherent $\hsD^{(k,m)}_{\frX}$-modules by reduction modulo $p^c$, cf. \cite[3.2.2 (ii)]{BerthelotDI}.

\vskip8pt

(ii) We deduce from \ref{vanishing_coh_Dnk} and \ref{completion} that for all $i>0$

$$H^i\Big(\frX,\hsD^{(k,m)}_{\frX}(r)\Big) = 0 \;,$$

\vskip8pt

whenever $r \ge r_0$, where $r_0$ is as in \ref{v_ample_sh_lemma}. Since the sheaf $\hsD^{(k,m)}_{\frX}$ is coherent, cf. \ref{global_sec_tsD}, and $\frX$ is a noetherian space of finite dimension, the statement in (ii) can now be deduced by descending induction on $i$
exactly as in the proof of part (ii) of \ref{surjection}. \end{proof}

\vskip8pt

\begin{prop}\label{compl_finite_power_torsion} Let $\sE$ be a coherent $\hsD^{(k,m)}_{\frX}$-module.

\vskip8pt

(i) There is $c=c(\sE) \in \Z_{\ge 0}$ such that for all $i>0$
the cohomology group $H^i(\frX,\sE)$ is annihilated by $p^c$.

\vskip8pt

(ii) $H^0(\frX,\sE) = \varprojlim_j H^0(\frX,\sE)/p^jH^0(\frX,\sE)$.
\end{prop}

\begin{proof} (i) Let $r\in\bbZ$. By \ref{completion} we have for $i>0$ that

$$H^{i}(\frX,\hsD^{(k,m)}_{\frX}(-r)) =  H^{i}(X,\sD^{(k,m)}_{X}(-r)) \;,$$

\vskip8pt

and this is annihilated by a finite power of $p$, by \ref{finite_power}. The proof now proceeds by descending induction exactly as in the proof of part (ii) of \ref{finite_power}.

\vskip8pt

(ii) Let $\sE_t \sub \sE$ be the subsheaf of torsion elements and $\sG = \sE/\sE_t$. Then the discussion in the beginning of the proof of \ref{completion} shows that there is $c \in \Z_{\ge 0}$ such that $p^c\sE_t = 0$. Part (i) gives that $p^cH^1(\frX,\sE) = p^cH^1(\frX,\sG) = 0$, after possibly increasing $c$. Now we can apply the same reasoning as in the proof of \ref{completion} (iii) to conclude that assertion (ii) is true. \end{proof}

\vskip8pt

\begin{para} Let $\Coh(\hsD^{(k,m)}_{\frX})$ (resp.  $\Coh(\hsD^{(k,m)}_{\frX,\Q})$) be the category of coherent $\hsD^{(k,m)}_{\frX}$-modules (resp. $\hsD^{(k,m)}_{\frX,\Q}$-modules). Let $\Coh(\hsD^{(k,m)}_{\frX})_\Q$ be the category of coherent $\hsD^{(k,m)}_{\frX}$-modules up to isogeny. We recall that this means that $\Coh(\hsD^{(k,m)}_{\frX})_\Q$ has the same class of objects as $\Coh(\hsD^{(k,m)}_{\frX})$, and for any two objects $\cM$ and $\cN$ one has

$$\Hom_{\Coh(\hsD^{(k,m)}_{\frX})_\Q}(\cM,\cN) = \Hom_{\Coh(\hsD^{(k,m)}_{\frX})}(\cM,\cN) \otimes_\Z \Q \;.$$

\vskip8pt

\end{para}

\vskip8pt

\begin{prop}\label{integral_models} (i) The functor $\cM \rightsquigarrow \cM_\Q = \cM \otimes_\Z \Q$ induces an equivalence between $\Coh(\hsD^{(k,m)}_{\frX})_\Q$ and $\Coh(\hsD^{(k,m)}_{\frX,\Q})$.

\vskip8pt

(ii) For every coherent $\sD^\dagger_{\frX,k}$-module $\sM$ there is $m \ge 0$ and a coherent $\hsD^{(k,m)}_{\frX,\Q}$-module $\sM_m$ and an isomorphism of $\sD^\dagger_{\frX,k}$-modules

$$\vep: \sD^\dagger_{\frX,k} \otimes_{\hsD^{(k,m)}_{\frX,\Q}} \sM_m \car \sM \;.$$

\vskip8pt

If $(m', \sM_{m'},\vep')$ is another such triple, then there is $l \ge \max\{m,m'\}$ and an isomorphism of $\hsD^{(k,l)}_{\frX,\Q}$-modules

$$\vep_l: \hsD^{(k,l)}_{\frX,\Q} \otimes_{\hsD^{(k,m)}_{\frX,\Q}} \sM_m \car \hsD^{(k,l)}_{\frX,\Q}
\otimes_{\hsD^{(k,m')}_{\frX,\Q}} \sM_{m'}$$

\vskip8pt

such that $\vep' \circ \Big({\rm id}_{\sD^\dagger_{\frX,k}} \otimes \vep_l\Big) = \vep$.
\end{prop}

\begin{proof} (i) This is \cite[3.4.5]{BerthelotDI}. Note that the sheaf $\hsD^{(k,m)}_{\frX}$ satisfies the conditions in \cite[3.4.1]{BerthelotDI}, by \ref{global_sec_tsD}. We point out that the formal scheme $\cX$ in \cite[sec. 3.4]{BerthelotDI} is not supposed to be smooth over a discrete valuation ring, but only locally noetherian, cf. \cite[sec. 3.3]{BerthelotDI}.

\vskip8pt

(ii) This is \cite[3.6.2]{BerthelotDI}. In this reference the formal scheme is supposed to be noetherian and quasi-separated, but not necessarily smooth over a discrete valuation ring. \end{proof}

\vskip8pt

\begin{thm}\label{acycl_tsD} Let
 $\sE$ be a coherent $\hsD^{(k,m)}_{\frX,\Q}$-module (resp. $\sD^\dagger_{\frX,k}$-module).

\vskip8pt

(i) There is $r(\sE) \in \Z$ such that for all $r \ge r(\sE)$ there is $s \in \Z_{\ge 0}$ and an epimorphism of $\hsD^{(k,m)}_{\frX,\Q}$-modules (resp. $\sD^\dagger_{\frX,k}$-modules)

$$\Big(\hsD^{(k,m)}_{\frX,\Q}(-r)\Big)^{\oplus s} \twoheadrightarrow \sE \; \hskip16pt (\; \mbox{resp.} \; \Big(\sD^\dagger_{\frX,k}(-r)\Big)^{\oplus s} \twoheadrightarrow \sE \;) \;.
$$

\vskip8pt

(ii) For all $i >0$ one has $H^i(\frX, \sE) = 0$.
\end{thm}

\begin{proof} (a) We first show both assertions (i) and (ii) for a coherent $\hsD^{(k,m)}_{\frX,\Q}$-module $\sE$. By \ref{integral_models} (i) there is a coherent  $\hsD^{(k,m)}_{\frX}$-module $\sF$ such that $\sF \otimes_\Z \Q  = \sE$. We use \ref{surj_compl} to find for every $r \ge r_1(\sF)$ a surjection

$$\Big(\hsD^{(k,m)}_{\frX}(-r)\Big)^{\oplus s} \twoheadrightarrow \sF \;,$$

\vskip8pt

for some $s$ (depending on $r$).  Tensoring with $\Q$ gives then the desired surjection onto $\sE$. Hence assertion (i). Furthermore, for $i>0$

$$H^i(\frX,\sE) = H^i(\frX,\sF) \otimes_\Z \Q = 0 \;,$$

\vskip8pt

by \ref{compl_finite_power_torsion}, and this proves (ii).

\vskip8pt

(b) Now suppose $\sE$ is a coherent $\sD^\dagger_{\frX,k}$-module.
By \ref{integral_models} (ii) there is $m \ge 0$ and a coherent module $\sE_m$ over $\hsD^{(k,m)}_{\frX,\Q}$ and an isomorphism of $\sD^\dagger_{\frX,k}$-modules

$$\sD^\dagger_{\frX,k} \otimes_{\hsD^{(k,m)}_{\frX,\Q}} \sE_m \stackrel{\simeq}{\lra} \sE \;.$$

\vskip8pt

Now use what we have just shown for $\sE_m$ in (a) and get the sought for surjection after tensoring with $\sD^\dagger_{\frX,k}$. This proves the first assertion. We have

$$\sE = \sD^\dagger_{\frX,k} \otimes_{\hsD^{(k,m)}_{\frX,\Q}} \sE_m = \varinjlim_{\ell \ge m} \hsD^{(k,l)}_{\frX,\Q} \otimes_{\hsD^{(k,m)}_{\frX,\Q}} \sE_m  = \varinjlim_{\ell \ge m} \sE_\ell$$

\vskip8pt

where $\sE_\ell = \hsD^{(k,l)}_{\frX,\Q} \otimes_{\hsD^{(k,m)}_{\frX,\Q}} \sE_m$ is a coherent $\hsD^{(k,l)}_{\frX,\Q}$-module. Then we have for $i>0$

$$H^i(\frX,\sE) = \varinjlim_{\ell \ge m} H^i(\frX,\sE_\ell) = 0 \;,$$

\vskip8pt

by part (a). And this proves assertion (ii). \end{proof}

\vskip12pt

\subsection{\texorpdfstring{$\frX$}{} is \texorpdfstring{$\hsD^{(k,m)}_{\frX,\Q}$}{}-affine and \texorpdfstring{$\sD^\dagger_{\frX,k}$}{}-affine}
\label{Ddag-affinity}

\begin{prop}\label{prop-genglobal} (i) Let $\sE$ be a coherent $\hsD^{(k,m)}_{\frX,\Q}$-module. Then $\sE$ is generated by its global sections as $\hsD^{(k,m)}_{\frX,\Q}$-module. Furthermore, $\sE$ has a resolution by finite free $\hsD^{(k,m)}_{\frX,\Q}$-modules.

\vskip8pt

(ii) Let $\sE$ be a coherent $\sD^\dagger_{\frX,k}$-module. Then $\sE$ is generated by its global sections as $\sD^\dagger_{\frX,k}$-module. $H^0(\frX,\sE)$ is a $H^0(\frX,\sD^\dagger_{\frX,k})$-module of finite presentation.
Furthermore, $\sE$ has a resolution by finite free $\sD^\dagger_{\frX,k}$-modules.
\end{prop}

\begin{proof} (i) Using \ref{acycl_tsD} it remains to see that any $\hsD^{(k,m)}_{\frX,\Q}$-module of type $\hsD^{(k,m)}_{\frX,\Q}(-r)$ admits a linear surjection $(\hsD^{(k,m)}_{\frX,\Q})^{\oplus s} \ra\hsD^{(k,m)}_{\frX,\Q}(-r)$ for suitable $s\geq 0$. We argue as in \cite[5.1]{Huyghe97}. Let $M:=H^0(X,\sD^{(k,m)}_{X}(-r))$, a finitely generated $D^{(m)}(\bbG(k))$-module by \ref{prop-auxiliaryII}. Consider the linear map of $\sD^{(k,m)}_{X}$-modules equal to the composite

$$ \sD^{(k,m)}_{X}\otimes_{D^{(m)}(\bbG(k))} M \ra \sD^{(k,m)}_{X}\otimes_{H^0(X,\sD^{(k,m)}_{X})} M \ra \sD^{(k,m)}_{X}(-r)$$

\vskip8pt

where the first map is the surjection induced by the map $Q^{(k,m)}_X$ appearing in \ref{global_sections_tcD}.
Let $\cE$ be the cokernel of the composite map. Since $D^{(m)}(\bbG(k))$ is noetherian, the source of the map is coherent and hence $\cE$ is coherent. Moreover, $\cE\otimes\Q=0$ since $\sD^{(k,m)}_{X}(-r)\otimes\Q$ is generated by global sections \cite{BB81}. All in all, there is $i$ with $p^{i}\cE=0$. Now choose a linear surjection  $(D^{(m)}(\bbG(k)))^{\oplus s} \ra M$. We obtain the exact sequence of coherent modules

$$(\sD^{(k,m)}_{X})^{\oplus s} \ra \sD^{(k,m)}_{X}(-r) \ra \cE \ra 0 \;.$$

\vskip8pt

Passing to $p$-adic completions (which is exact in our situation \cite[3.2]{BerthelotDI}) and inverting $p$ yields the linear surjection

$$(\hsD^{(k,m)}_{\frX,\Q})^{\oplus s} \ra \hsD^{(k,m)}_{\frX,\Q}(-r) \;.$$

\vskip8pt

This shows (i).

\vskip8pt

(ii) This follows from (i) exactly as in \cite{Huyghe97}. \end{proof}

\begin{para} {\it The functors $\Loc^{(k,m)}_{\frX}$ and $\Loc^\dagger_{\frX,k}$.} Let $E$ be a finitely generated $H^0(\frX,\hsD^{(k,m)}_{\frX,\Q})$-module (resp. a finitely presented $H^0(\frX,\sD^\dagger_{\frX,k})$-module). Then we let $\Loc^{(k,m)}_{\frX}(E)$ (resp. $\Loc^\dagger_{\frX,k}(E)$) be the sheaf on $\frX$ associated to the presheaf

$$U \rightsquigarrow \hsD^{(k,m)}_{\frX,\Q}(U) \otimes_{H^0(\frX,\hsD^{k,m}_{k,\Q})} E \hskip16pt ({\rm resp.} \;\; U \rightsquigarrow \sD^\dagger_{\frX,k}(U) \otimes_{H^0(\frX,\sD^\dagger_{\frX,k})} E \;) \;.$$

\vskip8pt

It is obvious that $\Loc^{(k,m)}_{\frX}$ (resp. $\Loc^\dagger_{\frX,k}$) is a functor from the category of finitely generated $H^0(\frX,\hsD^{(k,m)}_{\frX,\Q})$-modules (resp. finitely presented $H^0(\frX,\sD^\dagger_{\frX,k})$-modules) to the category of sheaves of modules over $\hsD^{(k,m)}_{\frX,\Q}$ (resp. $\sD^\dagger_{\frX,k}$).
\end{para}

\vskip8pt

\begin{thm}\label{thm-equivalence} (i) The functors $\Loc^{(k,m)}_{\frX}$ and $H^0$ (resp. $\Loc^\dagger_{\frX,k}$ and $H^0$) are quasi-inverse equivalences between the categories of finitely generated $H^0(\frX,\hsD^{(k,m)}_{\frX,\Q})$-modules and coherent $\hsD^{(k,m)}_{\frX,\Q}$-modules (resp. finitely presented $H^0(\frX,\sD^\dagger_{\frX,k})$-modules and coherent $\sD^\dagger_{\frX,k}$-modules).

\vskip8pt

(ii) The functor $\Loc^{(k,m)}_{\frX}$ (resp. $\Loc^\dagger_{\frX,k}$) is an exact functor.
\end{thm}

\begin{proof} The proof of (i) uses the same arguments as the proof of \cite[2.3.7]{NootHuyghe09}. The second assertion then follows because any equivalence between abelian categories is exact. \end{proof}

\section{Localization of representations of \texorpdfstring{$\bbG(L)$}{}}\label{loc}

Although we do recall a few basic facts in the beginning of this section, we assume from now on some familiarity with the theory of locally analytic representations as developed by P. Schneider and J. Teitelbaum \cite{ST02b, ST03}, and we also make use of the point of view introduced by M. Emerton in \cite{EmertonA}.

\vskip8pt

For the sake of convenience, all representations which we consider in this section are on topological $L$-vector spaces, and all modules over distribution algebras are topological $L$-vector spaces. We thus assume throughout this section that the so-called {\it coefficient field}, cf. \cite[beginning of sec. 2]{ST02b}, usually denoted by $K$ in papers like \cite{ST02b, ST03}, over which those topological vector spaces are defined, is equal to our base field $L$. However, all results in this section also hold when the representations (or the modules over distribution algebras) are topological $K$-vector spaces, where $K/L$ is a complete and discretely valued extension (such that the valuation topology on $K$ induces the valuation topology on $L$), cf. \ref{coeff_field}.

\subsection{Locally analytic representations and distribution algebras}\label{dist_algs_and_reps}

\begin{para}\label{module} {\it The module associated to a locally analytic representation.} In the following we will be interested in locally analytic representations of the compact locally $L$-analytic group $\GO = \bbG_0(\fro)$. Let $C^{\rm la}(\GO,L)$ be the space of $L$-valued locally $L$-analytic functions on $\GO$, and let

$$D(\GO,L) := C^{\rm la}(\GO,L)'_b$$

\vskip8pt

be its strong dual, i.e. its continuous dual space equipped with the strong topology, which carries the structure of a Fr\'echet-Stein algebra \cite[5.1]{ST03}. The product of $\delta_1, \delta_2 \in D(\GO,L)$ is defined by

$$(\delta_1 \cdot \delta_2)(f) = \delta_1\Big(x \mapsto \delta_2(y \mapsto f(xy))\Big) \;,$$

\vskip8pt

for $f \in C^{\rm la}(\GO,L)$. Given an admissible locally analytic representation $V$ of $\GO$, cf. \cite[sec. 6]{ST03}, we let $M := V'_b$ be its strong dual, which is, by the very definition of ``admissible representation'', a {\it coadmissible module} over $\DGO$. Explicitly, if we denote by $g.v$ the action of $g \in \GO$ on $v \in V$, then the $\DGO$-module structure on $M$ is given by

$$(\delta \cdot m)(v) = \delta\Big(g \mapsto m(g^{-1}.v)\Big) \;,$$

\vskip8pt

for $m \in M$ and $\delta \in D(\GO,L)$.  For $g \in \GO$ the delta
distribution $\delta_g \in D(\GO,L)$ is defined by $\delta_g(f) = f(g)$. These delta distributions are invertible in $D(\GO,L)$, and the map $g \mapsto \delta_g$ is an injective group homomorphism from $\GO$ into
the group of units of $D(\GO,L)$.

We also recall that the category of coadmissible $D(\GO,L)$-modules is a full abelian subcategory of all
abstract $D(\GO,L)$-modules \cite[Thm. 5.1]{ST03} and, by construction, anti-equivalent to the category of admissible locally analytic $\GO$-representations.
\end{para}

\begin{para}\label{DgkO} {\it The distribution algebras $\DgkO$.} Recall the wide open congruence subgroup $\Gkc$ introduced in \ref{passing_to_completion} and its analytic distribution algebra $\Dgk = \cO(\Gkc)'_b$. Given a continuous representation $W$ of $\GO$, one can consider the subspace $W_{\Gkc-\rm an} \sub W$ of $\Gkc$-analytic vectors, cf. \cite[3.4.1]{EmertonA}. This applies to the action of $\GO$ on the space $C^{\rm cts}(\GO,L)$ of continuous $L$-valued functions given by the formula $(g.f)(x) = f(g^{-1}x)$. With this notation, one has a canonical isomorphism of topological $L$-vector spaces

\begin{numequation}\label{inductive_limit} \varinjlim_k C^{\rm cts}(\GO,L)_{\Gkc-\rm an}\car C^{\rm la}(\GO,L)
\end{numequation}

Following the notation introduced in \cite[proof of 5.3.1]{EmertonA} we denote by $\DgkO$ the strong dual of the space of $\Gkc$-analytic vectors of $C^{\rm cts}(\GO,L)$, i.e.,

$$\DgkO:= (C^{\rm cts}(\GO,L)_{\Gkc-\rm an})'_b \;.$$

\vskip8pt

The ring $\DgkO$ naturally contains $\Dgk$. Moreover, the delta distributions $\delta_g$, for $g$ in the normal subgroup $G_{k+1} := \Gkc(\fro) = \bbG(k+1)(\fro)$ of $\GO$, are contained in this subring too. One obtains a decomposition of $\DgkO$ as a $\Dgk$-module:

\begin{numequation}\label{equ-finitefree}\DgkO = \oplus_{g \in \GO/G_{k+1}} \Dgk \delta_g \;,
\end{numequation}

cf. \cite[proof of 5.3.1]{EmertonA}. This is a topological direct sum decomposition in the sense that the subspace topology of $\Dgk$ is equal to its topology as an $L$-algebra of compact type, and the topology on $\DgkO$ is equal to the product topology on the right of \ref{equ-finitefree}. Dualizing the isomorphism \ref{inductive_limit} then yields
an isomorphism of topological $L$ algebras

$$\DGO \car \varprojlim_k \DgkO \;.$$

\vskip8pt

This is the weak Fr\'echet-Stein structure on the locally analytic distribution algebra $\DGO$ as introduced by Emerton in \cite[Prop. 5.3.1]{EmertonA}. In an obviously similar manner we may define the ring $\DgkOt$ and obtain an isomorphism $\DGOt \car \varprojlim_k \DgkOt$.
\end{para}

\begin{para}\label{DGkO_sequences_on_modules} Let $V$ be again an admissible locally analytic representation of $\GO$, and $M = V'_b$ be as in \ref{module}. The subspace $V_{\Gkc-\rm an} \sub V$ is naturally a nuclear  Fr\'echet space \cite[6.1.6]{EmertonA}, and we let $M_k := (V_{\Gkc-\rm an})'_b$ be its strong dual. It is a space of compact type and a topological $\DgkO$-module which is finitely generated \cite[6.1.13]{EmertonA}. According to \cite[6.1.20]{EmertonA} the modules $M_k:= (V_{\Gkc-\rm an})'$ form a $(\DgkO)_{k\in\bbN}$-sequence, in the sense of \cite[1.2.8]{EmertonA}, for the coadmissible module $M$ relative to the weak Fr\'echet-Stein structure on $\DGO.$ This implies that one has

\begin{numequation}\label{equ-weakfamily}
M_k = \DgkO\hat{\otimes}_{\DGO} M
\end{numequation}

as $\DgkO$-modules for any $k$. Here, the completed tensor product is understood in the sense of \cite[Lem. 1.2.3]{EmertonA}.

\end{para}

\begin{lemma}\label{lem-refine} (i) The $\DgkO$-module $M_k$ is finitely presented.

\vskip8pt

(ii) There are natural isomorphisms

$$D(\bbG(k-1)^\circ,\GO)\otimes_{\DgkO} M_k\car M_{k-1} \;.$$

\vskip8pt

(iii) The natural map $\DgkO\otimes_{\DGO} M\car M_k$ is bijective.

\vskip8pt
\end{lemma}

\begin{proof} The points (i) and (ii) can be proved exactly as \cite[5.2.4]{PSS4}. For (iii) we consider the $\DgkO$-submodule generated inside $M_k$ by $M$. It clearly forms a dense subspace and is closed according to \cite[5.1.1 (ii)]{PSS4}. Hence the map in question is surjective. Moreover, this argument shows that the finitely generated $\DgkO$-module $M_k$ is generated by finitely many elements in the image of $M$. To prove injectivity of the map in question, we abbreviate $A:= \DGO$ and $A_k:= \DgkO$ and consider an element $b_1\otimes x_1+\ldots +b_s\otimes x_s \in A_k\otimes_A M$ such that $b_1x_1+\ldots +b_sx_s=0$ in $M_k$. Consider the homomorphism

$$(A^s_{k'})_{k'} \lra (M_{k'})_{k'}, (a_1, \ldots, a_s)\mapsto a_1x_1+ \ldots +a_sx_s$$

\vskip8pt

where $k'\geq k$. Let $N$ be the kernel of the corresponding map of coadmissible modules $A^s \ra M$. By the above surjectivity argument, there are finitely many elements $(c^{(1)}_1, \ldots, c^{(1)}_s), \ldots, (c^{(r)}_1,\ldots , c^{(r)}_s)$ in $N$ whose images generate the kernel of the map $A^s_{k} \lra M_{k}$ as an $A_k$-module. From here one may follow the argument in the proof of \cite[Cor. 3.1]{ST03} word for word. \end{proof}

\vskip8pt

{\bf Remark.} These results have obvious analogues when the character $\theta_0$ is involved.

\subsection{\texorpdfstring{$\GO$}{}-equivariance and the functor \texorpdfstring{$\Loc^{\GO}$}{}}\label{subsec_G0}

\begin{para}\label{gp_actions_blow_ups} {\it Group actions on blow-ups.} We recall that it is our convention that the group scheme $\bbG_0$ acts on the right on $X_0 = \bbB_0 \bksl \bbG_0$, cf. \ref{group_actions}. This yields a right action of the group $\GO$ on $X_0$, and we denote the automorphism of $X_0$ given by $g \in \GO$ by $\rho_g$, i.e., $\rho_g: X_0 \ra X_0$. As the action of $\GO$ on $X_0$ is on the right, we have $\rho_g \circ \rho_h = \rho_{hg}$ for all $g,h \in \GO$. We also denote by $\rho_g^\sharp: \cO_{X_0} \ra (\rho_g)_* \cO_{X_0}$ the comorphism of $\rho_g$. We then have

\begin{numequation}\label{composition}
(\rho_g)_*(\rho_h^\sharp) \circ \rho_g^\sharp = \rho_{hg}^\sharp \;.
\end{numequation}

Now let $H \sub \GO$ be an open subgroup. We say that an open ideal sheaf $\cI \sub \cO_{X_0}$ is {\it $H$-stable} if for all $g \in H$ the comorphism $\rho_g^\sharp$ maps $\cI \sub \cO_{X_0}$ into $(\rho_g)_*\cI \sub (\rho_g)_* \cO_{X_0}$. In that case $\rho_g^\sharp$ induces a morphism of sheaves of graded rings

$$\bigoplus_{d \ge 0} \cI^d \lra (\rho_g)_*\Big(\bigoplus_{d \ge 0} \cI^d\Big)$$

\vskip8pt

on $X_0$. This morphisms of sheaves in turn induces an automorphism of the blow-up $X = {\rm {\bf Proj}}\Big(\bigoplus_{d \ge 0} \cI^d\Big)$, and the action of $H$ on $X_0$ lifts thus to an action  of $H$ on $X$, which we again denote by $\rho$ for ease of notation.

\vskip8pt

The same considerations apply when we pass to the formal completion $\frX_0$ of $X_0$, in which case we denote the morphism $\frX_0 \ra \frX_0$ induced by $\rho_g$ also by $\rho_g$, for ease of notation.
If now $\frI$ is an open ideal sheaf on $\frX_0$ which is $H$-stable, and if $\frX$ is the formal blow-up of $\frI$, we also say that $\frX$ is {\it $H$-equivariant}. There is at most one way to lift the action of $H$ on $X_0$ (resp. $\frX_0$) to $X$ (resp. $\frX$), because the blow-up morphism induces an isomorphism between the generic fibers $X_\eta \car X_{0,\eta}$ (resp. rigid spaces $\frX^{\rm rig} \car \frX_0^{\rm rig}$), and the group action on the generic fiber (resp. associated rigid space), is thus pre-determined, and in turn determines the action on $X$ (resp. $\frX$) uniquely.
\end{para}

\begin{lemma}\label{Gk_acts_trivially} Let $\pr: \frX \ra \frX_0$ be an admissible blow-up, and assume $k \ge k_\frX$. Then $\frX$ is $G_k = \bbG(k)(\fro)$-equivariant and the induced action of every $g \in G_{k+1}$ on the special fiber of $\frX$ is the identity. Therefore, $G_{k+1}$ acts trivially on the topological space underlying $\frX$.
\end{lemma}

\begin{proof} Consider the action $\mu: X_0 \times_{\Spec(\fro)} \bbG_0 \ra X_0$ of $\bbG_0$ on $X_0$. If $g: \Spec(\fro) \ra \bbG_0$ is in $G_1$, then the induced map on the mod-$\vpi$-fibers $g_s: \Spec(\Fq) \ra \bbG_0  \times_{\Spec(\fro)} \Spec(\Fq)$ is the identity element in $\bbG_0(\Fq)$. Because $\rho_g$ is defined in terms of $\mu$, and since $\mu$ is compatible with base change $\Spec(\Fq) \ra \Spec(\fro)$, it follows that all elements $g \in G_1$ act trivially on the special fiber of $X_0$. In particular, the morphism $\rho_g: \frX_0 \ra \frX_0$ is the identity map on the topological space underlying $\frX_0$ if $g \in G_1$. This takes care of the case when $k=0$ (hence $k_\frX = 0$, and thus $\frX = \frX_0$). We therefore assume in the following $k \ge 1$.

\vskip8pt

For the purpose of this proof we let $\frG$ be the completion of $\bbG_0$ along its special fiber (this formal group scheme is denoted by $\frG(0)$ in \ref{completion}). The quotient morphism $\sigma: \bbG_0 \ra X_0$ induces a quotient morphism $\sigma^\wedge: \frG \ra \frX_0$ of the corresponding formal schemes. Moreover, the right multiplication of $g \in G_0$ on $\bbG_0$ induces a right multiplication $\widetilde{\rho}_g: \frG \ra \frG$, such that the following diagram is commutative

\begin{numequation}\label{action_on_quotient}
\xymatrix{
\frG \ar[r]^{\widetilde{\rho}_g} \ar[d]^{\sigma^\wedge} & \frG \ar[d]^{\sigma^\wedge}\\
\frX_0 \ar[r]^{\rho_g} & \frX_0 \\
}
\end{numequation}

If $g \in G_1$, then, as we remarked above, the map underlying the morphism $\rho_g$ is the identity map on $\frX_0$, and, for the same reason, the map underlying the morphism $\widetilde{\rho}_g$ is the identity map on $\frG$. It follows from the very definition of $G_k$ that for $g \in G_k$, for all open subsets $U \sub \frG$, and for all $f \in \cO_\frG(U)$ one has $(\widetilde{\rho}_g)^\sharp_U(f) \equiv f \mbox{ mod } (\vpi^k)$. If now $V \sub \frX_0$ is an open subset and $U := (\sigma^\wedge)^{-1}(V)$, then \ref{action_on_quotient} gives rise to a commutative diagram

\[\xymatrixcolsep{5pc}
\xymatrix{
\cO_\frG(U) & \ar[l]^{(\widetilde{\rho}_g)^\sharp_U} \cO_\frG(U) \\
\cO_{\frX_0}(V) \ar[u]^{(\sigma^\wedge)^\sharp_V} & \ar[l]^{(\rho_g)^\sharp_V} \cO_{\frX_0}(V) \ar[u]^{(\sigma^\wedge)^\sharp_V}\\
}\]

As $U \ra V$ is a locally trivial fiber bundle, the ring homomorphism $(\sigma^\wedge)^\sharp_V$ is injective \cite[I.5.7 (1)]{Jantzen} and identifies $\cO_{\frX_0}(V)$ with the subring of $\frB$-invariants
of $\cO_\frG(U)$ where $\frB$ denotes the completion of $\bbB_0$
along its special fiber \cite[ I.5.8 (2)]{Jantzen}. In the following we will therefore suppress the notation $(\sigma^\wedge)^\sharp_V$ and view this homomorphism as an inclusion. By the above discussion, we then have for all $f \in \cO_{\frX_0}(V)$ that

$$(\rho_g)^{\sharp}_V (f) - f  = \varpi^k \tilde{f}$$

\vskip8pt

with some $\tilde{f}\in \cO_{\frG}(U)$. Now $\tilde{f}$ is $\frB$-invariant: indeed, $\varpi^k \tilde{f}$ is $\frB$-invariant, and so we have

$$\Delta( \varpi^k \tilde{f}) -  \varpi^k \tilde{f} \otimes 1  = 0$$

\vskip8pt

in $\cO_{\frG}(U) \otimes_{\fro} \cO_{\frB}(\frB)$ where $\Delta$ denotes the comodule map of the $\frB$-module
$\cO_{\frG}(U)$ \cite[I.2.10 (2)]{Jantzen}. Since $\Delta$ is $\fro$-linear and $\cO_{\frG}(U) \otimes_{\fro} \cO_{\frB}(\frB)$ is $\fro$-torsionfree, this implies $\Delta(\tilde{f}) - \tilde{f} \otimes 1  = 0$, as claimed. Since $\tilde{f}$ is $\frB$-invariant, we may conclude that $(\rho_g)^\sharp_V(f) \equiv f \mbox{ mod } (\vpi^k)$ for all $f \in \cO_{\frX_0}(V)$. Now suppose $\frI \sub \cO_{\frX_0}$ is an open ideal sheaf, and assume $\vpi^k \in \frI$ and $g \in G_k \sub G_1$.
Then, for any open subset $V \sub \frX_0$, and any $f \in \frI(V)$ we have $(\rho_g)^\sharp_V(f) = f + \vpi^k\tilde{f}$ for some $\tilde{f} \in \cO_{\frX_0}(V)$. Since $\vpi^k \tilde{f} \in \frI(V)$, we find that $(\rho_g)^\sharp_V$ maps $\frI$ into itself, and the blow-up $\frX$ of $\frI$ is $G_k$-equivariant.

\vskip8pt

If now $g$ is in $G_{k+1}$ we even have $(\rho_g)^\sharp_V(f) = f + \vpi^{k+1}\tilde{f}$ for some $\tilde{f} \in \cO_{\frX_0}(V)$. And since $\vpi^k \in \frI$ we conclude that $(\rho_g)^\sharp_V(f) \equiv f \mbox{ mod } \vpi\frI$. This implies that the morphism induced by $(\rho_g)^\sharp$ on the sheaf $\Big(\bigoplus_{d \ge 0} \frI^d\Big) \otimes_\fro \fro/(\vpi)$, which is a sheaf on the special fiber of $\frX_0$, is the identity. And $\mbox{{\bf Proj}}\Big(\Big(\bigoplus_{d \ge 0} \frI^d\Big) \otimes_\fro \fro/(\vpi)\Big)$ is the special fiber of the formal blow-up $\frX$ of $\frI$. \end{proof}

\vskip8pt

\begin{para}\label{twoproperties} For the rest of this section we let $H \sub \GO$ be an open subgroup. If $\frX \ra \frX_0$ is an $H$-equivariant admissible blow-up with lifted action $\rho$, then there is an induced action of $H$ on the sheaf $\sD^\dagger_{\frX,k}$

\begin{numequation}\label{equ-ringiso0}
\Ad(g): \sD^\dagger_{\frX,k} \car (\rho_g)_*\sD^\dagger_{\frX,k} \;, \;\; P \mapsto \rho_g^\sharp P (\rho_g^\sharp)^{-1} \;,
\end{numequation}

for all $k \ge k_\frX$. This is an action on the left in the sense that

$$(\rho_g)_*(\Ad(h)) \circ \Ad(g) = \Ad(hg) \;,$$

\vskip8pt

as follows from \ref{composition}. Furthermore, the group $G_{k+1}$ is contained in $\Dgk$ as a set of delta distributions, and for $g \in G_{k+1}$ we also write $\delta_g$ for its image in $H^0(\frX, \sD^\dagger_{\frX,k}) = \Dgk_{\theta_0}$, cf. \ref{global_sec_tsD}.
\end{para}

\begin{dfn}\label{dfn-stequiv} Let $\frX$ be an $H$-equivariant admissible blow-up of $\frX_0$. A {\it strongly $H$-equivariant
$\sD^\dagger_{\frX,k}$-module} is a $\sD^\dagger_{\frX,k}$-module $\sM$ together with a family $(\phi_g)_{g \in H}$ of isomorphisms

$$\phi_g: \sM \lra (\rho_g)_* \sM$$

\vskip8pt

of sheaves of $L$-vector spaces, satisfying the following conditions:

\vskip8pt

\begin{enumerate}
\item For all $g,h \in H$ we have $(\rho_g)_*(\phi_h) \circ \phi_g =\phi_{hg}$.

\vskip8pt

\item For all open subsets $U \sub \frX$, all $P \in \sD^\dagger_{\frX,k}(U)$, and all $m \in \sM(U)$ one has $\phi_g(P.m) = \Ad(g)(P).\phi_g(m)$.

\vskip8pt

\item\footnote{To make sense of this condition, we use that elements $g \in G_{k+1}$ act trivially on the topological space underlying $\frX$, cf. \ref{Gk_acts_trivially}.} For all $g \in H \cap G_{k+1}$ the map $\phi_g: \sM \ra (\rho_g)_* \sM  = \sM$ is equal to the multiplication by $\delta_g \in H^0(\frX, \sD^\dagger_{\frX,k})$.
\end{enumerate}

\vskip8pt

A morphism between two strongly $H$-equivariant $\sD^\dagger_{\frX,k}$-modules $(\sM, (\phi^\sM_g)_{g \in H})$ and \linebreak $(\sN,(\phi^\sN_g)_{g \in H})$ is a $\sD^\dagger_{\frX,k}$-linear morphism $\psi: \sM \ra \sN$ such that $$\phi^\sN_g \circ \psi = (\rho_g)_*(\psi) \circ \phi^\sM_g$$ for all $g \in H$. We denote the category  of strongly $H$-equivariant $\sD^\dagger_{\frX,k}$-modules which are, moreover, coherent as $\sD^\dagger_{\frX,k}$-modules by ${\rm Coh}(\sD^\dagger_{\frX,k},H)$.
\end{dfn}

\vskip8pt

{\bf Remarks.} 'Strongly equivariant' refers to the additional condition that the action coincides with multiplication by $\delta_g$ if $g \in H \cap G_{k+1}$. This is the analogue of \cite[Prop. 2.6]{VandenBerghED} in our situation. We also note that any $\sD^\dagger_{\frX,k}$-module is strongly $G_{k+1}$-equivariant for the natural $G_{k+1}$-action, cf. \ref{Gk_acts_trivially}. The following result could be stated in greater generality for $H$-equivariant blow-ups $\frX \ra \frX_0$ if we had introduced the ring $D(\bbG(k)^\circ,H)$ also for open subgroups $H \sub \GO$ (containing $G_{k+1}$) instead of just $\GO$.

\vskip8pt

\begin{thm} Let $\frX \ra \frX_0$ be a $\GO$-equivariant admissible blow-up, and let $k \ge k_\frX$. The functors $\Loc^\dagger_{\frX,k}$ and $H^0$ induce quasi-inverse equivalences between the category of finitely presented $D(\bbG(k)^\circ,G_0)_{\theta_0}$-modules and ${\rm Coh}(\sD^\dagger_{\frX,k},\GO)$.
\end{thm}

\begin{proof} This follows from \ref{thm-equivalence}, \ref{global_sec_tsD}, the definition of ${\rm Coh}(\sD^\dagger_{\frX,k},\GO)$, and the description of $D(\bbG(k)^\circ,G_0)$ in \ref{equ-finitefree}. \end{proof}

\begin{para} Suppose now that $\pi:\frX' \ra \frX$ is a $\GO$-equivariant morphism over $\frX_0$ between admissible formal $\GO$-equivariant blow-ups of $\frX_0$ (whose lifted actions we denote by $\rho^{\frX'}$ and $\rho^{\frX}$ respectively), and that $k \ge k_\frX$ and $k' \geq  \max\{k_{\frX'},k\}$. According to \ref{prop-exactdirectimage} there is then a morphism of sheaves of rings

\begin{numequation}\label{equ-transit_sheaf}
\Psi: \pi_*\sD^\dagger_{\frX',k'} = \sD^\dagger_{\frX,k'} \hra \sD^\dagger_{\frX,k}
\end{numequation}
which is $\GO$-equivariant, i.e. satisfying

$${\rm Ad}(g) \circ \Psi = (\rho^{\frX}_g)_*(\Psi) \circ \pi_*({\rm Ad}(g))$$

\vskip8pt

for all $g\in\GO$. Suppose we are given two modules $\sM_{\frX'} \in {\rm Coh}(\sD^\dagger_{\frX',k'},\GO)$ and $\sM_{\frX}\in {\rm Coh}(\sD^\dagger_{\frX,k},\GO)$ together with a morphism

$$\psi: \pi_*\sM_{\frX'} \lra \sM_{\frX}$$

\vskip8pt

linear relative to (\ref{equ-transit_sheaf}) and which is $\GO$-equivariant, i.e. satisfying

$$\phi^{\sM_\frX}_g \circ \psi = (\rho^{\frX}_g)_*(\psi) \circ \pi_*(\phi^{\sM_{\frX'}}_g)$$

\vskip8pt

for all $g\in G_0$. We obtain thus a morphism

$$\sD^\dagger_{\frX,k} \otimes_{\pi_* \sD^\dagger_{\frX',k'}} \pi_*\sM_{\frX'} \lra \sM_\frX$$

\vskip8pt

of $\sD^\dagger_{\frX,k}$-modules. Denote, by $\sK$ the submodule of $\sD^\dagger_{\frX,k} \otimes_{\pi_* \sD^\dagger_{\frX',k'}} \pi_*\sM_{\frX'}$ locally generated by all elements of the form $P \delta_h \otimes m - P \otimes (h.m)$, where $h\in G_{k+1}$, $m$ is a local section of $\pi_* \sM_{\frX'}$, and $P$ is a local section of $\sD^\dagger_{\frX,k}$. For convenience we will abbreviate the quotient $(\sD^\dagger_{\frX,k} \otimes_{\pi_*\sD^\dagger_{\frX',k'}}  \pi_*\sM_{\frX'}) / \sK$ by

$$\sD^\dagger_{\frX,k} \otimes_{\pi_*\sD^\dagger_{\frX',k'},G_{k+1}}  \pi_*\sM_{\frX'} \;.$$

\vskip8pt

Now since the target of the preceding morphism is strongly equivariant, the morphism will factor through this quotient and we thus obtain a morphism of $\sD^\dagger_{\frX,k}$-modules

\begin{numequation}\label{factormap} \overline{\psi}: \sD^\dagger_{\frX,k} \otimes_{\pi_*\sD^\dagger_{\frX',k'},G_{k+1}}  \pi_*\sM_{\frX'} \lra \sM_{\frX} \;.
\end{numequation}

The domain of this morphism lies in ${\rm Coh}(\sD^\dagger_{\frX,k},\GO)$ when we equip it with the action defined on simple tensors by

$$g.( P \otimes m) = \Ad(g)(P) \otimes (g.m) \;,$$

\vskip8pt

for $g \in \GO$, where $P$ and $m$ are local sections of $\sD^\dagger_{\frX,k}$ and $\pi_* \sM_{\frX'}$, respectively. Since (\ref{equ-transit_sheaf}) is $G_0$-equivariant, the map (\ref{factormap}) is then seen to be in fact a morphism in ${\rm Coh}(\sD^\dagger_{\frX,k},\GO)$. The question, in which situations this morphism will actually be an isomorphism will be crucial in the definition of a coadmissible $G_0$-equivariant arithmetic $\sD$-module, cf. \ref{dfn-coadmod0} below.

\vskip8pt

We finish this paragraph by an auxiliary result which will be used in the proof of thm. \ref{prop-equivalenceII}.
\end{para}

\begin{lemma} \label{lem-aux}
Let $\frX',\frX\in \cF_{\frX_0}$ be $G_0$-equivariant. Suppose $(\frX',k')\succeq (\frX,k)$ with canonical morphism $\pi:\frX'\ra\frX$ over $\frX_0$ and let $M$ be a coherent
$D(\bbG(k')^\circ,G_0)_{\theta_0}$-module with localization $\sM=\Loc^\dagger_{\frX',k'}(M)\in {\rm Coh}(\sD^\dagger_{\frX',k'},G_0)$. Then there is a canonical isomorphism in ${\rm Coh}(\sD^\dagger_{\frX,k},G_0)$ given by
$$\sD^\dagger_{\frX,k} \otimes_{\pi_* \sD^\dagger_{\frX',k'},G_{k+1}} \pi_*\sM \car \Loc^\dagger_{\frX,k}(D(\bbG(k)^\circ,\GO)\otimes_{D(\bbG(k')^\circ,\GO)} M). $$
\end{lemma}

\begin{proof} We denote a system of representatives in $G_{k+1}$ for the cosets in $G_{k+1}/G_{k'+1}$ by $\tH$. For simplicity, we abbreviate

$$D(k):=\Dgkt \hskip10pt {\rm and}\hskip10pt D(k,\GO):=\DgkOt$$

\vskip8pt

and similarly for $k'$. We have the natural inclusion $D(k)\hookrightarrow D(k,\GO)$ from \ref{equ-finitefree} which is compatible with variation in $k$. Now suppose $M$ is a $D(k',\GO)$-module. We then have the free $D(k)$-module $D(k)^{\oplus M\times \tH}$
on a basis $e_{m,h}$ indexed by the elements $(m,h)$ of the set $M\times R$. Its formation is functorial in $M$: if $M'$ is another module and $f: M \ra M'$ a linear map, then $e_{m,h} \ra e_{f(m),h}$ induces a linear map between the corresponding free modules. The free module comes with a linear map

$$f_M: D(k)^{\oplus M\times \tH} \ra D(k)\otimes_{D(k')} M$$

\vskip8pt

given by

$$\oplus_{(m,h)} \lambda_{m,h}e_{m,h} \mapsto (\lambda_{m,h} \delta_h) \otimes m - \lambda_{m,h} \otimes (\delta_h \cdot m)$$

\vskip8pt

for $\lambda_{m,h}\in D(k)$ where we consider $M$ a $D(k')$-module via the inclusion $D(k')\hookrightarrow D(k',\GO)$. Note that, since $M$ is a $D(k',\GO)$-module, and because $\GO$ is contained in $D(k',\GO)$, the expression $\delta_h \cdot m$ is defined for any particular $h \in G_{k+1}$. The linear map is visibly functorial in $M$ and gives rise to the sequence of linear maps

$$ D(k)^{\oplus M\times \tH}\stackrel{f_M}{ \lra} D(k)\otimes_{D(k')} M \stackrel{can_M}{ \lra}
D(k,\GO)\otimes_{D(k',\GO)} M \lra 0$$

\vskip8pt

where the second map is induced from the inclusion $D(k')\hookrightarrow D(k',\GO)$.
The sequence is functorial in $M$, since so are both occuring maps.

\vskip8pt

{\it Claim 1: If $M$ is a finitely presented $D(k',\GO)$-module, then the above sequence is exact.}

\vskip8pt

\begin{proof} This can be proved as in the proof of \cite[Prop. 5.3.5]{PSS4}. \end{proof}

\vskip8pt

{\it Claim 2: Suppose $M$ is a finitely presented $D(k')$-module and let $\sM:=  \Loc^\dagger_{\frX',k'}(M)$. The natural morphism

$$\Loc^\dagger_{\frX,k}(D(k)\otimes_{D(k')} M)\car \sD^\dagger_{\frX,k} \otimes_{\pi_*\sD^\dagger_{\frX',k'}} \pi_*\sM $$

\vskip8pt

is bijective.}

\vskip8pt

\begin{proof} The functor $\pi_*$ is exact on coherent $\sD^\dagger_{\frX',k'}$-modules according to \ref{prop-exactdirectimage}.
Choosing a finite presentation of $M$ reduces to the case $M=D(k')$ which is obvious. \end{proof}

\vskip8pt

Now let $M$ be a finitely presented $D(k',\GO)$-module. Let $m_1, \ldots ,m_r$ be generators for $M$ as a $D(k')$-module. We have a sequence of $D(k)$-modules

$$\bigoplus_{i,h} D(k)e_{m_i,h}\stackrel{f'_M}{ \lra} D(k)\otimes_{D(k')} M \stackrel{can_M}{ \lra}
D(k,\GO)\otimes_{D(k',\GO)} M \lra 0$$

\vskip8pt

where $f'_M$ denotes the restriction of the map $f_M$ to the free submodule of $D(k)^{\oplus M\times \tH}$
generated by the finitely many vectors $e_{m_i,h}, i=1,\ldots,r$, $h \in \tH$. Since ${\rm im}(f'_M)={\rm im}(f_M)$ the sequence is exact by the first claim. Since it consists of finitely presented $D(k)$-modules, we may apply the exact functor $\Loc^\dagger_{\frX,k}$ to it. By the second claim, we get an exact sequence

$$ (\sD^\dagger_{\frX,k})^{\oplus r|\tH|} \ra \sD^\dagger_{\frX,k} \otimes_{\pi_*\sD^\dagger_{\frX',k'}} \pi_*\sM \ra
\Loc^\dagger_{\frX,k}(D(k,\GO)\otimes_{D(k',\GO)} M) \ra 0$$

\vskip8pt

where $\sM = \Loc^\dagger_{\frX',k'}(M)$. The cokernel of the first map in this sequence equals by definition

$$\sD^\dagger_{\frX,k} \otimes_{\pi_* \sD^\dagger_{\frX',k'},G_{k+1}} \pi_*\sM \;,$$

\vskip8pt

whence an isomorphism
$$\sD^\dagger_{\frX,k} \otimes_{\pi_* \sD^\dagger_{\frX',k'},G_{k+1}} \pi_*\sM \car \Loc^\dagger_{\frX,k}(D(k,\GO)\otimes_{D(k',\GO)} M).$$ This proves the lemma. \end{proof}

\begin{para}\label{intro-projlim} The purpose of the rest of this section is to first explain how to form $\GO$-equivariant compatible systems of coherent $\sD^\dagger_{\frX,k}$-modules when the formal models $\frX$ and the congruence levels $k$ vary in a suitable family. Here we will only be considering formal models of the rigid-analytic flag variety which are admissible formal blow-ups of $\frX_0$. In a second step, we will relate such $G_0$-equivariant systems to coadmissible $D(G_0,L)_{\theta_0}$-modules thus establishing a version of the classical {\it localization theorem} for equivariant algebraic $D$-modules \cite{BB81} in our setting. In sec. \ref{G_equivariance} these constructions will be generalized to the setting of $G$-equivariant compatible systems.

\vskip8pt

We recall that we denote by $\bbX = \bbB \bksl \bbG$ the flag variety of $\bbG$, and by $\bbX^\rig$ the rigid-analytic space associated by the GAGA functor to $\bbX$, cf. \ref{models_gp_actions}. Furthermore, we denote by $\frX_\infty$ the projective limit of all formal models of $\bbX^\rig$ (in the sense of \ref{models_gp_actions}). This space is known to be homeomorphic to the adic space corresponding to $\bbX^\rig$, cf. \cite[Thm. 4 in sec. 2, Thm. 4 in sec. 3]{SchnVPut} where this space is denoted by ${\rm Val}(\bbX^\rig)$.

\vskip8pt

Consider the set $\cF_{\frX_0}$ of admissible formal blow-ups $\frX \ra \frX_0$. This set is ordered by $\frX' \succeq \frX$ if the blow-up morphism $\pi: \frX' \ra \frX_0$ factors as the composition of a morphism $\frX' \ra \frX$ and the blow-up morphism $\frX \ra \frX_0$. In fact, the morphism $\frX' \ra \frX$ is then necessarily unique \cite[II, 7.14]{HartshorneA}, and is itself a blow-up morphism \cite[ch. 8, 1.24]{LiuBook}. By \cite[Remark 10 in sec. 8.2]{BoschLectures} the set $\cF_{\frX_0}$ is directed in the sense that any two elements have a common upper bound,
and it is cofinal in the set of all formal models. In particular, $\frX_\infty=\varprojlim_{\cF_{\frX_0}} \frX$.
\end{para}

\vskip8pt

\begin{prop}\label{prop-cofinal} Any formal model $\frX$ of $\bbX^{\rm rig}$ is dominated by one which is a $\GO$-equivariant admissible blow-up of $\frX_0$.
\end{prop}

\begin{proof}  By \cite[Remark 10 in sec. 8.2]{BoschLectures} we may assume that $\frX$ is already an admissible blow-up of $\frX_0$. Let $\cI$ be the ideal which is blown up to obtain $\frX$. If $\vpi^k \in \cI$ for some $k \ge 1$, then $G_k$ acts trivially on the topological space underlying $\frX_0$ and stabilizes $\cI$ in the sense that $\rho_g^\sharp: \cO_{\frX_0} \ra \cO_{\frX_0}$ maps $\cI$ into $\cI$ for all $g \in G_k$. Let $1 = g_1, \ldots, g_N$ be a system of representatives for $\GO/G_k$ and let $\cJ$ be the product of the finitely many ideals $\rho^\sharp_{g_i}(\cI)$. Then $\cJ$ is $\GO$-stable and contains $\cI$. Blowing up $\cJ$ on $\frX_0$ yields a $\GO$-stable formal scheme $\frX'$, and $\frX'$ is also the admissible formal blow-up of the sheaf $\pr^{-1}\cJ \cdot \cO_\frX$ on $\frX$, and the blow-up morphism $\frX' \ra \frX_0$ factors as the composition of the blow-up morphisms $\frX' \ra \frX \ra \frX_0$. \end{proof}

\vskip8pt

\begin{dfn}\label{partial_ordering} We denote the set of pairs $(\frX,k)$, where $\frX \in \cF_{\frX_0}$ and $k \in \bbN$ satisfies $k \ge k_\frX$,
by $\underline{\cF}_{\frX_0}$. This set is ordered by $(\frX',k') \succeq (\frX,k)$ if and only if $\frX' \succeq \frX$ and $k' \ge k$.
\end{dfn}

\vskip8pt

Since $\cF_{\frX_0}$ is directed, the set $\underline{\cF}_{\frX_0}$ is directed, too.

\begin{lemma}\label{translated_model} Let $\frI$ be an open ideal sheaf on $\frX_0$, and let $g \in \GO$. Then $\frK = (\rho_g^\sharp)^{-1}((\rho_g)_*(\frI))$ is again an open ideal sheaf on $\frX_0$. Let $\frX$ be the blow-up of $\frI$, and let $\frX.g$ be the blow-up of $\frK$. Then there is a morphism $\rho_g: \frX \ra \frX.g$ such that the following diagram is commutative (where the vertical maps are the blow-up morphisms):

$$\xymatrix{
\frX \ar[r]^{\rho_g} \ar[d] & \frX.g \ar[d]\\
\frX_0 \ar[r]^{\rho_g} & \frX_0 \\
}$$

\vskip8pt

We have $k_{\frX.g} = k_\frX$. Moreover, for any two elements $g,h \in G_0$ we have a canonical isomorphism $(\frX.g).h \simeq \frX.(gh)$, and the morphism $\frX \stackrel{\rho_g}{\lra} \frX.g \stackrel{\rho_h}{\lra} (\frX.g).h \simeq \frX.(gh)$ is equal to $\rho_{gh}$. This gives a right action of the group $G_0$ on the family $\cF_{\frX_0}$.
\end{lemma}

\begin{proof} It is easy to check that $\frK$ is indeed an open ideal sheaf. Moreover, the comorphism $\rho_g^\sharp: \cO_{\frX_0} \ra
(\rho_g)_* \cO_{\frX_0}$ induces a morphism

\begin{numequation}\label{translation_morphism}\bigoplus_{d \ge 0} \frK^d \lra (\rho_g)_*\left(\bigoplus_{d \ge 0} \frI^d\right)
\end{numequation}

of sheaves of graded rings which is linear with respect to $\rho_g^\sharp$ and which coincides with $\rho_g^\sharp$ in degree zero. The morphism of sheaves \ref{translation_morphism} induces the morphism between the blow-ups $\frX$ and $\frX.g$. That \ref{translation_morphism} is linear with respect to $\rho_g^\sharp$ implies the existence of the commutative diagram. The assertion about the congruence levels follows straightforwardly from the definition \ref{defkX}. The remaining assertions follow directly from the construction. \end{proof}

\vskip8pt

\begin{cor}\label{cor-actionpi} Assume that $(\frX',k') \succeq (\frX,k)$ for $\frX,\frX'\in\cF_{\frX_0}$ and let $\pi: \frX' \ra \frX$ be the unique morphism over $\frX_0$. Let $g \in \GO$. Then $(\frX'.g,k') \succeq (\frX.g,k)$ and if we denote the unique morphism $\frX'.g \ra \frX.g$ over $\frX_0$ by $\pi.g$, then the diagram

$$\xymatrix{
\frX' \ar[r]^{\rho_g} \ar[d]^{\pi} & \frX'.g \ar[d]^{\pi.g}\\
\frX \ar[r]^{\rho_g} & \frX.g \\
}$$

is commutative.
\end{cor}

\begin{proof} Follows easily from the preceding lemma. \end{proof}

\vskip8pt

\begin{dfn}\label{dfn-coadmod0} A {\it coadmissible $\GO$-equivariant arithmetic $\sD$-module} on $\cF_{\frX_0}$ consists of a family $\sM := (\sM_{\frX,k})$ of coherent $\sD^\dagger_{\frX,k}$-modules $\sM_{\frX,k}$, for all $(\frX,k) \in \underline{\cF}_{\frX_0}$, with the following properties:

\vskip8pt

{\rm (a)} For any $g \in G_0$ with morphism $\rho_g: \frX \ra \frX.g$ (cf. \ref{translated_model}), there exists an isomorphism

$$\phi_g: \sM_{\frX.g,k} \lra (\rho_g)_*\sM_{\frX,k}$$

\vskip8pt

of sheaves of $L$-vector spaces, satisfying the following conditions:

\vskip8pt

\begin{enumerate}
\item For all $g,h \in G_0$ we have $(\rho_g)_*(\phi_h) \circ \phi_g = \phi_{hg}$.

\vskip8pt

\item For all open subsets $U \sub \frX.g$, all $P \in \sD^\dagger_{\frX.g,k}(U)$, and all $m \in \sM_{\frX.g,k}(U)$ one has $\phi_g(P.m) = \Ad(g)(P).\phi_g(m)$.

\vskip8pt

\item\footnote{To make sense of this condition, we use that for $k \ge k_\frX$ the action of $G_{k+1}$ on $\frX_0$ lifts to $\frX$, cf. \ref{Gk_acts_trivially}. In this case $\frX.g = \frX$, and elements $g \in G_{k+1}$ act trivially on the topological space underlying $\frX$.} For all $g \in G_{k+1}$ the map $\phi_g: \sM_{\frX,k} \ra (\rho_g)_* \sM_{\frX,k}  = \sM_{\frX,k}$ is equal to the multiplication by $\delta_g \in H^0(\frX,\sD^\dagger_{\frX,k})$.
\end{enumerate}

\vskip5pt

{\rm (b)} Suppose $\frX', \frX \in \cF_{\frX_0}$ are both $\GO$-equivariant, and assume further that $(\frX',k') \succeq (\frX,k)$, and that $\pi: \frX' \ra \frX$ is the unique morphism over $\frX_0$. In this situation we require the existence of a transition morphism $\psi_{\frX',\frX}: \pi_*\sM_{\frX',k'} \ra \sM_{\frX,k}$, linear relative to the canonical morphism $\Psi: \pi_*\sD^\dagger_{\frX',k'} \ra \sD^\dagger_{\frX,k}$ (\ref{equ-transit_sheaf}) and satisfying

\begin{numequation}\label{compatible(a)(b)1} \phi^{\sM_{\frX,k}}_g \circ \psi_{\frX', \frX} = (\rho_g^{\frX})_*(\psi_{\frX',\frX}) \circ \pi_*(\phi^{\sM_{\frX',k'}}_g)
\end{numequation}

for any $g \in \GO$ (note that $\pi_*\circ(\rho_g^{\frX'})_*=(\rho_g^{\frX})_*\circ\pi_*$ according to cor. \ref{cor-actionpi} and so the composition of maps on the right-hand side makes sense). The morphism induced by $\psi_{\frX',\frX}$, cf \ref{factormap},

\begin{numequation}\label{isocondition1}
\overline{\psi}_{\frX', \frX}: \sD^\dagger_{\frX,k} \otimes_{\pi_*\sD^\dagger_{\frX',k'},G_{k+1}}  \pi_*\sM_{\frX'} \car \sM_\frX
\end{numequation}

is required to be an isomorphism of $\sD^\dagger_{\frX,k}$-modules. Additionally, the morphisms $\psi_{\frX', \frX}: \pi_*\sM_{\frX',k'} \ra \sM_{\frX,k}$ are required to satisfy the transitivity condition $\psi_{\frX',\frX} \circ \pi_*(\psi_{\frX'',\frX'}) = \psi_{\frX'',\frX}$, whenever $(\frX'',k'') \succeq (\frX',k') \succeq (\frX,k)$ in $\underline{\cF}_{\frX_0}$. Moreover, $\psi_{\frX,\frX} = {\rm id}_{\sM_{\frX,k}}$.

\vskip8pt

A morphism $\sM \ra \sN$ between two such modules consists of morphisms $\sM_{\frX,k} \ra \sN_{\frX,k}$ of $\sD^\dagger_{\frX,k}$-modules compatible with the extra structures. We denote the resulting category by $\sC^{\GO}_{\frX_0}$.
\end{dfn}

\vskip8pt

\begin{para} We now build the bridge to the category of coadmissible $\DGOt$-modules, cf. \ref{module}. Given such a module $M$ we have its associated admissible locally analytic $\GO$-representation $V=M'_b$ together with its subspace of $\Gkc$-analytic vectors $V_{\Gkc-\rm an}\subset V$. The latter is stable under the $\GO$-action and its dual $M_k:= (V_{\Gkc-\rm an})'$ is a finitely presented $\DgkOt$-module, cf. \ref{lem-refine}. In this situation thm. \ref{thm-equivalence} produces a coherent $\sD^\dagger_{\frX,k}$-module

$$\Loc^\dagger_{\frX,k}(M_k)= \sD^\dagger_{\frX,k} \otimes_{\Dgkt} M_k$$

\vskip8pt

for any element $(\frX,k)$ in $\underline{\cF}_{\frX_0}$. On the other hand, let $\sM$ be an arbitrary coadmissible $\GO$-equivariant arithmetic $\sD$-module on $\cF_{\frX_0}$. The transition morphisms $\psi_{\frX',\frX}: \pi_*\sM_{\frX',k'} \ra \sM_{\frX,k}$ induce maps $H^0(\frX',\sM_{\frX',k'}) \ra H^0(\frX,\sM_{\frX,k})$ on global sections. We let

$$ \Gamma(\sM):=\varprojlim_{(\frX,k) \in \underline{\cF}_{\frX_0}} H^0(\frX,\sM_{\frX,k})$$

where the projective limit is taken in the sense of abelian groups and over the cofinal subfamily, cf. prop \ref{prop-cofinal}, consisting
 of those $(\frX,k)$ with $\GO$-equivariant $\frX$. This limit naturally carries the structure of a coadmissible $\DGOt$-module, as will follow from part (ii) of the next theorem.
\end{para}

\begin{thm}\label{prop-equivalenceII} (i) The family

$$\Loc^{\GO}(M):=(\Loc^\dagger_{\frX,k} (M_{k}))_{(\frX,k) \in \underline{\cF}_{\frX_0}}$$

\vskip8pt

forms a coadmissible $\GO$-equivariant arithmetic $\sD$-module on $\cF_{\frX_0}$, i.e., gives an object of $\sC^{\GO}_{\frX_0}$. The formation of $\Loc^{\GO}(M)$ is functorial in $M$.

\vskip8pt

(ii) The functors $\Loc^{\GO}$ and $\Gamma(\cdot)$ induce quasi-inverse equivalences between the category of coadmissible $\DGOt$-modules and $\sC^{\GO}_{\frX_0}$.
\end{thm}

\vskip8pt

\begin{proof} Let $M$ be a coadmissible $D(G_0,L)_{\theta_0}$-module and let $\sM\in\sC^{\GO}_{\frX_0}$. Both parts of the theorem follow from the four following assertions.
\vskip8pt

{\it Assertion 1: One has $\Loc^{\GO}(M) \in \sC^{\GO}_{\frX_0}$ and $\Loc^{\GO}(M)$ is functorial in $M$.}

\vskip8pt

\begin{proof} We start by verifying condition ${\rm (a)}$ for $\Loc^{\GO}(M)$ and define the morphisms, for $g \in \GO$,

$$\phi_g: \Loc^{\GO}(M)_{\frX.g,k}\lra  (\rho_g)_*\Loc^{\GO}(M)_{\frX,k}$$

\vskip8pt

satisfying the requirements ${\rm (i)}, {\rm (ii)}$ and ${\rm (iii)}$ in definition \ref{dfn-coadmod0}. So consider

$$\Loc^{\GO}(M)_{\frX.g,k} = \Loc^\dagger_{\frX.g,k}(M_k) = \sD^\dagger_{\frX.g,k}\otimes_{\cD^{\rm an}(\bbG(k)^\circ)_{\theta_0}} M_{k} \;.$$

\vskip8pt

Let $\tilde{\phi}_g: M_{k} \ra M_{k}$ denote the map dual to the map
$V_{\bbG(k)^\circ-{\rm an}} \ra V_{\bbG(k)^\circ-{\rm an}}$ given by $w\mapsto g^{-1}w$. Let $U \sub \frX.g$ be an open subset and $P \in \sD^\dagger_{\frX.g,k}(U)$, $m\in M_{k}$.
We define

$$\phi_g (P\otimes m) :=  \Ad(g)(P)\otimes \tilde{\phi}_g(m) \;.$$

\vskip8pt

One has an isomorphism

$$(\rho_g)_*\left(\Loc^\dagger_{\frX',k'}(M_{k'})\right)\car
\left((\rho_g)_*\sD^\dagger_{\frX',k'}\right) \otimes_{\cD^{\rm an}(\bbG(k')^\circ)_{\theta_0}} M_{k'} \;.$$

\vskip8pt

Indeed, $(\rho_g)_*$ is exact and so choosing a finite presentation of $M_{k'}$ as $\cD^{\rm an}(\bbG(k')^\circ)_{\theta_0}$-module
reduces to the case of $M_{k'}=\cD^{\rm an}(\bbG(k')^\circ)_{\theta_0}$ which is trivially true. This means that the above definition extends to a map

$$\phi_g: \sD^\dagger_{\frX.g,k}\otimes_{\cD^{\rm an}(\bbG(k)^\circ)_{\theta_0}} M_{k} \lra (\rho_g)_* \left(\sD^\dagger_{\frX,k}\otimes_{\cD^{\rm an}(\bbG(k)^\circ)_{\theta_0}} M_{k}\right) \;.$$

\vskip8pt

By construction, it satisfies the requirements ${\rm (i)}, {\rm (ii)}$ and ${\rm (iii)}$.
We next verify condition ${\rm (b)}$. So suppose that $\frX',\frX$ are $\GO$-equivariant and we have $(\frX',k')\succeq (\frX,k)$ with canonical morphism $\pi: \frX' \ra \frX$ over $\frX_0$. One then has an isomorphism

$$\pi_*\left(\Loc^\dagger_{\frX',k'}(M_{k'})\right)\car
\left(\pi_*\sD^\dagger_{\frX',k'}\right) \otimes_{\cD^{\rm an}(\bbG(k')^\circ)_{\theta_0}} M_{k'} \;.$$

\vskip8pt

Indeed, $\pi_*$ is exact by \ref{prop-exactdirectimage} and we may argue as for $(\rho_g)_*$. Furthermore, $\bbG(k')^\circ\subseteq \bbG(k)^\circ$ and we denote by $\tilde{\psi}_{\frX',\frX}: M_{k'}\ra M_{k}$ the map dual to the natural inclusion $V_{\bbG(k)^\circ-{\rm an}}\subseteq V_{\bbG(k')^\circ-{\rm an}}$. Let $U \sub \frX$ be an open subset and $P \in \pi_*\sD^\dagger_{\frX',k'}(U)$, $m \in M_{k'}$. We then define

$$\psi_{\frX', \frX}(P\otimes m):=\Psi_{\frX', \frX}(P)\otimes \tilde{\psi}_{\frX',\frX}(m)$$

\vskip8pt

where $\Psi_{\frX', \frX}$ denotes the canonical morphism $\pi_*\sD^\dagger_{\frX',k'} \ra \sD^\dagger_{\frX,k}$. This definition extends to a map

$$\psi_{\frX', \frX}: \pi_* \left(\Loc^\dagger_{\frX',k'}(M_{k'})\right) \ra  \Loc^\dagger_{\frX,k}(M_{k})$$

\vskip8pt

according to our above description of $\pi_* \left(\Loc^\dagger_{\frX',k'}(M_{k'})\right)$. The map $\psi_{\frX', \frX}$ satisfies condition \ref{compatible(a)(b)1} and the required transitivity properties. It remains to see that the corresponding map $\overline{\psi}_{\frX',\frX}$ is an isomorphism, as required in \ref{isocondition1}. But $\overline{\psi}_{\frX',\frX}$ corresponds under the isomorphism of lem. \ref{lem-aux} to the linear extension

$$D(\bbG(k)^\circ,\GO)\otimes_{D(\bbG(k'),\GO)} M_{k'}\ra M_k$$

\vskip8pt

of $\tilde{\psi}_{\frX',\frX}$ via functoriality of $\Loc^\dagger_{\frX,k}$. But the linear extension of $\tilde{\psi}_{\frX',\frX}$ is an isomorphism by part (i) of lem. \ref{lem-refine} and hence, so is $\overline{\psi}_{\frX',\frX}$. This shows $\Loc^{\GO}(M)\in\sC^{\GO}_{\frX_0}$. Given a morphism $M\ra N$ of coadmissible $\DGOt$-modules, one obtains maps $M_k\ra N_k$ for any $k$ which are compatible with the maps $\tilde{\phi}_g$ and $\tilde{\psi}_{\frX',\frX}$. By functoriality of $\Loc^\dagger_{\frX,k}$, they give rise to linear morphisms

$$\Loc^\dagger_{\frX,k}(M_{k})\lra \Loc^\dagger_{\frX,k}(N_{k})$$

\vskip8pt

which are compatible with the maps $\phi_g$ and $\psi_{\frX',\frX}$. In other words, $\Loc^{\GO}(M)$ is functorial in $M$. \end{proof}

\vskip8pt

{\it Assertion 2: $\Gamma(\sM)$ is a coadmissible $D(G_0,L)_{\theta_0}$-module.}

\vskip8pt

\begin{proof} For given $k$ we choose a $(\frX,k)\in\cF_{\frX_0}$ and let $N_k:= H^0(\frX,\sM_{\frX,k})$. By \ref{isocondition1} together with
lem. \ref{lem-aux}, we then have linear isomorphisms $$D(\bbG(k)^\circ,\GO)\otimes_{D(\bbG(k'),\GO)} N_{k'}\simeq N_k$$ whenever $k'\geq k$.
Thus, the modules $N_k$ form a $(\DgkO)_{k\in\bbN}$-sequence, in the sense of \cite[1.2.8]{EmertonA} and their projective limit is therefore a coadmissible module. \end{proof}

\vskip8pt

{\it Assertion 3: $\Gamma\circ\Loc^{\GO}(M)\simeq M$.}

\vskip8pt

\begin{proof} Let $V=M'_b$. We have compatible isomorphisms $H^0(\frX,\Loc^{\GO}(M)_{\frX,k})\simeq  (V_{\bbG(k)^\circ-{\rm an}})'$ for all $(\frX,k)$ by \ref{thm-equivalence} and the coadmissible modules $\Gamma\circ\Loc^{\GO}(M)$ and $M$ have therefore isomorphic $(\DgkO)_{k\in\bbN}$-sequences.\end{proof}

\vskip8pt

{\it Assertion 4: $\Loc^{G}\circ\Gamma(\sM) \simeq \sM$.}

\vskip8pt

\begin{proof} Let $N:=\Gamma(\sM)$ and $V=N'_b$ the corresponding admissible representation. Let $\sN=\Loc^{\GO}(N)$. According to part (ii) in lem. \ref{lem-refine}
setting $N_k=D(\bbG(k)^\circ,G_0)\otimes_{D(G_0,L)} N$ produces a $(\DgkO)_{k\in\bbN}$-sequence for the coadmissible module $N$ which is isomorphic to its constituting sequence $H^0(\frX,\sM_{\frX,k})$ from Assertion 2. Now let $(\frX,k)\in\cF_{\frX_0}$. By what we just said we have linear isomorphisms

$$\sN_{\frX,k}=\Loc_{\frX,k}^\dagger(N_k)\simeq \Loc_{\frX,k}^\dagger( H^0(\frX,\sM_{\frX,k}) ) \simeq \sM_{\frX,k} \;,$$

\vskip8pt

where the final isomorphism comes from \ref{thm-equivalence}. Via this isomorphism, the action map $\phi_g^{\sN_{\frX,k}}$, constructed for $\sN=\Loc^{\GO}(N)$ along the lines of Assertion 1, corresponds to $\phi_g^{\sM_{\frX,k}}$, as follows directly from the ${\rm Ad}(g)$-linearity of these two maps. Similarly, if $(\frX',k')\succeq (\frX,k)$ for $G_0$-equivariant $\frX',\frX$, then the transition map $\psi^{\sN}_{\frX',\frX}$, constructed for $\sN=\Loc^{\GO}(N)$ along the lines of Assertion 1, corresponds to $\psi^{\sM}_{\frX',\frX}$, as follows directly from the $\Psi_{\frX',\frX}$-linearity of these two maps. Hence, $\sN\simeq \sM$ in $\sC^{\GO}_{\frX_0}$. \end{proof}

This finishes the proof of the theorem. \end{proof}

\vskip8pt

\begin{para}\label{para-sheaf}
We indicate how coadmissible $\GO$-equivariant $\sD$-modules can be 'realized' as honest (equivariant) sheaves on the topological space
$\frX_\infty=\varprojlim_{\cF_{\frX_0}} \frX$, cf. \ref{intro-projlim}. The induced $G_0$-action on $\frX_\infty$ is denoted by $\rho_g:\frX_\infty\rightarrow\frX_\infty$ for $g\in G_0$. We denote the canonical projection map $\frX_\infty \ra \frX$ by ${\rm sp}_{\frX}$ for each $\frX$ and define the following sheaf of rings on $\frX_\infty$.
Assume $V\sub\frX_\infty$ is an open subset of the form ${\rm sp}_{\frX}^{-1}(U)$ with an open subset $U\sub\frX$ for a model $\frX\in\cF_{\frX_0}$. We have that

$${\rm sp}_{\frX'}(V)=\pi^{-1}(U)$$

\vskip8pt

for any morphism $\pi:\frX' \ra \frX$ over $\frX_0$ and so, in particular, ${\rm sp}_{\frX'}(V)\sub\frX'$ is an open subset for such $\frX'$.
Moreover,

$$\pi^{-1}({\rm sp}_{\frX'}(V))={\rm sp}_{\frX''}(V)$$

\vskip8pt

whenever $\pi:\frX'' \ra \frX'$ is a morphism over $\frX$. In this situation, the morphism (\ref{equ-transit_sheaf}) induces the ring homomorphism

\begin{numequation}\label{equ-transit_hom} \sD^\dagger_{\frX'',k''}({\rm sp}_{\frX''}(V))=\pi_*\sD^\dagger_{\frX'',k''}(({\rm sp}_{\frX'}(V)) \ra
\sD^\dagger_{\frX',k'}({\rm sp}_{\frX'}(V))
\end{numequation}

and we form the projective limit

$$\sD_{\infty}(V):=\varprojlim_{\frX' \ra \frX} \sD^\dagger_{\frX',k'}({\rm sp}_{\frX'}(V))$$

\vskip8pt

over all these maps. The open subsets of the form $V$ form a
basis for the topology on $\frX_\infty$ and $\sD_{\infty}$ is a presheaf on this basis. We denote the associated sheaf
on $\frX_\infty$ by the symbol $\sD_{\infty}$ as well. It is a $\GO$-equivariant sheaf of rings on $\frX_\infty$ in the usual sense: given $g\in\GO$, the actions ${\rm Ad}(g)$ on each individual sheaf $\sD^\dagger_{\frX,k}$, cf. (\ref{equ-ringiso0}), assemble to a left action

\begin{numequation}\label{equ-ringisoinfty}
{\rm Ad}(g): \sD_{\infty}\car (\rho_g)_*\sD_{\infty}
\end{numequation}

on $\sD_{\infty}$.
\end{para}

\begin{para}
Suppose $\sM:=(\sM_{\frX,k})$ is an object of $\sC^{\GO}_{\frX_0}$. We have the transition maps $\psi_{\frX',\frX}: \pi_*\sM_{\frX',k'} \ra \sM_{\frX,k}$ which are linear relative to the morphism (\ref{equ-transit_sheaf}). In a completely analogous manner as above, we obtain a sheaf $\sM_\infty$ on $\frX_\infty$ together with a family $(\phi_g)_{g \in \GO}$ of isomorphisms

\begin{numequation}\label{equ-modulisoinfty}
\phi_g:\sM_\infty\lra (\rho_g)_*\sM_{\infty}
\end{numequation}

of sheaves of $L$-vector spaces, satisfying the following conditions:

\vskip8pt

\begin{enumerate}
\item For all $g,h \in \GO$ we have $(\rho_g)_*(\phi_h) \circ \phi_g =\phi_{hg}$.

\vskip5pt

\item For all open subsets $U \sub \frX_\infty$, all $P \in \sD_{\infty}(U)$, and all $m \in \sM_\infty(U)$ one has $\phi_g(P.m) = \Ad(g)(P).\phi_g(m)$.
\end{enumerate}

\vskip8pt

In particular, $\sM_\infty$ is an equivariant $\sD_{\infty}$-module on the topological $G_0$-space $\frX_\infty$ in the usual sense. The formation of $\sM_\infty$ is functorial in $\sM \in \sC^{\GO}_{\frX_0}.$
\end{para}

\begin{prop}\label{faithful}
The functor $\sM \rightsquigarrow \sM_\infty$ from the category $\sC^{\GO}_{\frX_0}$ to $\GO$-equivariant $\sD_{\infty}$-modules is a faithful functor.
\end{prop}

\begin{proof} We have ${\rm sp}_{\frX}(\frX_\infty)=\frX$ for all $\frX$.
The global sections of $M_\infty$ are therefore equal to

$$H^0(\frX_\infty,\sM_\infty)=\varprojlim_{(\frX,k) \in \underline{\cF}_{\frX_0}}H^0(\frX,\sM_{\frX,k})=\Gamma(\sM)$$

\vskip8pt

where we have used prop. \ref{prop-cofinal}. Now let $f, h$ be two morphisms $\sM\ra \sN$ in $\sC^{\GO}_{\frX_0}$ such that $f_\infty=h_\infty$. By the equivalence of categories in \ref{prop-equivalenceII}, it suffices to verify $\Gamma(f)=\Gamma(h)$ (as maps between sets, say). But this is clear since $H^0(\frX_\infty,f_\infty)=H^0(\frX_\infty,h_\infty)$. \end{proof}

\vskip8pt

We denote by $\Loc^{\GO}_\infty$ the composite of the functor $\Loc^{\GO}$ with $(\cdot)_\infty$, i.e.

$$\{ ~coadmissible~ \DGO_{\theta_0}-modules~\} \xrightarrow{\Loc^{\GO}_\infty} \{ ~\GO-equivariant~ \sD_{\infty}-modules~ \} \;.$$

\vskip8pt

Since $\Loc^{\GO}$ is an equivalence, the preceding proposition implies that $\Loc^{\GO}_\infty$ is a faithful functor.

\vskip8pt

\subsection{\texorpdfstring{$G$}{}-equivariance and the functor \texorpdfstring{$\Loc^{G}$}{}}\label{G_equivariance} Let $G:=\bbG(L)$. Denote by $\cB$ the (semi-simple) Bruhat-Tits building of the $p$-adic group $G$ together with its natural $G$-action. In accordance with our convention that the group $G$ acts on the right on the flag variety, we also consider $\cB$ with a {\it right} action: $\cB \times G \ra \cB$, $(x,g) \mapsto xg$. We reserve the letter $v$ for special vertices of $\cB$.

\vskip8pt

The purpose of this subsection is to extend the above results from $\GO$-equivariant objects to objects equivariant for the full group $G$.

\begin{para}\label{G_action_on_F}
To each special vertex $v \in \cB$ Bruhat-Tits theory associates a connected reductive group scheme $\bbG_v$ over $\fro$. The generic fiber of $\bbG_v$ is canonically isomorphic to $\bbG$.  We denote by $X_{v,0}$ the flag scheme of $\bbG_v$. Is is a smooth scheme over $\fro$ whose generic fiber is canonically isomorphic to the flag variety $\bbX$ of $\bbG$. All constructions in sections \ref{models} and \ref{loc_n} are associated with the group scheme $\bbG_0$ with vertex $v_0$, say, but can be done canonically for any other of the reductive group schemes $\bbG_v$. We distinguish the various constructions from each other by adding the corresponding vertex $v$ to them, i.e., we write $X_v$ for an admissible blow-up of the smooth model $X_{v,0}$, $\GOv$ for the group of points $\bbG_v(\fro)$, and $G_{v,k}$ for the group of points $\bbG_v(k)(\fro)$. The same conventions apply when we work with the formal completions, i.e., $\frX_{v,0}$ is the formal completion of $X_{v,0}$, and $\frX_v$ always denotes an admissible formal blow-up of $\frX_{v,0}$. We make the general convention that the blow-up morphism $\frX_v \ra \frX_{v,0}$ is part of the datum of $\frX_v$. That is to say, even if a blow-up $\frX_v$ of $\frX_{v,0}$ also allows for a blow-up morphism to another smooth formal model $\frX_{v',0}$, with $v' \neq v$, we only consider it a blow-up of $\frX_{v,0}$. We denote by $\cF_v := \cF_{\frX_{v,0}}$ the set of all admissible formal blow-ups $\frX_v \ra \frX_{v,0}$ of $\frX_{v,0}$
 and by $\underline{\cF}_v:=\underline{\cF}_{\frX_{v,0}}$ the set of pairs defined analogously to \ref{partial_ordering}. By the convention we just introduced, the sets $\cF_v$ and $\cF_{v'}$ are disjoint if $v$ and $v'$ are two distinct vertices. Let

$$\cF := \coprod_v \cF_v \;,$$

\vskip8pt

where $v$ runs over all special vertices of $\cB$, be the disjoint union of all these models.
We recall that $\frX_\infty$ equals the projective limit of all formal models of $\bbX^\rig$, cf. \ref{intro-projlim}.
The set $\cF$ is partially ordered via
$\frX_{v'} \succeq \frX_v$ if the projection $\pr_{\frX_v}: \frX_\infty\ra \frX_v$ factors through the projection $\pr_{\frX_{v'}}: \frX_\infty\ra \frX_{v'}$. In this case, the resulting morphism $\frX_{v'}\ra\frX_v$ is an admissible formal blow-up of $\frX_v$ \cite[Thm. 8.1.24]{LiuBook}. Finally, by the property recalled at the end of \ref{models_and_formal_models}, the ordered set $(\cF,\succeq)$ is directed in the sense that any two elements have a common upper bound.
\end{para}

\begin{dfn}\label{partial_orderingII} We denote by $\underline{\cF} = \coprod_v \underline{\cF}_v$ the disjoint union of all $\underline{\cF}_v$, where $v$ runs through all special vertices of $\cB$. We define an ordering on this set by declaring $(\frX_{v'},k')\succeq (\frX_{v},k)$ if and only if $\frX_{v'}\succeq\frX_v$ and $\varpi^{k'}\Lie(\bbG_{v'}) \subseteq \varpi^{k}\Lie(\bbG_v)$ as lattices in $\Lie(\bbG)$.
\end{dfn}

\begin{para}\label{G_action}
For any special vertex $v \in \cB$, any element $g \in G$ induces a isomorphism $\rho_g^v: X_{v,0} \ra X_{vg,0}$. The morphism induced by $\rho_g^v$ on the generic fibers $X_{v,0} \times \Spec(L) \simeq \bbX \simeq X_{vg,0} \times \Spec(L)$ coincides with the right translation by $g$ on $\bbX$. Moreover, $\rho_g^v$ induces a morphism $\frX_{v,0}\lra \frX_{vg,0}$,which we again denote by $\rho_g^v$ or $\rho_g$, and which coincides with the right translation action on $\frX_{v,0}$ for
$g \in \GOv$ (note that $vg = v$ in this case). Let $\rho_g^\sharp: \cO_{\frX_{vg,0}} \ra (\rho_g)_*\cO_{\frX_{v,0}}$ be the comorphism of $\rho_g$. If $\pi: \frX_v \ra \frX_{v,0}$ is an admissible blow-up of an ideal $\cI\subset \cO_{\frX_{v,0}}$, then blowing-up $(\rho_g^\sharp)^{-1}((\rho_g)_*\cI)$ produces a formal scheme $\frX_{vg}$ (which, for $g\in G_{v,0}$, we denoted by $\frX_{v}.g$ in \ref{translated_model}), together with an isomorphism $\rho_g = \rho_g^v: \frX_{v} \ra \frX_{vg}$. We have again
 $k_{\frX_v}=k_{\frX_{vg}}$ in this situation. For any $g,h
 \in G$ and any admisible formal blow-up $\frX_v$ of $\frX_{v,0}$ we have $\rho^{vg}_h \circ \rho^v_g = \rho^v_{gh}: \frX_v \ra \frX_{vgh}$. This gives a right $G$-action on the family $\cF$ and on the projective limit space $\frX_\infty$.\footnote{The existence of the $G$-action on $\frX_\infty$ can also be deduced from the fact that $\frX_\infty$ is canonically and functorially associated to $\bbX^{\rig}$ whose $G$-action is induced by the $\bbG$-action on $\bbX$.} Finally, if $\frX_{v'}\succeq \frX_v$ with morphism $\pi:\frX_{v'} \ra \frX_v$
 and $g\in G$, then $\frX_{v'g} \succeq \frX_{vg}$ with a resulting morphism $\frX_{v'g} \ra \frX_{vg}$ which we denote by $\pi.g$, as in cor. \ref{cor-actionpi}.

\vskip8pt

On the level of differential operators, we have the following two key properties as before, cf. paragraph \ref{twoproperties}.
Let $g \in G$. The isomorphism $\rho_g: \frX_v \lra \frX_{vg}$ induces an adjoint action

\begin{numequation}\label{equ-ringiso1}{\rm Ad}(g): \sD^\dagger_{\frX_{vg,k}}\car (\rho_g)_*\sD^\dagger_{\frX_{v},k} \;, \;\; D\mapsto \rho_g^\sharp D (\rho_g^\sharp)^{-1} \;,
\end{numequation}

for $k\geq k_{\frX_v}=k_{\frX_{vg}}$. Secondly, we identify the global sections $\Gamma(\frX_v, \sD^\dagger_{\frX_v,k})$ with $\cD^{\rm an}(\bbG_v(k)^\circ)_{\theta_0}$ and obtain the group homomorphism

\begin{numequation}\label{equ-ringiso2} G_{v,k+1}\lra \Gamma(\frX_v, \sD^\dagger_{\frX_v,k})^\times \;, \;\; g\mapsto \delta_g \;,
\end{numequation}

where $G_{v,k+1}=\bbG_v(k)^\circ(L)$ denotes the group of $L$-rational points.
\end{para}

\begin{prop}\label{equ-transit_sheafII}
Suppose $(\frX_{v'},k')\succeq (\frX_v,k)$ for two pairs $(\frX_{v'},k'), (\frX_{v},k)\in\underline{\cF}$ with morphism $\pi: \frX_{v'}\ra\frX_v$. There exists a canonical morphism of sheaves of rings\footnote{In order to alleviate notation we do not indicate that these maps depend on $(\frX_{v'},k')$ and $(\frX_v,k)$. The source and target of these maps should be clear from the context.}

$$\Psi: \pi_*\sD^\dagger_{\frX_{v'},k'} \ra \sD^\dagger_{\frX_v,k}$$

\vskip8pt

which is $G$-equivariant in the sense that for every $g \in G$ the following diagram is commutative:

\[\xymatrixcolsep{11pc}
\xymatrix{
(\pi.g)_*\sD^\dagger_{\frX_{v'}.g,k'} \ar[r]^{\Psi} \ar[d]^{(\pi.g)_*({\rm Ad}(g))} &  \sD^\dagger_{\frX_v.g,k} \ar[d]^{{\rm Ad}(g)}\\
(\pi.g)_*(\rho^{v'}_g)_*\sD^\dagger_{\frX_{v'},k} = (\rho^v_g)_* \pi_*\sD^\dagger_{\frX_{v'},k} \ar[r]^{(\rho^v_g)_*(\Psi)} & (\rho^v_g)_*\sD^\dagger_{\frX_v,k} \\
}\]

\vskip8pt

\end{prop}

\begin{proof} Let $\pr: \frX_v \ra \frX_{v,0}$ and $\pr': \frX_{v'} \ra \frX_{v',0}$ be the blow-up morphisms, and put $\widetilde{\pr} = \pr \circ \pi$. The following diagram displays these morphisms:

\[\xymatrixcolsep{5pc}
\xymatrix{
\frX_{v'} \ar[d]^{\pr'} \ar[r]^{\pi} \ar[dr]^{\widetilde{\pr}}& \frX_v  \ar[d]^{\pr}\\
\frX_{v',0} & \frX_{v,0}  \\
}\]

\vskip8pt

Fix $m \in \bbN$. We show first the existence of a canonical morphism of sheaves of $\fro$-algebras

\begin{numequation}\label{transit_1} \sD^{(k',m)}_{X_{v'}} \lra \widetilde{\pr}^* \sD^{(k,m)}_{X_{v,0}} \;.
\end{numequation}

\vskip8pt

Here $X_{v'}$, $X_{v',0}$, $X_v$, and $X_{v,0}$ are the schemes of finite type over $\fro$ whose completions are $\frX_{v'}$, $\frX_{v',0}$, $\frX_v$, and $\frX_{v,0}$, respectively, cf. \ref{algebraization}. The morphisms between these schemes of finite type over $\fro$ will be denoted by the same letters, e.g., $\pr: X_v \ra X_{v,0}$. We recall that there this a canonical surjective morphism

$$\xi^{(k',m)}_{X_{v'}}: \cA^{(k',m)}_{X_{v'}} = \cO_{X_{v'}} \otimes_\fro D^{(m)}(\bbG_{v'}(k')) \twoheadrightarrow \sD^{(k',m)}_{X_{v'}} \;,$$

\vskip8pt

cf. \ref{prop-auxiliaryI} of sheaves on $X_{v'}$. On the other hand we apply $\widetilde{\pr}^*$ to the surjection

$$\xi^{(k,m)}_{X_{v,0}}: \cA^{(k,m)}_{X_{v,0}} = \cO_{X_{v,0}} \otimes_\fro D^{(m)}(\bbG_{v}(k)) \twoheadrightarrow \sD^{(k,m)}_{X_{v,0}} \;,$$

\vskip8pt

and obtain a surjection $\cO_{X_{v'}} \otimes_\fro D^{(m)}(\bbG_{v}(k)) \twoheadrightarrow \widetilde{\pr}^* \sD^{(k,m)}_{X_{v,0}}$. Recall that $(\frX_{v'},k') \succeq (\frX_v,k)$ implies that $\vpi^{k'} \Lie(\bbG_{v'})$ is contained in $\vpi^k \Lie(\bbG_v)$. The description of the ring $D^{(m)}(\bbG_{v}(k))$ in \ref{div_power_env_algs} shows that the inclusion $\vpi^{k'} \Lie(\bbG_{v'}) \sub \vpi^k \Lie(\bbG_v)$ gives rise to an injective ring homomorphism $D^{(m)}(\bbG_{v'}(k')) \hra D^{(m)}(\bbG_{v}(k))$. We now claim that the composition

$$\cO_{X_{v'}} \otimes_\fro D^{(m)}(\bbG_{v'}(k')) \hra \cO_{X_{v'}} \otimes_\fro D^{(m)}(\bbG_{v}(k)) \twoheadrightarrow \widetilde{\pr}^*\sD^{(k,m)}_{X_{v,0}}$$

\vskip8pt

factors through $\sD^{(k',m)}_{X_{v'}}$. As all those sheaves are $\vpi$-torsion free, this can be checked after tensoring with $L$ in which case we use that $\sD^{(k',m)}_{X_{v'}} \otimes_\fro L \simeq \widetilde{\pr}^* \sD^{(k,m)}_{X_{v,0}} \otimes_\fro L$ is the (push-forward of the) sheaf of (algebraic) differential operators on the generic fiber of $X_{v'}$. We thus get a canonical morphism of sheaves \ref{transit_1}. Passing to completions induces a canonical morphism $\hsD^{(k',m)}_{\frX_{v'}} \ra \widetilde{\pr}^* \hsD^{(k,m)}_{\frX_{v,0}}$. Taking the inductive limit over all $m$ and inverting $\vpi$ gives a canonical morphism $\sD^\dagger_{\frX_{v'},k'} \ra \widetilde{\pr}^* \sD^\dagger_{\frX_{v,0},k}$. Now we consider the formal scheme $\frX_{v'}$ as a blow-up of $\frX_{v,0}$ via $\widetilde{\pr}$. Then $\pi$ becomes a morphism of formal schemes over $\frX_{v,0}$, and we can consider $\widetilde{\pr}^* \sD^\dagger_{\frX_{v,0},k}$ as the sheaf of arithmetic differential operators with congruence level $k$ defined on $\frX_{v'}$ via $\widetilde{\pr}$, as introduced in \ref{completions}. Using \ref{prop-exactdirectimage} in this setting shows then that $\pi_*\Big(\widetilde{\pr}^* \sD^\dagger_{\frX_{v,0},k}\Big) = \sD^\dagger_{\frX_v,k}$. Then, applying $\pi_*$ to the morphism $\sD^\dagger_{\frX_{v'},k'} \ra \widetilde{\pr}^* \sD^\dagger_{\frX_{v,0},k}$
gives the morphism $\Psi: \pi_*\sD^\dagger_{\frX_{v'},k'} \ra \sD^\dagger_{\frX_v,k}$ of the statement. Making use of the maps $\xi^{(k,m)}_X$, as above, the assertion regarding $G$-equivariance can similarly be reduced to some obvious functorial properties of the rings $D^{(m)}(\bbG_{v}(k))$.
\end{proof}

\vskip8pt

\begin{dfn}\label{dfn-coadmod} A {\it coadmissible $G$-equivariant arithmetic $\sD$-module} on $\cF$ consists of a family $\sM:=(\sM_{\frX,k})_{(\frX,k)\in\underline{\cF}}$ of coherent $\sD^\dagger_{\frX,k}$-modules $\sM_{\frX,k}$ with the following properties:\footnote{From now on we use the notation $\frX_v$ instead of $\frX$ to indicate that the model is an admissible formal blow-up of
$\frX_{v,0}$.}

\vskip5pt

{\rm (a)} For any $v$ and $g\in G$ with isomorphism $\rho^v_g: \frX_v\lra\frX_{vg}$, there exists a isomorphism

$$\phi^v_g: \sM_{\frX_{vg},k}\lra (\rho^v_g)_*\sM_{\frX_v,k}$$

\vskip8pt

of sheaves of $L$-vector spaces, satisfying the following conditions:

\vskip8pt

\begin{enumerate}
\item For all $g,h \in G$ we have $(\rho^v_g)_*(\phi^v_h) \circ \phi^v_g = \phi^v_{hg}$.

\vskip8pt

\item For all open subsets $U \sub \frX_{vg}$, all $P \in \sD^\dagger_{\frX_{vg},k}(U)$, and all $m \in \sM_{\frX_{vg},k}(U)$ one has $\phi^v_g(P.m) = \Ad(g)(P).\phi^v_g(m)$.

\vskip8pt

\item\footnote{To make sense of this condition, we use that elements $g \in G_{k+1,v}$ act trivially on the topological space underlying $\frX_v$, cf. \ref{Gk_acts_trivially}.} For all $g \in G_{k+1,v}$ the map $\phi^v_g: \sM_{\frX_v,k} \ra (\rho^v_g)_* \sM_{\frX_v,k}  = \sM_{\frX_v,k}$ is equal to the multiplication by $\delta_g \in H^0(\frX_v, \sD^\dagger_{\frX_v,k})$.
\end{enumerate}

\vskip5pt

{\rm (b)} For any two pairs $(\frX_{v'},k') \succeq (\frX_v,k)$ in $\underline{\cF}$ with morphism $\pi:\frX_{v'}\ra\frX_v$ there is a transition morphism $\psi_{\frX_{v'}, \frX_v}: \pi_*\sM_{\frX_{v'}} \ra \sM_{\frX_v}$, linear relative to the canonical morphism $\Psi: \pi_*\sD^\dagger_{\frX_{v'},k'} \ra \sD^\dagger_{\frX_v,k}$ (\ref{equ-transit_sheafII}) and satisfying
\begin{numequation}\label{compatible(a)(b)} \phi_g^v \circ \psi_{\frX_{v'g}, \frX_{vg}}=(\rho_g^v)_*(\psi_{\frX_{v'}, \frX_v})\circ (\pi.g)_*(\phi^{v'}_g)
\end{numequation}
for any $g\in G$. If $v'=v$, and $(\frX',k') \succeq (\frX,k)$ in $\underline{\cF}_v$, and if $\frX,\frX'$ are $G_{v,0}$-equivariant, then we require additionally that the morphism induced by $\psi_{\frX', \frX}$, cf \ref{factormap},

\begin{numequation}\label{isocondition2}
\overline{\psi}_{\frX', \frX}: \sD^\dagger_{\frX,k} \otimes_{\pi_*\sD^\dagger_{\frX',k'},G_{k+1}}  \pi_*\sM_{\frX',k'} \car \sM_{\frX,k}
\end{numequation}

is an isomorphism of $\sD^\dagger_{\frX,k}$-modules. In general, the morphisms $\psi_{\frX_{v'}, \frX_v}: \pi_*\sM_{\frX_{v'},k'} \ra \sM_{\frX_v,k}$ are required to satisfy the transitivity condition $\psi_{\frX_{v'}, \frX_v} \circ \pi_*(\psi_{\frX_{v''},\frX_{v'}}) = \psi_{\frX_{v''},\frX_v}$, whenever $(\frX_{v''},k'')\succeq (\frX_{v'},k')\succeq (\frX_v,k)$ in $\underline{\cF}$. Moreover, $\psi_{\frX_v, \frX_v}:={\rm id}_{\sM_{\frX_v,k}}$.

\vskip5pt

A {\it morphism} $\sM \ra \sN$ between two coadmissible $G$-equivariant arithmetic $\sD$-modules consists of morphisms $\sM_{\frX,k} \ra \sN_{\frX,k}$ of  $\sD^\dagger_{\frX,k}$-modules which are compatible with the extra structure. We denote the resulting category by $\sC^G_\cF$.
\end{dfn}

\begin{para} We now make the link to the category of coadmissible $D(G,L)_{\theta_0}$-modules, cf. \ref{module}. Let $M$ be such a module and let $V:=M'_b$. Fix a special vertex $v$. Let $V_{\bbG_v(k)^\circ-{\rm an}}$ be the subspace of $\bbG_v(k)^\circ$-analytic vectors and let $M_{v,k}$ be its continuous dual. For any $(\frX_v,k)\in\underline{\cF}$ we have the coherent $\sD^\dagger_{\frX_v,k}$-module

$$\Loc^\dagger_{\frX_v,k} (M_{v,k})=\sD^\dagger_{\frX_v,k}\otimes_{\cD^{\rm an}(\bbG_v(k)^\circ)_{\theta_0}} M_{v,k} \;,$$

\vskip8pt

according to thm. \ref{thm-equivalence}. On the other hand, given an object $\sM\in \sC^G_\cF$, we may consider the projective limit

$$\Gamma(\sM):=\varprojlim_{(\frX,k) \in \underline{\cF}} H^0(\frX,\sM_{\frX,k})$$

\vskip8pt

with respect to the transition maps $\psi_{\frX', \frX}$. Here, the projective limit is taken in sense of abelian groups and over the cofinal family
of pairs $(\frX_v,k)\in\underline{\cF}$ with $G_{v,0}$-equivariant $\frX_v$.

\end{para}

\begin{thm}\label{thm_G_equiv}
(i) The family

$$\Loc^{G}(M):=(\Loc^\dagger_{\frX_v,k} (M_{v,k}))_{(\frX_v,k) \in \underline{\cF}}$$

\vskip8pt

forms a coadmissible $G$-equivariant arithmetic $\sD$-module on $\cF$, i.e., gives an object of $\sC^{G}_{\cF}$. The formation of $\Loc^{G}(M)$ is functorial in $M$.

\vskip8pt

(ii) The functors $\Loc^{G}$ and $\Gamma(\cdot)$ induce quasi-inverse equivalences between the category of coadmissible $D(G,L)_{\theta_0}$-modules and $\sC^G_\cF$.
\end{thm}

\begin{proof} The proof is an extension, taking into account the additional $G$-action, of the proof for the compact subgroup $\GO$ treated in the preceding subsection, cf.\ref{prop-equivalenceII}. Let $M$ be a coadmissible $D(G,L)_{\theta_0}$-module and let $\sM\in\sC^G_\cF$. The theorem follows from the four following assertions.

\vskip8pt

{\it Assertion 1: One has $\Loc^{G}(M)\in\sC^G_\cF$ and
$\Loc^{G}(M)$ is functorial in $M$.}

\vskip8pt

\begin{proof} For condition ${\rm (a)}$ for $\Loc^{G}(M)$ we need the maps

$$\phi^v_g: \Loc^G(M)_{\frX_{vg},k}\lra  (\rho^v_g)_*\Loc^G(M)_{\frX_{v},k}$$

\vskip8pt

satisfying the requirements ${\rm (i)}, {\rm (ii)}$ and ${\rm (iii)}$. Let $\tilde{\phi}^v_g: M_{vg,k} \ra M_{v,k}$ denote the map dual to the map $V_{\bbG_v(k)^\circ-{\rm an}} \lra V_{\bbG_{vg}(k)^\circ-{\rm an}}$ given by $w\mapsto g^{-1}w$ (note that $\bbG_{vg}(k)^\circ = g^{-1} \bbG_{v}(k)^\circ g$ in $\bbG^{\rig}$). Let $U \sub \frX_{vg}$ be an open subset and $P \in \sD^\dagger_{\frX_{vg},k}(U)$, $m\in M_{vg,k}$.
We define

\begin{numequation}\label{equ-ver} \phi^v_g (P\otimes m) :=  \Ad(g)(P)\otimes \tilde{\phi}^v_g(m) \;.
\end{numequation}

This definition extends to a map

$$\phi^v_g: \sD^\dagger_{\frX_{vg},k}\otimes_{\cD^{\rm an}(\bbG_{vg}(k)^\circ)_{\theta_0}} M_{vg,k} \lra (\rho^v_g)_* (\sD^\dagger_{\frX_{v},k}\otimes_{\cD^{\rm an}(\bbG_{v}(k)^\circ)_{\theta_0}} M_{v,k})$$

\vskip8pt

which satisfies the requirements ${\rm (i)}, {\rm (ii)}$ and ${\rm (iii)}$. We next verify condition ${\rm (b)}$. Given $(\frX_{v'},k') \succeq (\frX_v,k)$  in $\underline{\cF}$, we have $\bbG_{v'}(k')^\circ\subseteq \bbG_{v}(k)^\circ$
in $\bbG^{\rig}$ and we denote by $\tilde{\psi}_{\frX_{v'},\frX_v}: M_{v',k'}\ra M_{v,k}$ the map dual to the natural inclusion $V_{\bbG_{v}(k)^\circ-{\rm an}}\subseteq V_{\bbG_{v'}(k')^\circ-{\rm an}}$. Let $U \sub \frX_{v}$ be an open subset and $P \in \pi_*\sD^\dagger_{\frX_{v'},k'}(U)$, $m \in M_{v',k'}$. We then define

\begin{numequation}\label{equ-ver2} \psi_{\frX_{v'}, \frX_v}(P\otimes m):=\Psi_{\frX_{v'}, \frX_v}(P)\otimes \tilde{\psi}_{\frX_{v'},\frX_v}(m) \end{numequation}

where $\Psi_{\frX_{v'}, \frX_v}$ denotes the canonical morphism $\pi_*\sD^\dagger_{\frX_{v'},k'} \ra \sD^\dagger_{\frX_v,k}$
from prop. \ref{equ-transit_sheafII}. This definition extends to a map

$$\psi_{\frX_{v'}, \frX_v}: \pi_* \Loc^G(M)_{\frX_{v'},k'} \ra \Loc^G(M)_{\frX_{v},k}$$

\vskip8pt

which satisfies all required conditions. The functoriality of $\Loc^G$ is verified entirely similar to the case of $\Loc^{\GO}$.
\end{proof}

\vskip8pt

{\it Assertion 2: $\Gamma(\sM)$ is a coadmissible $D(G,L)_{\theta_0}$-module.}

\vskip8pt

\begin{proof} We already know that $\Gamma(\sM)$ is a coadmissible $D(\GOv,L)_{\theta_0}$-module for any $v$, cf. thm. \ref{prop-equivalenceII}. So it suffices to exhibit a compatible $G$-action on $\Gamma(\sM)$. Let $g\in G$. The isomorphism
$$\phi^v_g: \sM_{\frX_{vg},k}\lra (\rho^v_g)_*\sM_{\frX_v,k}$$
is compatible with transition maps according to
\ref{compatible(a)(b)}. We therefore obtain an isomorphism

$$\Gamma(\sM)=\varprojlim_{\underline{\cF}_{vg}} \Gamma(\frX_{vg},\sM_{\frX_{vg},k})\stackrel{g}{\lra} \varprojlim_{\underline{\cF}_{v}} \Gamma(\frX_{v},\sM_{\frX_{v},k})= \Gamma(\sM) \;.$$

\vskip8pt

According to {\rm (i)}, {\rm (ii)} and {\rm (iii)} in \ref{dfn-coadmod}, this gives indeed a $G$-action on $\Gamma(\sM)$ which is compatible with its various $D(\GOv,L)$-module structures. \end{proof}

\vskip8pt

{\it Assertion 3: $\Gamma\circ\Loc^{G}(M)\simeq M$.}

\vskip8pt

\begin{proof} We already know that this hold as coadmissible $D(G_0,L)_{\theta_0}$-modules, cf. thm. \ref{prop-equivalenceII}, so it suffices to identify the $G$-action on both sides. Let $v$ be a special vertex. According to \ref{equ-ver}, the action

$$\Gamma\circ\Loc^{G}(M)\simeq \varprojlim_k M_{vg,k}\ra \varprojlim_k M_{v,k}\simeq \Gamma\circ\Loc^{G}(M)$$

\vskip8pt

of an element $g\in G$ on $\Gamma\circ\Loc^{G}(M)$ is induced by $\tilde{\phi}^v_g: M_{vg,k}\ra M_{v,k}.$ The identification $M\simeq \varprojlim_k M_{vg,k} \simeq \varprojlim_k M_{v,k}$ (coming from dualizing $V=\cup_k V_{\bbG_{vg}(k)^\circ-{\rm an}} = \cup_k V_{\bbG_v(k)^\circ-{\rm an}}$)  therefore gives back the original action of $g$ on $M$. \end{proof}

\vskip8pt

{\it Assertion 4: $\Loc^{G}\circ\Gamma(\sM) \simeq \sM$.}

\vskip8pt

\begin{proof} We know that $\Loc^{G}(\Gamma(\sM))_{\frX_v,k}\simeq\sM_{\frX_v,k}$ as $\sD^\dagger_{\frX_v,k}$-modules for any $(\frX_v,k)\in\underline{\cF}$, cf. \ref{thm-equivalence}. It now remains to check that these isomorphisms are compatible with the maps $\phi_g^v$ and $\psi_{\frX_{v'},\frX_v}$ on both sides.
This works as in the $\GO$-case, but let us spell out the argument for the maps
$\phi_g^v$ in detail. The maps
$\phi_g^v$ on the left-hand side are induced by the maps on the right-hand side as follows. Given

$$\phi^v_g: \sM_{\frX_{vg},k} \lra   (\rho^v_g)_*\sM_{\frX_{v},k} \;,$$

\vskip8pt

the corresponding map

$$ \phi^v_g: \Loc^{G}(\Gamma(\sM))_{\frX_{vg},k} \lra (\rho^v_g)_* (\Loc^{G}(\Gamma(\sM))_{\frX_{v},k})$$

\vskip8pt

equals the map

$$\sD^\dagger_{\frX_{vg},k}\otimes_{\cD^{\rm an}(\bbG_{vg}(k)^\circ)_{\theta_0}} H^0(\frX_{vg},\sM_{\frX_{vg},k})
 \lra (\rho^v_g)_* (\sD^\dagger_{\frX_v,k}\otimes_{\cD^{\rm an}(\bbG_{v}(k)^\circ)_{\theta_0}} H^0(\frX_{v},\sM_{\frX_v,k}))$$

\vskip8pt

given locally by $\Ad(g)(\cdot)\otimes H^0(\frX_{vg},\phi^v_g)$, cf. \ref{equ-ver}. Let $U \sub \frX_{v}$ be an open subset and $P \in \sD^\dagger_{\frX_{v},k}(U)$, $m\in M_{v,k} = H^0(\frX_{vg},\sM_{\frX_{v},k}) $.
The isomorphisms $\Loc^{G}(\Gamma(\sM))_{\frX_{v},k} \simeq\sM_{\frX_v,k}$ are induced (locally) by $P\otimes m \mapsto P.(m|_U)$. Using condition ${\rm (ii)}$ in \ref{dfn-coadmod}, one then sees that these isomorphisms interchange the maps $\phi^v_g$, as desired.
The compatibility with transition maps $\psi_{\frX_{v'}, \frX_v}$ for two models $(\frX_{v'},k') \succeq (\frX_v,k)$ in $\underline{\cF}$ is deduced in an entirely similar manner from \ref{equ-ver2} and the fact that $\psi_{\frX_{v'}, \frX_v}$ is linear relative to the canonical morphism $\Psi: \pi_*\sD^\dagger_{\frX_{v'},k'} \ra \sD^\dagger_{\frX_v,k}$. \end{proof}

This finishes the proof of the theorem. \end{proof}

\vskip8pt

As in the case of the group $\GO$, we now indicate how objects from
$\sC^{G}_{\cF}$ can be 'realized' as honest $G$-equivariant sheaves on the $G$-space $\frX_\infty$. Recall that we have the $\GO$-equivariant sheaf $\sD_{\infty}$ on $\frX_\infty$, cf. \ref{para-sheaf}.
\begin{prop}
The $\GO$-equivariant structure on the sheaf $\sD_{\infty}$ extends to a $G$-equivariant structure.
\end{prop}

\begin{proof} This can be shown very similar to \cite[Proof of Prop. 5.4.5]{PSS4}. \end{proof}

\vskip8pt

Recall the faithful functor $\sM \rightsquigarrow \sM_\infty$ from coadmissible $\GO$-equivariant arithmetic $\sD$-modules on $\cF_{\frX_0}$ to $\GO$-equivariant $\sD_{\infty}$-modules on $\frX_\infty$, cf. \ref{faithful}. If $\sM$ comes from a coadmissible $G$-equivariant $\sD$-module on $\cF$, then $\sM_\infty$ is in fact $G$-equivariant. This gives the

\begin{prop}\label{faithful2}
The functor $\sM \rightsquigarrow \sM_\infty$ induces a faithful functor from $\sC^G_{\cF}$ to $G$-equivariant $\sD_{\infty}$-modules on $\frX_\infty$.
\end{prop}

\vskip8pt

\begin{rem} We explain briefly how our equivariant constructions on the flag variety relate to the (nonequivariant) theory of $\wideparen{\mathcal{D}}$-modules on smooth rigid-analytic spaces developed by Ardakov-Wadsley \cite{AWDcapI}. First of all, there is a nonequivariant version $\sC^{G=\{1\}}_\cF$ of the category  $\sC^{G}_\cF$ which can be construced by ignoring the $G$-action in the definition of $\sC^{G}_\cF$. That is to say, by deleting the condition ${\rm (a)}$ and by replacing \ref{isocondition2} of ${\rm (b)}$ by

$$\overline{\psi}_{\frX', \frX}: \sD^\dagger_{\frX,k} \otimes_{\pi_*\sD^\dagger_{\frX',k'}}  \pi_*\sM_{\frX'} \car \sM_\frX$$

\vskip8pt

in \ref{dfn-coadmod}. We then have a functor $\sM \rightsquigarrow \sM_\infty$ from $\sC^{G=\{1\}}_\cF$ to $\sD_{\infty}$-modules as in prop. \ref{faithful2}. Now by the equivalence of categories between abelian
sheaves on $\bbX^{\rig}$ and on $\frX_\infty$ \cite[Prop. 9.3.4]{BoschLectures} we may consider our sheaf of infinite order differential operators $\sD_{\infty}$ to be a sheaf on $\bbX^{\rig}$. One can show that this sheaf coincides with the sheaf $\wideparen{\mathcal{D}}_{\bbX^\rig}$ introduced by Ardakov-Wadsley. Given this identification, the functor $\sM \rightsquigarrow \sM_\infty$ induces then an equivalence between $\sC^{G=\{1\}}_\cF$ and Ardakov-Wadsley's category of coadmissible $\wideparen{\mathcal{D}}_{\bbX^\rig}$-modules.
\end{rem}

\begin{rem}\label{coeff_field} Let $L\subset K$ be a complete and discretely valued extension field such that the topology of $K$ induces the topology on $L$. If we consider the $K$-algebras $\DGO \widehat{\otimes}_L K$ and $D(G,L)\widehat{\otimes}_L K$ as well as the sheaf of $K$-algebras $\sD^\dagger_{\frX,k} \widehat{\otimes}_L K$, then one may establish versions 'over $K$' of the preceding theorems in a straightforward manner. Here, we use the completed topological tensor products for the projective tensor product topology on the ordinary tensor product of two locally convex $L$-vector spaces \cite[ch. IV]{NFA}.
\end{rem}

\section{Examples of localizations}\label{examples}
In this section we compute the $G$-equivariant arithmetic $\sD$-modules corresponding to certain classes of admissible locally analytic $G$-representations. The discussion is a generalization of the $GL(2)$-case treated in \cite{PSS4}. We keep the notation developed in the previous section.
For the rest of this section we fix an element $(\frX,k)\in\underline{\cF}_{\frX_0}$ such that $\frX$ is $\GO$-equivariant.

\vskip8pt
Let $\frg$ denote the Lie algebra of $G$ and
let $L\subset K$ be a complete and discretely valued extension field. To simplify notation, we make the convention that,
when dealing with universal enveloping algebras, distribution algebras, differential operators etc. we write $U(\frg)$, $D(\GO)$, $\sD^\dagger_{\frX,k}$ etc. to denote the corresponding objects {\it after base change to $K$}, i.e., what is precisely $U(\frg_K)$, $D(\GO)\hat{\otimes}_L K$, $\sD^\dagger_{\frX,k}\hat{\otimes}_L K$ and so on (compare also final remark in the preceding section).

\subsection{Smooth representations}\label{smooth_reps}
\setcounter{para}{0}

If $V$ is a smooth $G$-representation (i.e. the stabilizer of each vector $v\in V$ is an open subgroup of $G$), then
$V_{\Gkc-an}$ equals the space of fixed vectors $V^{G_{k+1}}$ in $V$ under the action of the compact subgroup $G_{k+1}$. If $V$ is admissible, then this vector space has finite dimension. In this case one finds, since $\frg V=0$, that

\begin{numequation}\label{smooth_loc}
\Loc^\dagger_{\frX,k} ((V^{G_{k+1}})')= \cO_{\frX,\Q}\otimes_K (V^{G_{k+1}})' \;,
\end{numequation}

\vskip8pt

where $\GO$ acts diagonally and $\sD^\dagger_{\frX,k}$ acts through its natural action on $\cO_{\frX,\Q}$.

\vskip8pt

\subsection{Representations attached to certain \texorpdfstring{$U(\frg)$}{}-modules}\label{reps_att_to_Lie_reps}

In this section, we will compute the arithmetic $\sD$-modules for a class of coadmissible $D(G)$-modules $\bM$ related to the pair $(\frg,B)$ where $B=\bbB(L)$. This includes the case of principal series representations which will be discussed separately in the next section.
Let $\frb$ be the Lie algebra of $B$. Let $\bbT\subset \bbB$ be a maximal split torus, put
$T:=\bbT(L)$ and let $\frt$ be the Lie algebra of $T$.

\vskip8pt

The group $G$ and its subgroup $B$ act via the adjoint representation on $U(\frg)$ and we denote by

\begin{numequation}\label{equ-DF} D(\frg,B):=D(B)\otimes_{U(\frb)} U(\frg)\end{numequation}

the corresponding skew-product ring. The skew-multiplication here is induced by

$$(\delta_{b'} \otimes x')\cdot (\delta_{b}\otimes x)=\delta_{b'b}\otimes \delta_{b^{-1}}(x')x$$

\vskip8pt

for $b,b'\in B$ and $x,x'\in U(\frg)$. A module over $D(\frg,B)$ is the same as a module over $\frg$ together with a compatible locally analytic
$B$-action \cite{OrlikStrauchJH}. Replacing $B$ by $B_0=B\cap \GO$, we obtain a skew-product ring $D(\frg,B_0)$ with similar properties.
Given a ${D(\frg,B)}$-module $M$ one has

\begin{numequation}\label{equ-G0}D(G)\otimes_{D(\frg,B)} M= D(\GO)\otimes_{D(\frg,B_0)} M\end{numequation}

as $D(\GO)$-modules \cite[4.2]{ScSt}. We consider the functor

\begin{numequation}\label{OSfunctor}
M \rightsquigarrow \bM := D(G)\otimes_{D(\frg,B)} M
\end{numequation}

from ${D(\frg,B)}$-modules to $D(G)$-modules \cite{OrlikStrauchJH}. If $M$ is finitely generated as $U(\frg)$-module, then $\bM$ is coadmissible by \cite[4.3]{ScSt}. From now on we assume that $M$ is a finitely generated $U(\frg)$-module. We let $V := \bM'_b$ be the locally analytic $G$-representation corresponding to $\bM$ and denote by

\begin{numequation}\label{anvect} \bM_k:=(V_{\Gkc-\rm an})'\end{numequation}

the dual of the subspace of its $\Gkc$-analytic vectors.
According to \cite[5.2.4]{PSS4} the $\DgkO$-module $\bM_k$ is finitely presented and has its canonical topology.

\begin{lemma}
The canonical map

$$\DgkO\otimes_{D(\GO)} \bM\car \bM_k$$

\vskip8pt

induced by dualising the inclusion $V_{\Gkc-\rm an}\sub V$ is an isomorphism.
\end{lemma}

\begin{proof} This can be proved as in \cite[6.2.4]{PSS4}. \end{proof}

\vskip8pt

Recall the congruence subgroup $G_{k+1}=\Gkc(L)$ of $\GO$. Put $B_{k+1}:=G_{k+1}\cap B_0$. The corresponding skew-product ring $D(\frg,B_{k+1})$ is contained in $\Dgk$ according to \ref{equ-finitefree}. Let $C(k)$ be a (finite) system of representatives in $\GO$ containing $1$ for the residue classes in $\GO/G_{k+1}$ modulo the subgroup $B_0/B_{k+1}$. Note that for an element $g\in \GO$ and a $\Dgk$-submodule $N$ of $D(\GO)$, the abelian group $\delta_gN$ is again a $\Dgk$-submodule because of the formula $x \delta_g = \delta_g {\rm Ad}(g^{-1})(x)$  for any $x\in\Dgk$.

\begin{lemma}
The natural map of $(\Dgk,D(\frg,B_0))$-bimodules

$$ \sum: \bigoplus_{g\in C(k)} \delta_g\Big(\Dgk\otimes_{D(\frg,B_{k+1})} D(\frg,B_0)\Big)\car \DgkO$$

\vskip8pt

is an isomorphism.
\end{lemma}

\begin{proof} This can be proved as in \cite[6.2.5]{PSS4}.
\end{proof}

\vskip8pt

The two lemmas allow us to write

$$ \bM_k = \oplus_{g\in C(k)} \delta_g \Big(\Dgk \otimes_{D(\frg,B_{k+1})} M\Big) = \oplus_{g\in C(k)}\delta_gM^{\rm an}_k$$

\vskip8pt

as modules over $\Dgk$. Here

$$M^{\rm an}_k:= \Dgk \otimes_{D(\frg,B_{k+1})} M \;,$$

\vskip8pt

a finitely presented $\Dgk$-module. If $M$ has character $\theta_0$, so has $M^{\rm an}_k$. As explained above, the 'twisted' module $\delta_gM^{\rm an}_k$ can and will be viewed as having the same underlying group as $M^{\rm an}_k$ but with an action of $\Dgk$ pulled-back by the automorphism ${\rm Ad}(g^{-1})$. Since $\bbG$ is connected, the adjoint action of $G$ fixes the center in $U(\frg)$ and so the character of the module $\delta_g M^{\rm an}_k$ (if existing) does not depend on $g$.

 \vskip8pt

If $M$ has character $\theta_0$, then the $\sD^\dagger_{\frX,k}$-module $\Loc^\dagger_{\frX,k}(\delta_gM^{\rm an}_k)$ on $\frX$ can be described as follows. For any $g\in \GO$ let, as before, $(\rho_g)_*$ denote the direct image functor coming from the automorphism $\rho_g$ of $\frX$. If $N$ denotes a (coherent) $\sD^\dagger_{\frX,k}$-module, then $(\rho_g)_*N$ is a (coherent) $\sD^\dagger_{\frX,k}$-module via
the isomorphism ${\rm Ad}(g): \sD^\dagger_{\frX,k}\car (\rho_g)_*\sD^\dagger_{\frX,k}$, cf. \ref{equ-ringiso0}.

\begin{lemma} One has

$$\Loc^\dagger_{\frX,k}(\delta_gM^{\rm an}_k) = (\rho_g)_* \Loc^\dagger_{\frX,k} (M^{\rm an}_k) = (\rho_g)_* \Big( \sD^\dagger_{\frX,k}\otimes_{D(\frg,B_{k+1})} M\Big) \;.$$

\vskip8pt
\end{lemma}

\begin{proof} This can be proved as in \cite[6.2.6]{PSS4}.
\end{proof}

\vskip8pt

Since $\Loc^\dagger_{\frX,k}$ commutes with direct sums, we may summarize the whole discussion in the general identity

\begin{numequation}\label{formula} \Loc^\dagger_{\frX,k} (\bM_k)= \oplus_{g\in C(k)}  \;(\rho_g)_*  \Big( \sD^\dagger_{\frX,k}\otimes_{D(\frg,B_{k+1})} M\Big)\end{numequation}
of $\sD^\dagger_{\frX,k}$-modules, valid for an arbitrary $D(\frg,B)$-module $M$ (finitely generated over $U(\frg)$) and its coadmissible module $\bM$.

\vskip8pt

\subsection{Principal series representations}\label{princ_series}
We first note the general observation which follows directly from the definition of the algebra $D(\frg,\cdot)$, cf. subsection \ref{reps_att_to_Lie_reps}. If $B'\sub B$ is an open
subgroup and if $\lambda$ denotes a locally analytic character of $B'$, then we have a canonical algebra isomorphism

\begin{numequation}\label{equ-reduction} D(\frg,B')/D(\frg,B')I(\lambda)\simeq U(\frg)/U(\frg)I(d\lambda)\end{numequation}
where $I(\lambda)$ and $I(d\lambda)$ denote the ideals equal to the kernel of $D(B')\stackrel{\lambda}{\lra}K$ and $\frb\stackrel{d\lambda}{\lra}K$ respectively.

\vskip8pt

Now let $\lambda$ be a locally analytic character of $T$ viewed as a character of $B$. We then have the locally analytic principal series representation

$$V:={\rm Ind}_B^G(\lambda^{-1})=\{ f \in C^{\rm la}(G,K): f(gb)=\lambda(b)f(g) {\rm ~for~all~}g\in G, b \in B\}$$

\vskip8pt

with $G$ acting by left translations. Here, $C^{\rm la}(\cdot,K)$ denotes $K$-valued locally analytic functions. We wish to compute the localization $\Loc^\dagger_{\frX,k}$ of the dual of its subspace of $\Gkc$-analytic vectors $V_{\Gkc-\rm an}$ for any sufficienly large $k$. We therefore assume in the following that $k$ is large
enough such that the restriction of $\lambda$ to $T\cap G_{k+1}$ is $\bbT(k)^\circ$-analytic. Let $d\lambda:\frt \ra K$ be the induced character of $\frt$ viewed as a character of $\frb$ and let

$$M(\lambda):=U(\frg)\otimes_{U(\frb)} K_{d\lambda}$$

\vskip8pt

be the induced module. Then $M(\lambda)$ is naturally a $D(\frg,B)$-module and the $D(G)$-module $\bM(\lambda)$ associated with $M(\lambda)$ by the functor \ref{OSfunctor} equals the coadmissible module of the representation $V$ \cite{OrlikStrauchJH}. In particular, $\bM(\lambda)_k=(V_{\Gkc-\rm an})'$ in our notation \ref{anvect}
and therefore

$$\Loc^\dagger_{\frX,k} (\bM(\lambda)_k)=\oplus_{g\in C(k)}  \;(\rho_g)_*  \Big( \sD^\dagger_{\frX,k}\otimes_{D(\frg,B_{k+1})} M(\lambda)\Big)$$

\vskip8pt

by the general formula \ref{formula}. We wish to reinterpret this formula in terms of the classical Beilinson-Bernstein localization of
the $U(\frg)$-module $M(\lambda)$ \cite{BB81}.

\vskip8pt

First of all,

$$M(\lambda)=D(\frg,B_{k+1})/D(\frg,B_{k+1})I_{k+1}(\lambda)$$

\vskip8pt

as a $D(\frg,B_{k+1})$-module
where $I_{k+1}(\lambda)$ denotes the kernel of $D(B_{k+1})\stackrel{\lambda}{\lra} K$, cf. \ref{equ-reduction}. By the choice of $k$ the character $d\lambda$ extends to a character of $\cD^{\rm an}(\bbB(k)^\circ)$ whose kernel is generated by $I(d\lambda)\subset U(\frb)$. It follows

\begin{numequation}\label{reinter} M(\lambda)^{\rm an}_k = \Dgk/\Dgk I_{k+1}(\lambda) = \Dgk\otimes_{U(\frg)} M(\lambda).\end{numequation}

Now the Beilinson-Bernstein localization \cite{BB81} of a finitely generated $U(\frg)$-module $M$ with character $\theta_0$ is a coherent $\sD_{\bbX}$-module ${\rm Loc}(M)$ over the sheaf $\sD_{\bbX}$ of usual algebraic differential operators on the algebraic flag variety $\bbX = \bbB \bksl \bbG$. Let $\bbX^\rig$ be the associated rigid-analytic space with its canonical morphism $\iota: \bbX^\rig \ra \bbX$ of locally ringed spaces.
Let ${\rm sp}_{\frX}: \bbX^\rig\rightarrow\frX$ denote the specialization morphism. Then $({\rm sp}_{\frX})_*\iota^* {\rm Loc}(M)$ is an $\cO_{\frX,\Q}$-module with an action of the sheaf $({\rm sp}_{\frX})_*\iota^*\sD_{\bbX}$. We denote its base change along the natural morphism

$$({\rm sp}_{\frX})_*\iota^*\sD_{\bbX}\lra \sD^\dagger_{\frX,k}$$

\vskip8pt

by

$${\rm Loc}(M)^\dagger_{\frX,k}:=\sD^\dagger_{\frX,k}\otimes ({\rm sp}_{\frX})_*\iota^* {\rm Loc}(M) \;,$$

\vskip8pt

a coherent $\sD^\dagger_{\frX,k}$-module. Suppose now that $\lambda$ is associated by the Harish-Chandra isomorphism to the central character $\theta_0$ and consider $M:=M(\lambda)$. We then have

$${\rm Loc}(M(\lambda))^\dagger_{\frX,k}=\sD^\dagger_{\frX,k}\otimes_{U(\frg)} M(\lambda)=\Loc^\dagger_{\frX,k} (M(\lambda)^{\rm an}_k)$$

\vskip8pt

according to \ref{reinter}. We may thus state

$$\Loc^\dagger_{\frX,k} ((V_{\Gkc-\rm an})')= \oplus_{g\in C(k)}  \;(\rho_g)_*{\rm Loc}(M(\lambda))^\dagger_{\frX,k} \;.$$

\vskip8pt

Let for example $\lambda=-2\rho$ where $\rho$ denotes half the sum over the positive roots (relative to $\bbB$) of $\bbG$.
The sheaf ${\rm Loc}(M(-2\rho))$ is known to be a skyscraper sheaf with support in the origin $\bbB \in \bbX$ \cite[5.1.1]{BK81}.
The fibre $\iota^{-1}(\bbB)$ is a single point in $\bbX^\rig$ and $o:={\rm sp}_{\frX}(\iota^{-1}(\bbB))$ is a closed point in $\frX$. It follows that
${\rm Loc}(M(-2\rho))^\dagger_{\frX,k}$ is a skyscraper sheaf supported at the point $o$. Hence if $V:={\rm Ind}_B^G(2\rho)$ (an irreducible representation by \cite{OrlikStrauchJH}), then the localization $\Loc^\dagger_{\frX,k} ((V_{\Gkc-\rm an})')$ is a sum of copies of this skyscraper sheaf placed at the finitely many points $go \in \frX$ for $g \in C(k)$.

\bibliographystyle{plain}
\bibliography{mybib}

\end{document}